\documentclass[12pt]{amsart}
\usepackage[margin= 0.75 in, top = 1 in, bottom = 1 in]{geometry}
\usepackage{appendix}
\usepackage{soul}
\usepackage[dvipsnames]{xcolor}
\usepackage{tabularx, graphicx, wrapfig, tikz, subcaption,pgfplots}
\usepackage[pagebackref]{hyperref}
\hypersetup{
    colorlinks = true,linkcolor=BrickRed,citecolor=teal,
    linkbordercolor = {white}}

\usepackage{amsmath, amsfonts, amsthm, amssymb, commath}
\pgfplotsset{compat=1.18}

\usepackage[shortlabels]{enumitem}
\usepackage{comment}

\newcommand{\R}{{\mathbb R}}
\newcommand{\N}{{\mathbb N}}
\newcommand{\C}{{\mathbb C}}

\newcommand{\Z}{{\mathbb Z}}
\newcommand{\I}{{\mathbb I}}
\newcommand{\T}{{\mathbb T}}

\newcommand{\Oo}{\mathcal{O}}

\newcommand{\mb}[1]{\mathbf{#1}}
\newcommand{\mc}[1]{\mathcal{#1}}

\newcommand{\re}{{\text{Re}}}
\newcommand{\im}{{\text{Im}}}
\newcommand{\Tr}{{\text{Tr}}}

\DeclareMathOperator*{\res}{Res}

\newmuskip\pFqskip
\mathchardef\pFcomma=\mathcode`, 

\newcommand*\pFq[5]{%
  \begingroup
  \begingroup\lccode`~=`,
    \lowercase{\endgroup\def~}{\pFcomma\mkern\pFqskip}%
  \mathcode`,=\string"8000
  {}_{#1}F_{#2}\biggl[\genfrac..{0pt}{}{#3}{#4} \ ; \ #5\biggr]%
  \endgroup
}

\newcommand{\ii}{{\mathrm i}}
\newcommand{\dd}{{\mathrm d}}
\newcommand{\ee}{{\mathrm e}}

\newcommand{\supp}{\text{supp}}
\renewcommand{\deg}{\text{deg}}
\renewcommand{\arg}{\text{arg}}

\renewcommand{\epsilon}{\varepsilon}
\renewcommand{\subset}{\subseteq}

\newcommand{\qandq}{\quad \text{and} \quad}
\newcommand{\qasq}{\quad \text{as} \quad }

\newtheorem{lemma}{Lemma}[section]
\newtheorem{theorem}{Theorem}

\newtheorem{proposition}[lemma]{Proposition}
\newtheorem{corollary}[lemma]{Corollary}
\newtheorem{conjecture}{Conjecture}
\theoremstyle{definition}
\newtheorem{remark}[lemma]{Remark}
\AtBeginEnvironment{remark}{%
  \pushQED{\qed}%
}
\AtEndEnvironment{remark}{\popQED\endexample}
\newtheorem{rhp}[]{Riemann-Hilbert Problem}
\newtheorem{definition}[lemma]{Definition}
\newtheorem{example}{Example}
\AtBeginEnvironment{example}{%
  \pushQED{\qed}%
}
\AtEndEnvironment{example}{\popQED\endexample}

\numberwithin{equation}{section}

\usepackage[final]{showlabels}

\title{\bf Universal Asymptotics and Exact Enumeration of Eulerian Maps}

\author{Ahmad Barhoumi}
\address{Department of Mathematics, Grinnell College, Grinnell, IA, USA}
\email{barhoumia@grinnell.edu}

\author{Roozbeh Gharakhloo}
\address{Mathematics Department, 
University of California, Santa Cruz, 
CA, USA}
\email{roozbeh@ucsc.edu}

\author{Nathan Hayford}
\address{Department of Mathematics,
KTH Royal Institute of Technology,
Stockholm, Sweden}
\email{nhayford@kth.se}

\date{\today}

\begin{document}
\begin{abstract}
    We calculate the asymptotics of the number of connected, labeled, genus $g$ Eulerian maps with an arbitrary degree sequence, in the limit as the total number of vertices tends to infinity. This asymptotic is universal, in the sense that the leading-order term depends on only three map characteristics, regardless of the choice of the degree sequence.
 The constant factor in this formula is related to the Painlev\'{e} I equation. Our methods combine for the first time the analysis of the recurrence coefficients associated to a particular family of orthogonal polynomials, and the theory of analytic combinatorics of several variables. We also derive an exact formula for the number of connected, labeled, genus $1$ Eulerian maps. These are the first results on this kind of enumeration problem for $g\geq 1$, non-regular (mixed-valence) maps.
    
    \vspace{0.4 cm}


\vspace{.05cm}

\end{abstract}

\maketitle 

\vspace{-0.5 cm}

\section{Introduction and Main Results}
A \emph{map} $G$ is a (multi-)graph which can be embedded without self-intersection on a compact surface and such that the complement of $G$ is a finite union of open cells. We say a map is of genus $g$ if it can be embedded into a surface of genus $g$ and cannot be embedded into any surface of smaller genus. A map is \emph{labeled} if each of the outgoing half-edges is given a natural number label. A map is \emph{regular} if all of its vertices have the same degree. A map is \emph{Eulerian} if all of its vertices have even degree. 

Let $\mathcal{N}_{g}(\mb {n})$ denote the number of connected, labeled, genus $g$ maps with degree sequence $\mb {n}:=(n_1,n_2,n_3,...)$. Since Eulerian maps have vertices of only even degree we will denote the vector 
    \begin{equation}
        \mb {n}=(0,n_2,0,n_4,0,n_6,\cdots) \qquad\text{ by }\qquad (n_2,n_4,n_6,\cdots),
    \end{equation}
unless otherwise specified. The enumeration of (labeled) maps is a classical problem, at least going back to Tutte \cite{MR142470}, where he gave an explicit formula for the number of planar ($g = 0$) Eulerian maps, see equation \eqref{eq:tutte-formula} below. The problem of enumeration of maps found renewed interest in the seminal works of t'Hooft \cite{Hooft}, Br\'ezin, Itzykson, Parisi, and Zuber \cite{MR471676}, David \cite{MR788098}, and Kazakov \cite{kazakov2} in connection to 2D quantum gravity. These works were the first to employ the method of matrix integrals, orthogonal polynomials, and Painlev\'{e} equations to study these kinds of enumerative problems. Furthermore, the combinatorial works of Bender and Canfield \cite{BC}, Bender, Gao, and Richmond \cite{BGR}, and their collaborators shed further light on the connection of such problems and Painlev\'{e} equations. Though some formulas for counting \emph{regular} Eulerian maps of higher genus exist \cite{MR2367197,MR4673885,ELT, GL}, there are no explicit formulas for the case of mixed degree Eulerian maps of genus $g \geq 1$, even asymptotically. The well-known connection between the enumeration of maps, orthogonal polynomials, and random matrix theory provides a starting point for addressing this gap, but is not sufficient to treat the mixed-degree case. In this work, we combine methods arising from this connection with new techniques from the program of Analytic Combinatorics in Several Variables (ACSV) \cite{ACSV,PW} to arrive at our first main result:



\begin{theorem}\label{main-theorem}
    Let $g\geq 0$ and $ p\geq 2$ be fixed integers and $\boldsymbol {\alpha}=(\alpha_2, \alpha_4, ..., \alpha_{2p})$ with $\alpha_{2k}>0$ and $\sum_{k=1}^p\alpha_{2k} = 1$. Define the quantities
    \begin{equation}
        \mathfrak{e} := \sum_{k=1}^{p}k\alpha_{2k},\qquad   \mathfrak{z}:= \sum_{k=1}^{p}k^2\alpha_{2k}, \qquad \mathfrak{L} := \sum_{k=1}^p\alpha_{2k}\log \left( k\binom{2k}{k} \right)
    \end{equation}
    \begin{equation}
        \Omega(\boldsymbol {\alpha}) :=\mathfrak{e}\log\mathfrak{e} -(\mathfrak{e}-1)\log(\mathfrak{e}-1) + \mathfrak{L}.
    \end{equation}
Then, as $V\to \infty$,
    \begin{equation}
        \mathcal{N}_g\left (\lfloor\alpha_{2}V \rfloor, \lfloor\alpha_{4}V \rfloor,\cdots,\lfloor\alpha_{2p}V \rfloor \right) =   \frac{\mathcal{K}_g}{\Gamma(\frac{5g-1}{2})} \cdot V^{\frac{1}{2}(5g-7)} \cdot V! \cdot \ee^{V\Omega(\boldsymbol {\alpha})}[1+\Oo (V^{-1/2})],
    \end{equation}
where $\mathcal{K}_g \equiv \mc K_g(\mathfrak{e}, \mathfrak{z})$ are defined by the recurrence
            \begin{equation}\label{C0-C1-def}
                \mathcal{K}_0 := -\left(\frac{2}{\mathfrak{e}(\mathfrak{e}-1)^5}\right)^{1/2},\qquad \mathcal{K}_1 := \frac{1}{24} \left(\frac{\mathfrak{z}-\mathfrak{e}}{\mathfrak{e}(\mathfrak{e}-1)}\right)^{1/2},
            \end{equation}
        and for $g\geq 2$,
            \begin{equation}
        \mathcal{K}_g := -\frac{1}{2\mathcal{K}_0}\sum_{\ell=1}^{g-1}\mathcal{K}_{\ell}\mathcal{K}_{g-\ell} - (5g-4)(5g-6)\frac{\mathcal{K}_{1}}{\mathcal{K}_0}\mathcal{K}_{g-1}.
        \label{eq:Kg-recurrence}
    \end{equation}
\end{theorem}
A few immediate remarks are in order:
\begin{remark}
    The quantities $\mathfrak{e}$ and $\mathfrak{z}$ are the limiting densities of well-known map characteristics:
        \begin{equation}
            \mathfrak{e} = \lim_{V\to \infty}\frac{E}{V} \ ,\qquad\qquad \mathfrak{z} = \lim_{V\to \infty}\frac{1}{4}\frac{Z}{V} \ ,
        \end{equation}
    where $E$ is the total number of edges (this is immediate from the handshaking lemma), and $Z :=\sum_{v\in V}\deg(v)^2$ is the so-called \textit{Zagreb index} \cite{Gutman} (see also the survey \cite{Gutman2}). The algebraic exponent $\frac{5g-7}{2}$ and the constants $\mathcal{K}_g$ are \textit{universal}, in the sense that the latter only depends on a finite number of map characteristics (namely, the linear functionals $\mathfrak{e},\mathfrak{z}$), and not explicitly on the individual components of $\boldsymbol{\alpha}$.
\end{remark}
\begin{remark}
    Consider the $1$-parameter family of Boutroux tronqu\'{e}e solutions $u(\tau)$ to Painlev\'{e} I:
        \begin{equation}
            u''(\tau) = 6u^2(\tau) + \tau,
        \end{equation}
    which carry the $\tau\to -\infty$ asymptotic expansion
    \begin{equation}
        u(\tau) \sim \sum_{g=0}^{\infty} u_g (-\tau/6)^{\frac{1}{2}(1-5g)},
    \end{equation}
where the constants $u_g$ are determined by the non-linear recurrence $u_0 = 1$, $u_1 = -\frac{1}{1728}$, and for $g\geq 2$
    \begin{equation}
        u_g  = -\frac{1}{2}\sum_{k=1}^{g-1}u_ku_{g-k}-(5g-4)(5g-6)u_1u_{g-1}.
    \end{equation}
See \cite{Boutroux,Kapaev} for the definition and details about these solutions. If we define constants
    \begin{equation}
        A := -(1728\mc K_1\mc K_0^4)^{1/5},\qquad\qquad B:=\left(\frac{\mc K_0^2}{384\mc K_1^2}\right)^{1/5},
    \end{equation}
then the function
    \begin{equation}
        v(\tau) := A\cdot u(B\tau)
    \end{equation}
carries the $\tau\to-\infty$ asymptotic expansion
    \begin{equation}
        v(\tau)\sim \sum_{g=0}^{\infty} \mc K_g (-\tau)^{\frac{1}{2}(1-5g)},
    \end{equation}
where the constants $\mc K_g$ are the same as in \eqref{C0-C1-def}, \eqref{eq:Kg-recurrence}. This can be seen by directly comparing the recurrence relations for the $u_g$'s and $\mc K_g$'s. This kind of connection has been observed earlier for closely related enumeration problems, see the next remark, as well as the discussion of Subsection \ref{subsection:regularmaps}.
\end{remark}
\begin{remark}
     The first few constants $\mathcal{K}_g$ are
     \begin{equation}
        \begin{split}
            \frac{\mathcal{K}_0}{\Gamma(-1/2)} &= \frac{1}{\sqrt{2\pi}} \frac{1}{\mathfrak{e}^{1/2}(\mathfrak{e}-1)^{5/2}},\qquad\qquad \frac{\mathcal{K}_1}{\Gamma(2)} = \frac{1}{24}\left(\frac{\mathfrak{z}-\mathfrak{e}}{\mathfrak{e}(\mathfrak{e}-1)}\right)^{1/2},\\
            \frac{\mathcal{K}_2}{\Gamma(9/2)} &= \frac{7}{1080\sqrt{2\pi}}\frac{(\mathfrak{e}-1)^{3/2}(\mathfrak{z}-\mathfrak{e})}{\mathfrak{e}^{1/2}},\qquad\qquad \frac{\mathcal{K}_3}{\Gamma(7)} =\frac{245}{1990656}\frac{(\mathfrak{e}-1)^{7/2}(\mathfrak{z}-\mathfrak{e})^{3/2}}{\mathfrak{e}^{1/2}}.
        \end{split}
        \label{eq:Kg-constants}
    \end{equation}
    Since $\mathfrak{z} \geq \mathfrak{e}^2>\mathfrak{e}$, we have the inequality
            \begin{equation}
                \mc K_1 \geq \frac{1}{24},
            \end{equation}
        with equality if and only if $\mathfrak{z} = \mathfrak{e}^2$. This only occurs if the class of maps we are considering is \emph{regular}: i.e. $\alpha_{2k} = 0$ for all but one index $k_0$.
\end{remark}
\begin{remark}  Our result extends several known results for regular maps; we will comment further on this later on in the introduction. Let us also observe here that the ratio
    \begin{equation}
        t_{g} := \frac{(2\sqrt{2})^{1-g}}{\Gamma\left(\frac{5g-1}{2}\right)}\cdot\frac{\mathcal{K}_g(\mathfrak{e},\mathfrak{z})\mathfrak{e}^{1/2}}{(\mathfrak{e}-1)^{\frac{1}{2}(4g-5)}(\mathfrak{z}-\mathfrak{e})^{\frac{1}{2}g}}
    \end{equation}
is independent of $\mathfrak{e},\mathfrak{z}$. The constants $t_g$ are the map constants appearing in the works of Bender, Canfield, Gao, and Richmond \cite{BC,Gao,BGR} arising from the recursion of Goulden and Jackson \cite{GouldenJackson}. Their large-$g$ asymptotics were worked out in \cite{BGR} (see also \cite{instant}).
\end{remark}

On the other hand, it is also of interest to find \textit{exact} formulas for $\mathcal{N}_{g}(\mb{n})$ at finite $\mb{n}$. Prior to the present work, such explicit formulas were known only for planar Eulerian maps. Based on purely combinatorial methods, Tutte showed that the number of unlabeled planar Eulerian maps with $n_{2j}$ vertices of degree $2j$ is (\cite{MR142470}, see also \cite{MR1465581}):
    \begin{equation}
        \dfrac{(E - 1)!}{(E - V + 2)!} \prod_{j \geq 1} \binom{2j - 1}{j}^{n_{2j}} \dfrac{1}{n_{2j}!}.
        \label{eq:tutte-formula}
    \end{equation}
    where $V, E$ are the total number of vertices and edges, respectively. Our methods allow us to reproduce this result and extend it to the genus 1 setting. For given integers $r_2, r_4, ..., r_{2p}$, we denote $\mb r := (r_2, r_4, ..., r_{2p})$ and
    \begin{equation}
        V(\mb r) = \sum_{k = 1}^p r_{2k}, \quad E(\mb r) = \sum_{k = 1}^p k r_{2k}, \quad  Z(\mb r) = \sum_{k = 1}^p k^2 r_{2k},
        \label{eq:graph-notation}
    \end{equation}
If $\mb r$ is the degree sequence of an Eulerian map, then $V(\mb r) = V$, the total number of vertices, and $E(\mb r) = E$, the number of edges. 
\begin{theorem}
    For any $p \in \N$ and any degree sequence $\mb n = (n_2, n_4, ...n_{2p})$ such that at least one $n_{2k} \neq 0$ for some $k>1$,
\begin{multline*}
    \mc N_1(n_2, n_4, ..., n_{2p}) =  -\frac{1}{12} \prod_{j=1}^{p}\left(j\binom{2j}{j}\right)^{n_{2j}} \left[(1-E(\mb n))\frac{(E(\mb n)-1)!}{(E(\mb n)-V(\mb n))!} \right. + \frac{V(\mb n)!}{E(\mb n)-V(\mb n)} \\
    \left.\times \left(\sum_{\mb 0 < \mb r < \mb n} \frac{\left(Z(\mb r)-E(\mb r)\right)\left(E(\mb n - \mb r)-V(\mb n - \mb r)\right)}{E(\mb r)\left(E(\mb r) - 1\right)}\cdot\dfrac{\dbinom{E(\mb r)}{V(\mb r)}\dbinom{E(\mb n-\mb r)}{V(\mb n-\mb r)} }{\dbinom{V(\mb n)}{V(\mb r)}}\prod_{j=1}^{p}\binom{n_{2j}}{r_{2j}}\right)\right].
\end{multline*}
\label{thm:g-1}
\end{theorem}
\begin{remark}
    Observe that if only $n_2 \neq 0$, the only possible map is a cycle, which is planar, and so $\mathcal{N}_1(n_2,0,0,...) \equiv 0$ for any $n_2 \geq 1$.
\end{remark}
\begin{remark}
We emphasize the following important details:
    \begin{itemize}
        \item In \cite{MR1465581}, the author enumerates maps \emph{with an oriented root}. The root is a designated edge with an orientation. For any given map, there are $2E$ choices of roots; this explains the extra factor of $2E$ in \cite[Corollary 1]{MR1465581}, in comparison to \eqref{eq:tutte-formula}. 
        \item We also point out that while Tutte counted the number of unlabeled maps, a standard argument relates the number of labeled and unlabeled maps (see, e.g., \cite[Section 1.1.4]{MR3468847}). Precisely, 
        \begin{equation}
            \mc N_0(n_{2}, n_4, ...) = \left(2^{V} \prod_{k \geq 1} k^{n_{2k}} \right) \dfrac{(E - 1)!}{(E - V + 2)!} \prod_{k \geq 1} \binom{2k - 1}{k}^{n_{2k}} 
            \label{eq:count-g-0}
        \end{equation}
        We will reproduce formula \eqref{eq:count-g-0} in Section \ref{sec:tutte-formula} through the associated matrix model and its orthogonal polynomials.
    \end{itemize}
    

\end{remark}
\begin{remark}
    Observe that there are two terms on the right hand side of Theorem \ref{thm:g-1}, the first of which is positive (since $1-E\leq 0$) and the second of which is negative. It is therefore somewhat remarkable that their sum is indeed always positive. It would be interesting to find a combinatorial interpretation of this fact. An alternative formula not involving cancellations is given in Proposition \ref{thm:g-1-prelim} below. It seems that it is nontrivial to directly prove the equivalence of these two formulae.
\end{remark}
Theorem \ref{thm:g-1} is obtained from the explicit calculation of the multivariate generating function of $\mc N_1(\mb n)$, using Lagrange Inversion, in terms of the generating function of the multi-parameter Catalan-Fuss numbers which appeared in \cite{MR114765, MR891548} (see also the unpublished manuscript \cite{MCF}). In the notation of \cite{MR891548}, for any $\mb v \in \Z^p$ satisfying $\mb v - \mb 1 \geq {\bf 0}$, where $\mb 1 = (1, 1, ..., 1) \in \Z^p$ and the inequality is understood component-wise, the multi-parameter Catalan-Fuss number is defined in terms of a multinomial coefficient:
\begin{equation}
    C_{\mb v}(\mb n) := \dfrac{1}{1 + \mb n \cdot \mb v} \binom{1 + \mb n \cdot \mb v}{n_2, n_4, ..., n_{2p}, 1 + \mb n \cdot (\mb v - \mb 1)}.
    \label{eq:catalan-fuss-def}
\end{equation}
\begin{remark}
    What we are calling the multi-parameter Catalan-Fuss numbers have been introduced multiple times and thus go by many names and notations. $C_{\mb v}(\mb n)$ is the notation of \cite{MR891548} which assumes $\mb v$ is a vector of integers whereas \cite{MCF} uses 
    \[
        C_{\mb v}(\mb n) = \mc A_{\mb n} (\mb v, 1) = \prod_{k = 1}^p \dfrac{1}{n_{2k}!} \prod_{j = 1}^{n_2 + n_4 + \cdots + n_{2p} - 1} (\mb n \cdot \mb v + 1 - j)
    \]
   which allows the author to consider complex vectors $\mb v$. In the more graph-theoretic notation, 
    \begin{equation}
        C_{\mb v}(\mb n) = \dfrac{E(\mb n)!}{(E(\mb n) - V(\mb n) + 1)!} \prod_{k = 1}^p \dfrac{1}{n_{2k}!}.
        \label{eq:eq:Chu-number}
    \end{equation}
\end{remark}
Set $\mb v = (1, 2, 3, ..., p)$ and let 
\begin{equation}
    \sigma(\mb t) \equiv \sigma(t_2, t_4, ..., t_{2p}) := \sum_{\mb n \geq 0} C_{\mb v}(\mb n) \prod_{k = 1}^p \left( - \binom{2k-1}{k} t_{2k}\right)^{n_{2k}}.
    \label{eq:sigma-def}
\end{equation}
$\sigma(\mb t)$ is algebraic; it is known (cf. \cite[Eq. (6.3)]{MR114765}), and we will re-prove, that
\begin{equation}
     1 = \sigma(\mb t) + \sum_{k = 1}^p \binom{2k-1}{k} t_{2k} \sigma^k(\mb t).
     \label{eq:sigma-poly}
\end{equation}
Denote the generating function of $\mc N_g(\mb n)$ by
\begin{equation}
    F_{g}(\mb t) = \sum_{\mb n \geq 0} \mc N_g (\mb n) \prod_{k = 1}^p \dfrac{(-t_{2k}/(2k))^{n_{2k}}}{n_{2k}!}.
    \label{eq:Fg}
\end{equation}
It was shown in \cite[Section 7]{MR2187941} that $F_0(\mb t)$ can be explicitly expressed in terms of $\sigma(\mb t)$. Other expressions for $F_0(\mb t), F_1(\mb t)$ have appeared in \cite[Eq. (5.18)]{ChekhovMakeenko}, \cite[Eqs. (5.184)]{MR3389055}, and \cite[Eq. (6)]{MR3398916}. In terms of $\sigma(\mb t)$, the results are the following.  
\begin{proposition}
    With the notation of \eqref{eq:Fg} and $\sigma(\mb t)$ as in \eqref{eq:sigma-def}, we have 
    \begin{multline}
        F_0(\mb t) = \frac{3}{4}+\frac{1}{2} \log \sigma(\mb t)-\sigma(\mb t)+\frac{1}{4} \sigma^2(\mb t)-\sum_{\ell=1}^p\binom{2 \ell-1}{\ell} t_{2 \ell}\left(\frac{\sigma^{\ell}(\mb t)}{\ell}-\frac{\sigma^{\ell+1}(\mb t)}{\ell+1}\right)\\
        +\sum_{k, \ell=1}^p\binom{2 k-1}{k}\binom{2 \ell-1}{\ell} t_{2 k} t_{2 \ell} \frac{\sigma^{k+\ell}(\mb t)}{k+\ell}.
        \label{eq:free-energy-0}
    \end{multline}
    and 
    \begin{equation}
        F_1(\mb t) = -\dfrac1{12} \log \left( \sigma(\mb t) + \sum_{k = 1}^p \binom{2k-1}{k} kt_{2k} \sigma^k(\mb t) \right).
        \label{eq:free-energy-1}
    \end{equation}
    \label{thm:free-energy-0-1}
\end{proposition}
Many of the results presented here have already appeared in the context of regular maps. We briefly survey these next. 

\subsection{Enumeration of regular maps}\label{subsection:regularmaps}

While there are no general analogs of Tutte's formula when $g \geq 2$, the important case of regular maps has been thoroughly explored. Indeed, the case of regular Eulerian maps, where $n_{2\nu} \neq 0$ and $n_{2j} = 0$ for all $j \neq \nu$, was solved by Ercolani, Lega, and Tippings in \cite{ELT}, where it was shown that 
\begin{multline}
    \mc N_1(n_{2\nu} = j) = \dfrac{j! (2\nu)^j}{12}  \binom{2\nu - 1}{\nu - 1}^j\left( (\nu - 1) \binom{\nu j - 1}{j - 1} \pFq{3}{2}{1,1,1-j}{2,(\nu - 1)j + 1}{1 - \nu} \right.\\
    \left.- (\nu- 1)^2 \binom{\nu j - 1}{j - 2} \pFq{3}{2}{1, 1, 2 - j}{2, (\nu - 1)j + 2}{1 - \nu} \right), \quad j \geq 2,
    \label{eq:g-1-pure}
\end{multline}
and the same result holds for $j = 1$ if one accepts the convention $\binom{a}{b} = 0$ whenever $b < 0$. In the same work, the authors show that 
\begin{multline}
    \mc N_g (n_{2\nu} = j) = j! (2\nu)^j \binom{2\nu - 1}{\nu - 1}^j (\nu - 1)^j \\
    \times \sum_{\ell = 0 }^{3g - 3} \left( b_{\ell}^{(g, \nu)} \binom{(2g - 4) + (\ell + j)}{j} \pFq{2}{1}{-j, 1 - \nu j}{4 - 2g - (\ell + j)}{(1 - \nu)^{-1}} \right)
\end{multline}
for all $g \geq 1$ and $j \geq 1$. This is not exactly an explicit formula since one needs to determine the unknown coefficients $b_{\ell}^{(g, \nu)}$. Explicit formulas in terms of $\nu, g$ were obtained for small genera and numbers of vertices in the recent manuscript \cite{GL}. Around the same time, the authors of \cite{MR4673885} considered $\nu = 2$ and were able to reproduce formulas found in \cite{MR603127} for $\mc N_g(0, n_4, 0, ...)$ when $g = 0, 1, 2$, and obtained an explicit count for $g = 3$. Furthermore, it was shown in \cite{MR4673885} that\footnote{Observe that the constant denoted by $\mc K_g$ in \cite{MR603127} differs from ours by a factor $\Gamma(\frac12(5 g - 1))$ \label{ftnote:Kg}.} 
\begin{equation}
    \mc N_g(0, j, 0, ...) = \frac{\mc K_g}{\Gamma(\frac{5g-1}{2})} 48^j j^{\frac{5g - 7}{2}} j! \left( 1 + \Oo\left(j^{-\frac12} \right)\right)
    \label{eq:count-asymp-quart}
\end{equation}
where $\mc K_g$ satisfies the recurrence \eqref{eq:Kg-recurrence}. This result is in fact a special case of Theorem \ref{main-theorem}, which one can obtain by setting $\alpha_{2k} = 0$ unless $k=2$, and where $\mathfrak{e} = 2$, $\mathfrak{z} = 4$ (specifically, compare Theorem 6.8 and Equation 6.161 in \cite{MR4673885} with Equation \eqref{eq:Kg-constants} above).

More generally, for regular maps of arbitrary even valence, Ercolani and Waters showed that \cite[Eq. (A.9)]{MR4480832} \footnote{There is a minor typo there: what is written as $\sqrt{2\nu} (\nu -1)$ should instead read $\sqrt{ 2\nu/(\nu-1)}$, which is corrected in \cite[Section 4.2]{ELT}.}
    \begin{equation}
        \mc N_g(n_{2\nu} = j) = \frac{\mc K_g}{\Gamma(\frac{5g-1}{2})} j^{\frac{5g-7}{2}} \left(\nu\binom{2\nu}{\nu}\frac{\nu^{\nu}}{(\nu-1)^{\nu-1}}\right)^j j! \left( 1 + \Oo\left(j^{-\frac12} \right)\right),
        \label{eq:count-asymp-2p-reg}
    \end{equation}
where the constants $\mc K_g$ again satisfy a non-linear recurrence related to Painlev\'e-I; by taking $\alpha_{2k} = 0$ unless $k=2\nu$ in Theorem \ref{main-theorem}, we can recover this result as well (here, $\mathfrak{e} = \nu$, $\mathfrak{z} = \nu^2$).\\

For regular maps with \emph{odd} valence, the picture is less complete. In \cite{MR471676}, a matrix model for 3-regular maps was introduced. This requires a deformation of the matrix model \eqref{eq:matrix-model} into the complex plane since \eqref{eq:partition-fn-def} diverges when $\deg V(z; \mb t) = 3$. A rigorous study of this model appeared in the work of Bleher and Dea\~{n}o \cite{MR3071662}. Specifically, denote the number of labeled maps with degree sequence $n_1, n_2, ...$ by $\mc N_g(n_1, n_2, ...)$. The authors showed that\footnote{a parity argument implies $\mc N_0(0, 0, 2j+1, 0, ...) = 0$ for all $j$.}
\begin{align*}
    \mc N_0(0, 0, 2j, 0, ....) &= \frac{72^j \Gamma\left(\frac{3 j}{2}\right)(2 j)!}{2 \Gamma(j+3) \Gamma\left(\frac{j}{2}+1\right)}, \\
    \mc N_1(0, 0, 2j, 0, ....) &= \dfrac{5\cdot 72^j \Gamma(\frac{2j}{3}) (2j)!}{48 (3j + 2)\Gamma(j+1)\Gamma(\frac12 j + 1)} \pFq{3}{2}{-j+1, 2, 6}{5, -\frac32 j + 1}{\frac32},
\end{align*}
and, for $g \geq 2$, deduced the analog of \eqref{eq:count-asymp-quart}, \eqref{eq:count-asymp-2p-reg} (see Footnote \ref{ftnote:Kg}):
\begin{equation}
    \mc N_g(0, 0, 2j, 0, ....) = \frac{\mathcal{K}_{g}}{\Gamma(\frac{5g-1}{2})} \left(\dfrac{3^{\frac{1}{4}}}{18}\right)^{2j} (2j)^{\frac{5g-7}{2}}(2j)! \left( 1 + \Oo \left(j^{-\frac12} \right)\right),
    \label{eq:odd-expansion}
\end{equation}
where again a Painlev\'{e} I-type recursion was found for the constants $\mathcal{K}_{g}$.
\begin{remark}
    Although the above result is clearly not a special case of Theorem \ref{main-theorem}, it is interesting to note that the exponent $\frac{5g-7}{2}$ remains the same here. Furthermore, one can extend the definitions of $\mathfrak{e}$ and $\mathfrak{z}$ to \textit{all} maps in the obvious way. For $3$-regular maps, 
        \begin{equation}\label{eq:3-reg-characteristics}
            \mathfrak{e} = \frac{3}{2}, \qquad \text{ and }\qquad \mathfrak{z} = \mathfrak{e}^2 = \frac{9}{4}.
        \end{equation}
    Remarkably, one can show that the constants $\mathcal{K}_g(\mathfrak{e},\mathfrak{z})$, taken with the choice \eqref{eq:3-reg-characteristics}, agrees with the constants found in \cite{MR3071662}, suggesting that the above formula is \textit{universal} beyond the class of Eulerian maps.
\end{remark}
The leading asymptotics in \eqref{eq:odd-expansion} turn out to be quite general; it was shown in \cite{MR4480832} that there exists a real number $t_c$ such that, for any fixed $m \in \N$: 
\[
    \mc N_{g}(n_{2m+1} = 2j) = C_g t_c^{j} (2j)^{\frac{5g - 7}{2}} (2j)! \qasq j \to \infty,
\]
where $C_g$ satisfy a recurrence which again mirrors the form of the recurrence which calculates the coefficients in an asymptotic expansion of the tritronqu\'ee solution of the Painlev\'e I transcendent. In view of Theorem \ref{main-theorem} and the above observations, we conjecture:

    \begin{conjecture}
        Let $g\geq 0$ and $p\geq 3$ be fixed integers, and $\boldsymbol {\alpha}=(\alpha_1, \alpha_2, ..., \alpha_{p})$ with $\alpha_{k}>0$ and $\sum_{k=1}^p\alpha_{k} = 1$. Define the quantities
    \begin{equation}
        \mathfrak{e} := \frac{1}{2}\sum_{k=1}^{p}k\alpha_{k},\qquad   \mathfrak{z} := \frac{1}{4}\sum_{k=1}^{p}k^2\alpha_{k}.
    \end{equation}
    Then, there exists a constant $\Omega = \Omega(\boldsymbol {\alpha})$ such that, as $V\to \infty$,
    \begin{equation}
        \mathcal{N}_g\left (\lfloor\alpha_{1}V \rfloor, \lfloor\alpha_{2}V \rfloor,\cdots,\lfloor\alpha_{p}V \rfloor \right) =  \frac{\mathcal{K}_g}{\Gamma(\frac{5g-1}{2})} \cdot V^{\frac{1}{2}(5g-7)} \cdot V! \cdot \ee^{V\Omega(\boldsymbol {\alpha})}[1+\Oo (V^{-1/2})],
    \end{equation}
where $\mathcal{K}_g \equiv \mc K_g(\mathfrak{e}, \mathfrak{z})$ are defined by the recurrence
            \begin{equation}
                \mathcal{K}_0 := -\left(\frac{2}{\mathfrak{e}(\mathfrak{e}-1)^5}\right)^{1/2},\qquad \mathcal{K}_1 := \frac{1}{24} \left(\frac{\mathfrak{z}-\mathfrak{e}}{\mathfrak{e}(\mathfrak{e}-1)}\right)^{1/2},
            \end{equation}
        and for $g\geq 2$,
            \begin{equation}
        \mathcal{K}_g = -\frac{1}{2\mathcal{K}_0}\sum_{\ell=1}^{g-1}\mathcal{K}_{\ell}\mathcal{K}_{g-\ell} - (5g-4)(5g-6)\frac{\mathcal{K}_{1}}{\mathcal{K}_0}\mathcal{K}_{g-1}.
    \end{equation}
\end{conjecture}

\subsection{Connection with Random Matrix Theory}

Let $\mc H_n$ be the set of $n \times n$ Hermitian matrices endowed with the probability measure 
\begin{equation}
    \dfrac{1}{ \mc Z_{n,N}(\mb t)} \ee^{-N \Tr V(\mb M; \mb t)}\dd \mb M,
    \label{eq:matrix-model}
\end{equation}
where $\dd \mb M = \prod_{j > k} \dd (\re (M_{jk}))\dd( \im (M_{jk}) )\prod_{j = 1}^n \dd M_{jj}$, and $V(z; \mb t)$ is the \emph{potential} 
\begin{equation}
    V(z; \mb t) = \dfrac12 z^2 + \sum_{k = 1}^p \dfrac{t_{2k}}{2k} z^{2k}. 
    \label{eq:potential}
\end{equation}
The function $\mc Z_{n, N}(\mb t)$ is the \emph{partition function}, given by 
\begin{equation}
    \mc Z_{n, N} (\mb t) = \int_{\mc H_n} \ee^{-N \Tr V(\mb M; \mb t)}\dd \mb M
    \label{eq:partition-fn-def}
\end{equation}
and plays a central role in the map enumeration problem. In the physics literature, this connection was observed and developed by t'Hooft \cite{Hooft, Hooft2}, Kazakov \cite{kazakov2}, David \cite{MR788098}, Witten \cite{MR1144529}, Kontsevich \cite{MR1171758}, Bessis, Itzykson, Parisi, and Zuber \cite{MR603127, MR471676, MR548746}. A review of these developments can be found in \cite{MR1492512, MR2231334,PGZ}. The key observation, which is central to our analysis, is that the \emph{free energy}
\begin{equation}
    \mc F_{n, N}(\mb t) := \dfrac{1}{N^2} \log \dfrac{\mc Z_{n, N}(\mb t)}{\mc Z_{n, N}(\mb 0)}
    \label{eq:free-energy-def}
\end{equation}
admits the following asymptotic expansion. 

\begin{theorem} \emph{(\cite[Theorem 1.1]{MR1953782})}.
For any given $T > 0$, $\gamma > 0$, let 
\[
    \mathbb{T}(T, \gamma) := \left\{ \mb t \in \R^p \ :  \ |\mb t| < T, \qandq t_{2p} > \gamma \sum_{j = 1}^{p - 1} |t_{2j}| \right\}.
\]
There exist $T > 0, \gamma > 0$ and $N_0 > 0$ such that for all $\mb t \in \T(T, \gamma)$ and all $N > N_0$, the following asymptotic expansion holds:
    \begin{equation}
    \mc F_{N, N}(\mb t) \sim \sum_{g = 0}^\infty \dfrac{F_{g}(\mb t)}{N^{2g}} \qasq  N \to \infty,
    \label{eq:top-exp}
\end{equation}
where $F_{g}(\mb t)$ is the multivariate generating function for connected Eulerian maps of genus $g$ defined in \eqref{eq:Fg}. Expansion \eqref{eq:top-exp} is known as the {\bf topological expansion}. 
\label{thm:top-exp}
\end{theorem}
For the topological expansion of the free energy when $n/N \neq 1$, see \cite{MR2187941,MR2367197}. The proof of Theorem \ref{thm:top-exp} exploits the connection with orthogonal polynomials, which we explain below. 

\subsection{Orthogonal Polynomials} \label{sec:OP}Instead of the entries of the matrices, one can recast the integral \eqref{eq:partition-fn-def} in terms of the eigenvalues of $\mb M$. The Jacobian of this transformation is explicit (cf. \cite{Mehta-Book}, for instance) and the integral now reads
\begin{equation}
    \mc Z_{n, N}(\mb t) = C_{n}\int_{\R} \cdots \int_{\mb \R} \prod_{1 \leq j < k \leq n} (z_j - z_k)^2 \prod_{k = 1}^n \ee^{-N V(z_k; \mb t)} \dd z_k 
    \label{eq:partition-mult-int}
\end{equation}
where $C_n$ is a ${\mb t}$-independent quantity that does not play a role in \eqref{eq:free-energy-def}. The right hand side of \eqref{eq:partition-mult-int} is recognizable as one side of the Heine formula for Hankel determinants \cite[Chapter 2]{MR372517}. Precisely, let 
\begin{equation}
    \mu_{j, N}(\mb t) = \int_\R z^j \ee^{-NV(z; \mb t)} \dd z
    \label{eq:moments}
\end{equation}
and denote the $n \times n$ Hankel determinant of moments by $D_{n, N}(\mb t) := \det [ \mu_{j + k, N}]_{i, j = 0}^{n - 1}$. Then, the Heine formula implies 
\begin{equation}
    \mc Z_{n, N} (\mb t) = C_n n! D_{n, N}(\mb t).
    \label{eq:partition-hankel-det}
\end{equation}
Equation \eqref{eq:partition-hankel-det} allows one to deduce that the partition function is an example of a so-called \emph{isomonodromic tau function}. This has long been known in the orthogonal polynomials and integrable systems literature \cite{MR1986408, MR2207650}; we will recall the connection to the isomonodromic tau function in Appendix \ref{appendix:JMU-free-energy}. It is also known that \eqref{eq:partition-hankel-det} satisfies hierarchies of integrable differential equations; this was leveraged in \cite[Section 4]{MR2367197} to extend the validity of the expansion \eqref{eq:top-exp} for $n, N \to \infty$ with $n/N \to \varkappa$ in a neighborhood of 1. More can be gained by delving deeper into the theory of orthogonal polynomials. Indeed, let $P_n(z; \mb t)$ be the monic polynomials satisfying 
\begin{equation}
    \int_\R P_n(z; \mb t, N) P_m(z; \mb t, N) \ee^{-N V(z; \mb t) } \dd z = h_{n, N}(\mb t) \delta_{nm}, \quad h_{n, N}(\mb t) := \dfrac{D_{n+1, N}(\mb t)}{D_{n, N}(\mb t)}.
    \label{eq:ortho}
\end{equation}
It is a classical fact (cf. \cite{MR372517}) that $P_n(z; \mb t, N)$ exists and is unique whenever $\mb t \in \R_+^{p}$ and is given explicitly by 
\begin{equation}
    P_n(z; \mb t, N) = \dfrac{1}{D_{n, N}(\mb t)} \det \begin{bmatrix}  \mu_{0, N}(\mb t) &  \mu_{1, N}(\mb t) & \cdots &  \mu_{n - 1, N}(\mb t) &  \mu_{n , N}(\mb t) \\
    \mu_{1, N}(\mb t) &  \mu_{2, N}(\mb t) & \cdots &  \mu_{n , N}(\mb t) &  \mu_{n +1 , N}(\mb t) \\
    \vdots & \vdots & \ddots & \vdots & \vdots \\
    \mu_{n - 1, N}(\mb t) &  \mu_{n , N}(\mb t) & \cdots &  \mu_{2n - 2, N}(\mb t) &  \mu_{2n - 1, N}(\mb t) \\
    1 & z & \cdots & z^{n - 1} & z^n
    \end{bmatrix}.
    \label{eq:p-det-form}
\end{equation}
Furthermore, the polynomials $P_n(z; \mb t, N)$ satisfy a three-term recurrence relation which, for even potentials such as \eqref{eq:potential}, takes the form 
\begin{equation}
    zP_n(z; \mb t, N) = P_{n + 1}(z; \mb t, N) + R_{n, N}(\mb t) P_{n - 1} (z; \mb t, N).
    \label{eq:3-term}
\end{equation}
It follows from \eqref{eq:moments} that the moments are meromorphic functions in the the entries of $\mb t$ on any domain where the right hand side of \eqref{eq:moments} converges. This, the definition of $D_{n, N}(\mb t)$, and \eqref{eq:p-det-form} implies that the coefficients of $P_n(z; \mb t, N)$ are also meromorphic functions in $\mb t$. 

The key observation, going back to \cite{MR603127}, is that one may express the free energy explicitly in term of the recurrence coefficients $R_{n, N}(\mb t)$. The authors of \cite{MR603127} then used this and asymptotic properties of $R_{n, N}(\mb t)$ (cf. Proposition \ref{prop:recurrence-expansion}) to deduce the topological expansion \eqref{eq:top-exp}. A different approach, using differential-difference identities connecting $\mc F_{n, N}(\mb t)$ and $R_{n, N}(\mb t)$ was used in, e.g., \cite{MR2187941, MR4673885, MR3071662}, to deduce the topological expansion \eqref{eq:top-exp} for a larger set of parameters $\mb t$. We will use an analog of these differential-difference identities (see Lemma \ref{Lemma:Bleher-Its} below) to relate the Taylor coefficients of $F_g(\mb t)$ to those of $R_{n, N}(\mb t)$.

\subsection{Outline}
The remainder of this work is organized as follows. In Section \ref{sec:FreeEnergy}, we perform the Riemann-Hilbert analysis necessary for our main results. In Section \ref{sec:sigma-section}, we establish important properties of the function $\sigma({\bf t})$, see Equation \eqref{eq:sigma-def}. In Section \ref{sec:tutte-section}, we prove Theorem \ref{thm:g-1}, as well as recover of Tutte's Formula \eqref{eq:count-g-0}. Finally, in Section \ref{sec:main-thm-pf}, we prove Theorem \ref{main-theorem}.

\subsection{Acknowledgments}
NH was supported by the European Research Council (ERC), Grant Agreement No. 101002013. NH thanks the LPSM at Sorbonne University for its
hospitality, where part of this work was completed.

\section{The Free Energy as an Isomonodromic Tau Function}\label{sec:FreeEnergy}

Our starting point will be the Riemann-Hilbert Problem (RHP) for orthogonal polynomials. In what follows, it will be convenient to set
\[
\I :=\begin{bmatrix} 1 & 0 \\ 0 & 1 \end{bmatrix}, \qquad  \sigma_1:=\begin{bmatrix} 0 & 1 \\ 1 & 0 \end{bmatrix}, \qquad  \sigma_2:=\begin{bmatrix} 0 & -\ii \\ \ii & 0 \end{bmatrix}, \qquad \sigma_3:=\begin{bmatrix} 1 & 0 \\ 0 & -1 \end{bmatrix}
\]
and, for any given function $f(z)$, 
\[
    f^{\sigma_3} := \begin{bmatrix} f(z) & 0 \\ 0 & 1/f(z)\end{bmatrix}.
\]
\begin{rhp} \label{rhp:y}
    Seek a $2 \times 2$ matrix function $\mb Y(z; \mb t)$ satisfying 
    \begin{itemize}
        \item[] {\bf Analyticity:} $\mb Y(z; \mb t)$ is analytic (in $z$) in $\C \setminus \R$, 

        \item[] {\bf Jump condition:} $\mb Y(z; \mb t)$ has continuous boundary values on $\R$ satisfying\footnote{\label{ftnote:real-orient}We will always orient the real line from left to right.}
        \[
            \mb Y_+ (z; \mb t) = \mb Y_-(z; \mb t) \begin{bmatrix}
                1 & \ee^{-NV(z; \mb t)} \\ 0 & 1
            \end{bmatrix}, \quad z \in \R,
        \]

        \item[] {\bf Normalization:} $\mb Y(z; \mb t)$ satisfies the asymptotic condition
        \begin{equation}
            \mb Y(z; \mb t) = \left( \I + \frac{\mb Y^{(1)}(\mb t)}{z}  + \Oo\left(\frac1{z^2} \right)  \right) z^{n\sigma_3} \qasq z \to \infty. 
            \label{eq:Y-infty-expansion}
        \end{equation}
    \end{itemize}
\end{rhp}

The connection of RHP \ref{rhp:y} to orthogonal polynomials was first demonstrated by Fokas, Its, and Kitaev in \cite{MR1101341} and lies in the following. The solution of RHP \ref{rhp:y} exists if and only if polynomials $P_n(z) = P_n(z;\mb t, N)$ satisfying the orthogonality relations \eqref{eq:ortho} exist and satisfy 
\begin{equation}
    \deg P_n(z) = n, \qandq \left(\mathcal{C}P_{n-1} \ee^{-NV}\right)(z) = z^{-n} (1 + \Oo(z^{-1})),
    \label{eq:rhp-cond}
\end{equation}
where $\mathcal{C}f(z)$ is the Cauchy transform of a function $f$ given on $\R$:
\[
    (\mathcal{C}f)(z) := \frac1{2\pi\ii}\int_\R\frac{f(s)}{s-z}\dd s.
\]
Conditions \eqref{eq:rhp-cond} hold for all $n \in \N$ with $\mb V(z; \mb t)$ as in \eqref{eq:potential} and $\mb t\in \T(T, \gamma)$ as in Theorem \ref{thm:top-exp}. Whenever RHP \ref{rhp:y} is solvable, its solution is unique, invertible, and given by 
\begin{equation}
    \label{eq:y}
    \mb Y(z) = \begin{bmatrix}
    P_n(z) & \big(\mathcal{C}P_n \ee^{-NV}\big)(z) \medskip \\
    -\frac{2\pi\ii}{h_{n-1}}P_{n-1}(z) & -\frac{2\pi\ii}{h_{n-1}}\big(\mathcal{C}P_{n-1} \ee^{-NV}\big)(z)
\end{bmatrix},
\end{equation}
where $h_n=h_{n, N}(\mb t)$ are the normalizing constants on the right hand side of \eqref{eq:ortho}. 

One can view $\mb Y(z;\mb t )$ as the solution of to a first order $2 \times 2$ system of differential equations in $z$ with parameters $\mb t$. Indeed, consider the matrix 
\[
    \mb \Phi(z;\mb t) := \mb Y(z;\mb t) \ee^{-NV(z; \mb t) \sigma_3/2}.
\]
Then, it follows from its definition and RHP \ref{rhp:y} that $\mb \Phi(z;\mb t)$ uniquely solves the following RHP. 

\begin{rhp} \label{rhp:phi}
    Seek a $2 \times 2$ matrix function $\mb \Phi(z;\mb t)$ satisfying 
    \begin{itemize}
        \item[] {\bf Analyticity:} $\mb \Phi(z;\mb t)$ is analytic in $\C \setminus \R$, 

        \item[] {\bf Jump condition:} $\mb \Phi(z;\mb t)$ has continuous boundary values on $\R$ satisfying 
        \begin{equation}
          \mb \Phi_+ (z;\mb t) = \mb \Phi_-(z;\mb t) \begin{bmatrix}
                1 & 1 \\ 0 & 1
            \end{bmatrix}, \quad z \in \R, 
            \label{eq:Phi-jump}
        \end{equation}

        \item[] {\bf Normalization:} $\mb \Phi(z;\mb t)$ satisfies the asymptotic condition
        \begin{equation}
            \mb \Phi(z;\mb t) = \left( \I + \Oo\left(\frac1z \right)  \right) z^{n\sigma_3} \ee^{-NV(z; \mb t) \sigma_3/2} \qasq z \to \infty. 
            \label{eq:phi-expansion}
        \end{equation}
    \end{itemize}
\end{rhp}

A standard argument using the analyticity of properties of $\mb \Phi(z;\mb t)$ (see, e.g., \cite[Chapter 22]{MR2191786}), implies that $\mb \Phi(z;\mb t)$ satisfies the differential equation 
\[  
    \dod{\mb \Phi}{z} = \mb A(z; \mb t) \mb \Phi(z; \mb t),
\]
where $\mb A(z; \mb t)$ is a $2\times 2$ matrix polynomial in $z$ of degree $\deg V - 1 = 2p - 1$. Viewing $\mb t$ as a vector of deformation parameters, Jimbo, Miwa, and Ueno developed a theory of monodromy-preserving (isomonodromic) deformations of ${\mb \Phi}(z; \mb t)$ \cite{MR630674, MR625446, MR636469}. In their work, they introduce what is now know as as the Jimbo-Miwa-Ueno (JMU) differential: 
\begin{equation}
    \omega_{n, N}(\mb t) := -\dfrac{N}{2} \res_{z = \infty} \Tr \left( \mb Y^{-1}(z; \mb t) \dod{\mb Y}{z}(z; \mb t) {\bf d} V(z; \mb t) \sigma_3 \right), \quad \text{where } \quad {\bf d}(\diamond) := \sum_{j = 1}^{p} \dpd{}{t_{2j}}(\diamond) \ \dd t_{j}.
    \label{eq:jmu-def}
\end{equation}
It was shown in \cite[Theorem 3]{MR630674} that $\omega_{n, N}(\mb t)$ is closed. In fact, one can explicitly compute $\omega_{n, N}(\mb t)$ in terms of the free energy \eqref{eq:free-energy-def}; a proof of the following is reproduced in Appendix \ref{appendix:JMU-free-energy} for the convenience of the reader. 
\begin{proposition}[\cite{MR1986408}]
    Let $\omega_{n, N}(\mb t)$ be the JMU differential defined in \eqref{eq:jmu-def}. Then, 
    \begin{equation}
    \omega_{n, N}(\mb t) = N^2 {\bf d} \mc F_{n, N} (\mb t). 
    \end{equation}
    \label{prop:JMU-free-energy}
\end{proposition}

To arrive at Proposition \ref{thm:free-energy-0-1}, we seek an asymptotic expansion of $\omega_{n, N}(\mb t)$. This can be done by first obtaining a corresponding asymptotic expansion of $\mb Y(z; \mb t)$. The latter was done in \cite{MR1702716} using the Deift-Zhou non-linear steepest descent method developed in \cite{MR1207209}. We now introduce the functions appearing in the expansion. 

\subsection{The Equilibrium Measure} For definitiveness, fix $\gamma, T >0$ as in Theorem \ref{thm:top-exp} and $\mb t \in \mathbb{T}(T, \gamma)$. The \emph{equilibrium measure in the external field $V$}, denoted $\mu_{V}$, is a probability measure satisfying 
\begin{equation}
    I_V(\mu_V) = \inf_{\mu} I_V(\mu),
\end{equation}
where 
\[
    I_V(\mu) := \iint \log|s - t|^{-1} \dd \mu(s) \dd \mu(t) + \int V(t) \dd \mu(t).
\]
This infimum is achieved and the minimizer is unique and compactly supported \cite[Chapter I]{MR1485778}. Furthermore, in the particular case of polynomial external fields, the support of the measure $\mu_V$ is a finite union of intervals \cite{MR1385683}. The appearance of multiple intervals, or ``cuts," complicates the analysis. Thankfully, when $\mb t \in \mathbb{T}(T, \gamma)$, it was observed in \cite[Theorem 3.1]{MR1953782}, based on results from \cite{MR1657691, MR1744002} that the support of $\mu_{V}$ is a single interval and
\begin{equation}
    \dd \mu_V(z) = \dfrac{\boldsymbol{1}_{\scriptscriptstyle \left(\alpha(\mb t), \beta(\mb t)\right)}}{2 \pi }  h(z; \mb t) \sqrt{ (z - \alpha(\mb t))(\beta(\mb t) - z)} \ \dd z
    \label{eq:equilibrium-measure}
\end{equation}
where $h(z; \mb t)$ is a polynomial of degree $2p - 2$ and the endpoints $\alpha(\mb t), \beta(\mb t)$ are analytic functions of $\mb t$. In the following proposition, whose proof is deferred to the end of this section, we specialize \eqref{eq:equilibrium-measure} to our setting.
\begin{proposition} \label{prop:eq-measure}
    Fix $\mb t\in \mathbb{T} (\gamma, T)$ and let $V(z; \mb t)$ be as in \eqref{eq:potential}. Then, the equilibrium measure \eqref{eq:equilibrium-measure} is given by 
    \[
        \dd \mu_V(z) = \dfrac{1}{2 \pi \ii }  h(z; \mb t) w_+(z; \mb t) \ \dd z, \qquad \supp(\mu_V) = \left(-2\sigma^\frac12(\mb t), 2\sigma^\frac12(\mb t)\right),
    \]
    with
    \begin{equation}
        w(z; \mb t) := (z^2 - 4\sigma(\mb t))^{\frac12}, \qquad z \in \C \setminus [-\sigma(\mb t), \sigma (\mb t)]
    \end{equation}
    is the branch satisfying $w(z; \mb t) \sim z + \Oo(z^{-1})$ as $z \to \infty$, 
        \begin{equation}
            h(z; \mb t) = \sum_{k = 0}^{p - 1}  h_{2k}(\mb t) z^{2k} = 1 + \sum_{k = 0}^{p - 1} \left( \sum_{j = 0}^{p - 1 - k} \binom{2j}{j} t_{2k + 2j+2} \sigma^j(\mb t) \right) z^{2k},
        \label{eq:h-formula}
        \end{equation}
    and $\sigma(\mb t)$ is the solution to the Equation \eqref{eq:sigma-poly} satisfying $\sigma(\mb t) \to 1$ as $\mb t \to 0$.
\end{proposition}

Using the equilibrium measure, one proceeds as in \cite{MR1702716} and constructs the so-called \emph{$g$-function}: 
\begin{equation}
    g(z; \mb t) := \int \log(z - s) \dd \mu_V(s), \quad z \in \C \setminus (-\infty, \beta(\mb t)].
    \label{eq:g-fun}
\end{equation}
Function \eqref{eq:g-fun}, in turn, can be used for the asymptotic evaluation of the matrix $\mb Y(z; \mb t)$. Roughly speaking, the function $g(z; \mb t)$ allows one to transform Riemann-Hilbert Problem \ref{rhp:y} to a problem with constant coefficients \emph{and} identity normalization at infinity (cf. Riemann-Hilbert problem \ref{rhp:phi}, particularly \eqref{eq:phi-expansion}). This latter problem can be (asymptotically) solved using 
\begin{equation}
    \mb M(z; \mb t) := \begin{bmatrix}
        \dfrac{\gamma(z) + \gamma^{-1}(z)}{2} & \dfrac{\gamma(z) - \gamma^{-1}(z)}{2\ii} \medskip \\
        \dfrac{\gamma(z) - \gamma^{-1}(z)}{-2\ii} & \dfrac{\gamma(z) + \gamma^{-1}(z)}{2}
    \end{bmatrix}, 
    \label{eq:global-parametrix} 
\end{equation}
where 
\begin{equation}
    \gamma(z; \mb t) := \left(\dfrac{z - \beta(\mb t)}{z - \alpha(\mb t)} \right)^{\frac14}, \quad z \in \C \setminus [\alpha(\mb t), \beta (\mb t)].
\end{equation}

\begin{proposition}\emph{(\cite[Eq. (3.60) and Theorem 3.6]{MR1953782}, \cite[Eq. (4.116)]{MR1702716})}.
    Let $N = n$ and $\mb Y(z; \mb t)$ be the solution to the corresponding Riemann-Hilbert problem \eqref{rhp:y} and $\mb M(z; \mb t)$ be as in \eqref{eq:global-parametrix}.  Then, there exist a $2\times 2$ matrix function $\mb R(z; \mb t)$ and a constant $\ell \in \C$ such that for $z \in \C \setminus [\alpha(\mb t), \beta(\mb t)]$ 
    \begin{equation}
        \mb Y(z; \mb t) = \ee^{N \ell \sigma_3 /2} \mb R(z; \mb t) \mb M(z; \mb t) \ee^{N (g(z; \mb t) - \ell/2) \sigma_3 },
    \end{equation}
    and, on compact subsets of $\C \setminus [\alpha(\mb t), \beta(\mb t)]$, $\mb R(z; \mb t)$ admits the asymptotic expansion 
    \begin{equation}
        \mb R(z; \mb t) \sim \I + \sum_{j = 1}^\infty \dfrac{\mb R_j(z; \mb t)}{N^j}, \quad N \to \infty.
        \label{eq:R-expansion}
    \end{equation}
    The coefficients $\mb R_j(z; \mb t)$ are piecewise analytic functions of $z$ which are analytic in a neighborhood of infinity with 
    \begin{equation}
        \lim_{z \to \infty} \mb R_j(z; \mb t) = \mb 0.
    \end{equation}
    Furthermore, $\mb R_j(z; \mb t)$ are analytic in $\mb t$ for $\mb t\in \mathbb{T}(T, \gamma)$, and can be iteratively and explicitly computed. Expansion \eqref{eq:R-expansion} remains valid upon term-by-term differentiation.
    \label{prop:R-expansion}
\end{proposition}

In view of Proposition \ref{prop:JMU-free-energy}, Proposition \ref{prop:R-expansion} implies an asymptotic expansion of ${\bf d}\mc F_{n, N}(\mb t)$ which can be evaluated term-by-term and integrated. To arrive there, we will use Proposition \ref{prop:eq-measure} to explicitly compute matrices appearing in \eqref{eq:R-expansion}.
\begin{proposition}\label{prop:R-expansion-explicit}
    Let $\mb R(z; \mb t), \mb R_j(z; \mb t)$ be as in Proposition \ref{prop:R-expansion}. Then, there exists $\delta > 0$ such that, for\footnote{We use the notation $B(z_0, \delta) = \{z \in \C \ : \ |z - z_0| < \delta\}$} $z \in \C \setminus B(-2\sigma^\frac12(\mb t), \delta) \cup B(2\sigma^\frac12(\mb t), \delta)$, 
    \begin{multline}
        \mb R_1(z; \mb t) =  \dfrac{1}{(z - 2\sigma^\frac12(\mb t))^2} \mb U_{12}(\mb t) + \dfrac{1}{z - 2\sigma^\frac12(\mb t)} \mb U_{11}(\mb t)  \\ + \dfrac{1}{(z + 2\sigma^\frac12(\mb t))^2} \sigma_3\mb U_{12}(\mb t)\sigma_3 - \dfrac{1}{z + 2\sigma^\frac12(\mb t)} \sigma_3\mb U_{11}(\mb t)\sigma_3   
        \label{eq:R-1}
    \end{multline}
    and 
   \begin{equation}
    \mb R_2(z; \mb t) = -\sum_{j = 1}^3 \dfrac{1}{(z - 2\sigma^\frac12(\mb t))^j} \mb U_{2j}(\mb t) -  \sum_{j = 1}^3 \dfrac{(-1)^j}{(z + 2\sigma^\frac12(\mb t))^j} \sigma_3\mb U_{2j}(\mb t) \sigma_3
    \label{eq:R2-z-expansion}
\end{equation}
     Omitting the dependence on $\mb t$ for brevity, matrices $\mb U_{ij}$ are given explicitly in terms of $\sigma(\mb t)$:
    \begin{align*}
        \mb U_{11}(\mb t) &= \frac{1}{96 \sigma^\frac12(\mb t)P_\sigma^2}\begin{bmatrix}3 P_\sigma + 4 \sigma P_{\sigma\sigma} & 4\ii\left( P_\sigma -  \sigma P_{\sigma\sigma}\right)\\
                4\ii\left( P_\sigma -  \sigma P_{\sigma\sigma}\right) & -3 P_\sigma - 4 \sigma(\mb t) P_{\sigma\sigma}\end{bmatrix}, \qquad 
        \mb U_{12}(\mb t) = \dfrac{5}{48 P_\sigma} \begin{bmatrix}
            -1 & \ii \\ \ii & 1
        \end{bmatrix} ,\\
    \mb U_{23}(\mb t) &= \dfrac{35}{4608 \sigma^{\frac12}P_\sigma^2  }\begin{bmatrix} 1 & -12\ii \\ 12\ii & 1   
    \end{bmatrix}, \qquad  \mb U_{21}(\mb t) = \dfrac{1}{36864 \sigma^\frac32 P_\sigma^4 } (f(\mb t) \I + g(\mb t) \sigma_2) \\
    \mb U_{22}(\mb t) &= -\dfrac{1}{18432\sigma P_\sigma^3 } \begin{bmatrix}
        13 P_\sigma + 96 \sigma P_{\sigma \sigma} & -4\ii (61 P_\sigma + 332\sigma P_{\sigma \sigma})\\
        4\ii (61 P_\sigma + 332\sigma P_{\sigma \sigma}) & 13 P_\sigma + 96 \sigma P_{\sigma \sigma}
    \end{bmatrix}
    \end{align*}
    where subscript $\sigma$ denotes a derivative in $\sigma$, 
    \begin{align}
       \label{eq:P-formula} P(\sigma; \mb t) &= \sigma + \sum_{k = 1}^p \binom{2k-1}{k} t_{2k} \sigma^k, \qquad P \equiv P(\sigma(\mb t); \mb t),
    \end{align}
    and 
    \begin{align}
        f(\mb t) &=  45 P_\sigma^2 + 32 \sigma P_\sigma P_{\sigma \sigma} + 32 \sigma^2 P_{\sigma \sigma}^2 , \\
        g(\mb t) &= 64 (P_\sigma^2 + 25 \sigma^2P_{\sigma \sigma}^2  - 2 \sigma P_{\sigma} (P_{\sigma \sigma} + 6\sigma  P_{\sigma \sigma \sigma} )).
    \end{align}
\end{proposition}

The proof of Proposition \ref{prop:R-expansion-explicit} is a technical calculation which we defer to Appendix \ref{sec:prop-R-pf}. Before moving on, we record a useful consequence of Proposition \ref{prop:R-expansion-explicit}. 
\begin{proposition}
    The recurrence coefficients $R_{N,N}(\mb t)$ defined in \eqref{eq:3-term} admit an asymptotic expansion of the form 
    \begin{equation}
        R_{N,N}(\mb t) \sim \sum_{g = 0}^\infty \dfrac{r_g(\mb t)}{N^{2g}}, \qasq N \to \infty
        \label{eq:recurrence-coef-asymptotics}
    \end{equation}
    where the coefficients $r_g(\mb t)$ are analytic in $\mb t$ and can be explicitly computed in terms of $\sigma(\mb t)$. In particular, 
    \begin{equation}
        r_0(\mb t) = \sigma(\mb t)  \qandq   r_1(\mb t) = \dfrac{ \sigma(\mb t)(2 P_{\sigma\sigma}^2(\sigma(\mb t); \mb t) - P_\sigma(\sigma(\mb t); \mb t) P_{\sigma \sigma \sigma}(\sigma(\mb t); \mb t) )}{12 P_\sigma^4(\sigma(\mb t); \mb t)}.
        \label{eq:r0-r1}
    \end{equation}
    \label{prop:recurrence-expansion}
\end{proposition}
\begin{remark}
    Expressions \eqref{eq:r0-r1} reduce to those in the case of the quartic model \cite[Eq. (6.86)]{MR4673885} and sextic model\footnote{In both cases, the authors use the index $r_{2g}(\mb t)$ for what we denote $r_g(\mb t)$.} \cite[Eq. (B7)]{GL}. In the general regular case, a version of these formulas appears in \cite[Eq. (5.17)]{MR2367197}, and the corresponding Taylor expansion is given in \cite{ELT}. 
\end{remark}
\begin{proof}
    That the recurrence coefficients admit an expansion in inverse powers of $N$ follows from the classical work \cite{MR1702716}. That the expansion is, in fact, in inverse powers of $N^{2}$ is shown in \cite[Theorem 5.2]{MR2187941}. To compute the coefficients, we note that, from \eqref{eq:y} (see, e.g., \cite[Eq. (3.21)]{MR1711036}), we have 
    \[
    R_{n, N}(\mb t) = [\mb Y^{(1)}(\mb t)]_{12} [\mb Y^{(1)}(\mb t)]_{21}
    \]
    where $[\mb A]_{ij}$ indicates the $(i, j)$-entry of the matrix $\mb A$. From Proposition \ref{prop:R-expansion}, it follows that $\mb R_j(z; \mb t)$ admit a Laurent expansion in a neighborhood of infinity of the form 
    \[
        \mb R_j(z; \mb t) = \dfrac{1}{z} \mb R_j^{(1)}(\mb t) + \Oo(z^{-2}).
    \]
    Matrices $\mb R_j^{(1)}(\mb t)$, $j = 1, 2$, can be readily computed using the explicit formulas for $\mb R_1(z; \mb t), \mb R_2(z; \mb t)$ in Proposition \ref{prop:R-expansion-explicit}, and the result is, again omitting dependence on $\mb t$ for brevity,
    \begin{equation}
    \begin{aligned}
        \mb R_1^{(1)}(\mb t) &= \mb U_{11}(\mb t) - \sigma_3 \mb U_{11}(\mb t) \sigma_3 = \dfrac{\ii (P_\sigma - \sigma P_{\sigma \sigma})}{\sigma^\frac12  P_{\sigma}^2} \sigma_1, \medskip \\
        \mb R_2^{(1)}(\mb t) &= -\mb U_{21}(\mb t) + \sigma_3 \mb U_{21}(\mb t) \sigma_3 = -\dfrac{ P_\sigma^2 + 25 \sigma^2 P_{\sigma \sigma}^2  - 
        2 \sigma P_\sigma (P_{\sigma \sigma} + 6 \sigma P_{\sigma \sigma\sigma} )}{288 \sigma^{\frac32} P_\sigma^4 } \sigma_2.
    \end{aligned}
    \label{eq:R1-R2-laurent}
    \end{equation}
    Using the definition of $\mb M(z; \mb t)$ in \eqref{eq:global-parametrix} and \eqref{eq:R-expansion}, we have 
    \begin{align*}
        R_{N, N}(\mb t) &= \left(\ii \sigma^{\frac12}(\mb t) + \frac{1}{N} [\mb R_1^{(1)}(\mb t)]_{12} + \frac{1}{N^2} [\mb R_2^{(1)}(\mb t)]_{12} + \Oo(N^{-3})\right) \\
        & \qquad \qquad  \times \left(-\ii \sigma^{\frac12}(\mb t) + \frac{1}{N} [\mb R_1^{(1)}(\mb t)]_{21} + \frac{1}{N^2} [\mb R_2^{(1)}(\mb t)]_{21} + \Oo(N^{-3}) \right)\\
        &= \sigma(\mb t) + \dfrac{1}{N^2} \left( [\mb R_1^{(1)}(\mb t)]_{12}[\mb R_1^{(1)}(\mb t)]_{21}  - \ii \sigma^\frac12 \left( [\mb R_2^{(1)}(\mb t)]_{12}-[\mb R_2^{(1)}(\mb t)]_{21}  \right)\right) + \Oo(N^{-4}).
    \end{align*}
    The result now follows by plugging in \eqref{eq:R1-R2-laurent}. 
\end{proof}

We can now find an expression for the first two terms in the asymptotic expansion of the JMU differential \eqref{eq:jmu-def}
\begin{proposition}\label{prop:dF-formulas}
    Let $F_g(\mb t)$ be as in \eqref{eq:top-exp}. With the notation of Proposition \eqref{prop:R-expansion-explicit}, we have 
    \begin{equation}
        \mb d F_0(\mb t) = \frac14\sum_{k = 1}^p \dfrac{1}{k} \res_{z = \infty} \left( z^{2k} h(z; \mb t) w(z; \mb t) \right) \dd t_{2k}, 
        \label{eq:dF-0}
    \end{equation}
    and 
    \begin{equation}
        \mb d F_1(\mb t) = \frac{\sigma(\mb t)}{12P^2_\sigma(\sigma(\mb t); \mb t)}\sum_{k = 1}^{p}\frac{1}{k } \res_{z=\infty} \left(z^{2k} \cdot \frac{ z^2 P_{\sigma \sigma }(\sigma(\mb t); \mb t) - \left(6 P_\sigma(\sigma(\mb t); \mb t) + 4 \sigma(\mb t) P_{\sigma \sigma}(\sigma(\mb t); \mb t) \right) }{(z^2-4\sigma(\mb t))^{5/2}}  \right) \dd t_{2k}.
        \label{eq:dF-1}
    \end{equation}
\end{proposition}
\begin{proof}
    Recalling Proposition \ref{prop:R-expansion}, we have 
    \begin{equation}
        \begin{aligned}
        \mb{Y}^{-1}(z)\mb{Y}'(z) &= \ee^{-N(g(z; \mb t) - \frac{1}{2}\ell)\sigma_3}\mb M^{-1}(z; \mb t)\mb{R}^{-1}(z; \mb t)\mb{R}'(z; \mb t) \mb M(z; \mb t)\ee^{N(g(z; \mb t) - \frac{1}{2}\ell)\sigma_3}\\
        &+ \ee^{-N(g(z;\mb t) - \frac{1}{2}\ell)\sigma_3} \mb M^{-1}(z; \mb t) \mb M'(z; \mb t)\ee^{N(g(z; \mb t) - \frac{1}{2}\ell)\sigma_3} + Ng'(z; \mb t)\sigma_3.
        \label{eq:y-inv-y-prime}
        \end{aligned}
    \end{equation}
    Since $\mb{R}(z;\boldsymbol{t}) = \mathbb{I} + \Oo(n^{-1})$ as $n\to \infty$, $\mb{R}^{-1}(z; \mb t)\mb{R}'(z; \mb t) = \Oo(n^{-1})$, and so, for any $k\geq 1$,
    \begin{equation}
        \Tr\left(\ee^{-N(g(z; \mb t) - \frac{1}{2}\ell)\sigma_3}\mb M^{-1}(z; \mb t)\mb{R}^{-1}(z; \mb t)\mb{R}'(z; \mb t) \mb M(z; \mb t)\ee^{N(g(z) - \frac{1}{2}\ell)\sigma_3}z^{2k}\sigma_3\right) = \Oo(n^{-1}).
        \label{eq:a-error}
    \end{equation}
    That is, up to the leading order in $N$, the first term does not contribute to $\omega_{N, N}$. Furthermore, 
    \[
    \mb M^{-1}(z; \mb t) \mb M'(z; \mb t) = \frac{{\sigma}^\frac{1}{2}(\mb t)}{z^2-4\sigma(\mb t)}\sigma_2,
    \]
    and so, for any $k\geq 1$,
        \begin{multline*}
            \Tr\left(\ee^{-N(g(z; \mb t) - \frac{1}{2}\ell ) \sigma_3}\mb M^{-1}(z; \mb t) \mb M'(z; \mb t) \ee^{N(g(z; \mb t) - \frac{1}{2}\ell)\sigma_3} z^{2k}\sigma_3\right) \\
            = \frac{\ii z^{2k}\sigma^\frac12(\mb t)}{z^2-4\sigma(\mb t)}\Tr\left( \ee^{-2N(g(z; \mb t) - \frac{1}{2}\ell ) \sigma_3}  \sigma_1 \right)  = 0.
        \end{multline*}
    The leading order of $\omega_{N,N}$ is then determined by the final term in \eqref{eq:y-inv-y-prime}, namely 
    \begin{equation*}
        \omega_{N,N} = -\sum_{k=1}^{p}\frac{N^2}{2k} \res_{z=\infty}\left(g'(z; \mb t)z^{2k}\right) \dd t_{2k} + \Oo(1).
    \end{equation*}
    Since $Z_{N,N}({\bf 0})$ is a constant,
        \begin{equation*}
            {\bf d} F_0({\bf t}) = \lim_{N\to \infty}\frac{1}{N^2}{\bf d} \log \frac{Z_{N,N}({\bf t})}{Z_{N,N}({\bf 0})} = \lim_{N\to \infty}\frac{1}{N^2}\omega_{N,N} = -\sum_{k=1}^{p}\frac{1}{2k}\res_{z=\infty}\left( g'(z; \mb t) z^{2k} \right) \dd t_{2k}.
        \end{equation*}
    It follows from the definition of $g(z; \mb t)$ in \eqref{eq:g-fun} (see e.g. \eqref{eq:g-V-jump-1} below) and the fact that $V'(z)z^k$ is a polynomial and thus contains no residue, that
    \[
    \res_{z=\infty}\left( g'(z; \mb t)z^{2k}\right) = -\frac{1}{2} \res_{z = \infty} \left(h(z; \mb t) w(z; \mb t) z^{2k} \right),
    \]
    which finishes the proof of \eqref{eq:dF-0}. Since the second term of \eqref{eq:y-inv-y-prime} is traceless and the third's dependence on $N$ is explicit, showing \eqref{eq:dF-1} requires us to further expand \eqref{eq:a-error}. Proposition \ref{prop:R-expansion} implies
    \begin{equation}
        \mb R^{-1}(z;\mb t) \mb R'(z ; \mb t) =\frac1N \mb R_1'(z; t) + \Oo(N^{-2})
    \end{equation}
    and so, recalling \eqref{eq:global-parametrix}, 
    \begin{multline}
         \frac{N}{4k}\res_{z = \infty} \Tr\left(\ee^{-N(g(z; \mb t) - \frac{1}{2}\ell)\sigma_3}\mb M^{-1}(z; \mb t)\mb{R}^{-1}(z; \mb t)\mb{R}'(z; \mb t) \mb M(z; \mb t)\ee^{N(g(z) - \frac{1}{2}\ell)\sigma_3}z^{2k}\sigma_3\right) \\
         = \frac{1}{4k}\res_{z=\infty}\left(\frac{z^{2k}}{w(z; \mb t)} \Tr\left(\mb{R}_1'(z; \mb t)
            \begin{bmatrix}
                z & -2\ii\sigma^\frac12 (\mb t)\\
                -2\ii \sigma^\frac12 (\mb t) & -z
            \end{bmatrix}\right)\right) + \Oo(N^{-2}).
    \end{multline}
    Plugging in \eqref{eq:R-1} yields \eqref{eq:dF-1}.
\end{proof}

\subsection{Asymptotics along almost Diagonal Sequences}
\label{sec:t'hooft-parameter}

While the asymptotic expansion \eqref{eq:recurrence-coef-asymptotics} was stated for the ``diagonal" case $n = N$, in Section \ref{sec:main-thm-pf} we will require expansions along more general sequences. To this end, let 
\begin{equation}
    \varkappa = \dfrac{n}{N}.
    \label{eq:tHooft}
\end{equation}
We will often refer to $\varkappa$ appearing in \eqref{eq:tHooft} as the \emph{t'Hooft parameter}. 
\begin{proposition}[\cite{MR2187941}]
Let $T>0, \gamma >0$ be as in Theorem \ref{thm:top-exp} and fix $\mb t \in \T(T, \gamma)$. Then, there exists $\mathfrak{e} >0$ such that for all $\varkappa \in [1- \mathfrak{e}, 1 + \mathfrak{e}]$, the recurrence coefficients $R_{n, N}(\mb t )$ defined in \eqref{eq:3-term} admit an asymptotic expansion of the form 
\begin{equation}
    R_{n, N}(\mb t ) \sim \sum_{g = 0}^\infty \dfrac{r_g(\mb t; \varkappa )}{N^{2g}}, \qasq N \to \infty
    \label{eq:recurrence-coef-asymptotics-lambda}
\end{equation}
where the coefficients $r_g(\mb t; \varkappa )$ are analytic in $\mb t$ and in $\varkappa$ and can be explicitly computed. In particular, let $\sigma(\mb t; \varkappa)$ be the unique solution of 
\begin{equation}
     \varkappa = \sigma(\mb t; \varkappa) + \sum_{k = 1}^p \binom{2k-1}{k} t_{2k} \sigma^k(\mb t; \varkappa), \qquad \sigma(\mb 0; \varkappa) = \varkappa.
     \label{eq:sigma-lambda-poly}
\end{equation}
Then, 
\begin{equation}
    r_0(\mb t; \varkappa) = \sigma(\mb t; \varkappa)  \qandq   r_1(\mb t; \varkappa) = \dfrac{ \sigma(\mb t; \varkappa)(2 P_{\sigma\sigma}^2(\sigma(\mb t; \varkappa); \mb t) - P_\sigma(\sigma(\mb t; \varkappa); \mb t) P_{\sigma \sigma \sigma}(\sigma(\mb t; \varkappa ); \mb t) )}{12 P_\sigma^4(\sigma(\mb t; \varkappa); \mb t)}.
    \label{eq:r0-r1-lambda}
\end{equation}
In particular, note that the dependence on $\varkappa$ in \eqref{eq:r0-r1-lambda} is through $\sigma(\mb t; \varkappa)$.
\label{prop:recurrence-expansion-lambda}
\end{proposition}
\begin{proof}
    The proofs of the existence of an expansion \eqref{eq:recurrence-coef-asymptotics-lambda} and analyticity of the coefficients appear in \cite[Theorem 5.2 and Eq. (5.4)]{MR2187941}. To prove the first identity in \eqref{eq:r0-r1-lambda}, we note that 
    \[
    NV(z; \mb t) = n \varkappa^{-1} V(z; \mb t) = nV(\varkappa^{-\frac12} z; \mb t_\varkappa), \quad \mbox{where} \quad \mb t_\varkappa = (t_2, \varkappa t_4, \varkappa^{2} t_6, \cdots, \varkappa^{p - 1} t_{2p}).
    \]
    Thus, the equilibrium measure corresponding to the potential $\varkappa^{-1} V(z; \mb t)$ is related to the equilibrium measure $V(z; \mb t)$ by making the change $\mb t \to \mb t_\varkappa$ and $\sigma \to \varkappa^{-1} \sigma$. In particular, $\supp(\mu_{\varkappa^{-1} V}) = [-2\sigma^{1/2}(\mb  t; \varkappa), 2\sigma^{1/2}(\mb t; \varkappa)]$, where $\sigma(\mb t; \varkappa)$ satisfies \eqref{eq:sigma-lambda-poly}. Now the first identity follows from \cite[Eq. (5.4)]{MR2187941}. The proof of the second identity is exactly the same as in Proposition \ref{prop:recurrence-expansion}. 
\end{proof}


\subsection{Proof of \texorpdfstring{Proposition \ref{prop:eq-measure}}{Proposition on equilibrium measure}}
We begin by noting that, since $V(z; \mb t)$ is an even function, we have $\alpha(\mb t) = -\beta(\mb t) $. Denote $4\sigma(\mb t) = (\alpha(\mb t))^2 = (\beta(\mb t))^2$, we will now show that $\sigma(\mb t)$ is given by \eqref{eq:sigma-def}. To do so, let 
\begin{equation}
    w(z; \mb t) := ((z - \alpha(\mb t))(z - \beta(\mb t)))^{\frac12} = (z^2 - 4\sigma(\mb t))^{\frac12}, \qquad z \in \C \setminus [\alpha(\mb t), \beta (\mb t)]
\end{equation} 
be the branch satisfying 
\begin{equation}
    w(z; \mb t) = z + \Oo(z^{-1}) \qasq z \to \infty.
    \label{eq:w-asymp}
\end{equation}
Then, the equilibrium measure \eqref{eq:equilibrium-measure} can be rewritten as 
\[
    \dd \mu_V(z) = \dfrac{1}{2 \pi \ii }  h(z; \mb t) w_+(z; \mb t) \ \dd z, \quad \supp(\mu_V) = \left(-2\sigma^\frac12(\mb t), 2\sigma^\frac12(\mb t)\right)
\]
Next, observe that $g(z; \mb t)$ has continuous boundary values on $(-\infty, \beta(\mb t))$, denoted $g_{\pm}(z)$ (see Footnote \ref{ftnote:real-orient}), and that the relationship between $g_{\pm}(z; \mb t)$ and $V(z; \mb t)$ can be made explicit. Indeed, it follows from this, the definition of the $g$-function \eqref{eq:g-fun}, and the Euler-Lagrange variational conditions (see e.g. \cite[Theorem I.1.3]{MR1485778}) that there exists a constant $\ell(\mb t) \in \R$ such that
\begin{itemize}
    \item for $z \in [\alpha(\mb t), \beta (\mb t)]$, 
    \begin{equation}
        g_+(z; \mb t) + g_-(z; \mb t) - V(z; \mb t) = \ell,
    \end{equation}
    \item for $z >\beta(\mb t)$, 
    \begin{equation}
        g_+(z; \mb t) + g_-(z; \mb t) - V(z; \mb t) = 2g(z; \mb t) - V(z; \mb t) - \ell(\mb t) = -\int_{\beta(\mb t)}^{z} w(s; \mb t) h(s; \mb t) \dd s.
        \label{eq:g-V-jump-1}
    \end{equation}
    \item Similarly, for $z < \alpha(\mb t)$, 
    \begin{equation}
        g_+(z; \mb t) + g_-(z; \mb t) - V(z; \mb t) =  -\int_{\alpha(\mb t)}^{z} w(s; \mb t) h(s; \mb t) \dd s.
    \end{equation}
\end{itemize}
Equation \eqref{eq:g-V-jump-1} can be used to compute $h(z; \mb t), \sigma(\mb t)$. Differentiating both sides of \eqref{eq:g-V-jump-1} with respect to\footnote{We denote $\pd{}{z}(\diamond) = (\diamond)'$} $z$, we find
\begin{equation}
   \frac12 V'(z; \mb t)  -  \int_{-2\sigma^\frac12(\mb t)}^{2 \sigma^\frac12(\mb t)} \dfrac{h(s; \mb t) w_+(s; \mb t) }{z - s}\dd s = \frac12 w(z; \mb t) h(z; \mb t).
\end{equation}
which implies 
\[
    h(z; \mb t) = \dfrac{1}{w(z; \mb t)} V'(z) + \Oo (z^{-1}) \qasq z \to \infty.
\]
Plugging in the Laurent expansion of the right hand side, we find \eqref{eq:h-formula}:
\begin{equation*}
    h(z; \mb t) = \sum_{k = 0}^{p - 1}  h_{2k}(\mb t) z^{2k} = 1 + \sum_{k = 0}^{p - 1} \left( \sum_{j = 0}^{p - 1 - k} \binom{2j}{j} t_{2k + 2j+2} \sigma^j(\mb t) \right) z^{2k}.
\end{equation*}
The requirement that $\mu_V$ is a probability measure now reads 
\begin{multline}
    1 = \dfrac{1}{2\pi \ii}\int_{-2\sqrt{\sigma}}^{2\sqrt{\sigma}} h(s; \mb t) w_+(s; \mb t) \dd s = \frac12 \res_{z = \infty} \left(h(z; \mb t) w(z; \mb t)\right) \\
    = \sigma(\mb t) + \sum_{n = 1}^{p}\left(\sum_{k = 0}^{n+1} \dfrac{1}{k+ 1} \binom{2k}{k} \binom{2(n - 1 - k)}{n - 1 -k} \right) t_{2n} \sigma(\mb t)^{n}.
    \label{eq:sigma-poly-unsimple}
\end{multline}
The final step is to match the coefficients of \eqref{eq:sigma-poly-unsimple} to those of \eqref{eq:sigma-poly}. 
\begin{lemma}
    For $n \in \N$, 
    \[
        \sum_{k = 0}^{n+1} \dfrac{1}{k+ 1} \binom{2k}{k} \binom{2(n - 1 - k)}{n - 1 -k} = \binom{2n - 1}{n}
    \]
\label{lemma:sigma-coef}
\end{lemma}
\begin{proof}[Proof of Lemma \ref{lemma:sigma-coef}]
    We recognize the left hand side as a convolution of the $k$th Catalan number with a binomial coefficient. Noting the identities 
\[
    C(x):= \dfrac{1 - \sqrt{1 - 4x}}{2x} = \sum_{k = 0}^{\infty} \dfrac{1}{k+ 1} \binom{2k}{k}x^k \qandq B(x):=\dfrac{1}{\sqrt{1 - 4x}} = \sum_{k = 0}^{\infty} \binom{2k}{k} x^k
\]
We have 
\begin{multline*}
    \sum_{k = 0}^{n+1} \dfrac{1}{k+ 1} \binom{2k}{k} \binom{2(n - 1 - k)}{n - 1 -k} = [x^{n-1}] C(x)B(x) \\
    = [x^{n - 1}] \dfrac{1}{2x} \left( \dfrac{1}{\sqrt{1 - 4x}} - 1\right) = \dfrac{1}{2} \binom{2n}{n} = \binom{2n - 1}{n}
\end{multline*}
\end{proof}
Plugging the result of Lemma \ref{lemma:sigma-coef} into \eqref{eq:sigma-poly-unsimple} yields \eqref{eq:sigma-poly}.

\section{More on the function \texorpdfstring{$\sigma(\mb t)$}{Sigma}} \label{sec:sigma-section}
The proof of Theorem \ref{thm:g-1} and reproduction of Tutte's Formula \eqref{eq:count-g-0} require the multivariate Taylor expansion of $\sigma(\mb t)$, which we obtain in this section. Note that, up to rescaling the variable and a change of notation, equation \eqref{eq:sigma-poly} can be thought of as a \emph{general} polynomial equation, and $\sigma(\mb t)$ its formal solution. As such, the problem of finding a series expansion for $\sigma(\mb t)$ is classical and various special cases were worked out as examples of his inversion method by Lagrange himself \cite{L}. For a modern proof of Lagrange Inversion, see \cite[Section 7.32]{MR4286926}, \cite[Theorem 2.1.1]{Gessel} or \cite[Theorem 1]{SW}, where special cases of the following result are also presented. 

\begin{proposition}
    Let $\sigma(\mb t)$ be the solution of the polynomial equation \eqref{eq:sigma-poly} satisfying 
    \(
        \lim_{\mb t \to \mb 0} \sigma(\mb t) = 1.
    \)
    Then, there exists $r >0$ such that for $|\mb t| < r$, $\sigma(\mb t)$ has the convergent power series expansion \eqref{eq:sigma-def}.
    \label{prop:sigma-series}
\end{proposition}
\begin{proof}
    For $k \in \N$, let
    \begin{equation}
        x_{2k} = -\dfrac{t_{2k}}{2k} \qandq u({\bf x}) = \sigma({\bf t}).
        \label{eq:x-change}
    \end{equation}
    Then, \eqref{eq:sigma-poly} reads 
    \begin{equation*}
        u({\bf x}) = 1 + \sum_{k=1}^{p} k\binom{2k}{k} x_{2k}u({\bf x})^{k}.
    \end{equation*}
    We now consider the equation
        \begin{equation}
            U(z;{\bf x}) = 1+ z \left(\sum_{k=1}^{p} k\binom{2k}{k} x_{2k}U(z;{\bf x})^{k}\right)
            \label{eq:U}
        \end{equation}
    and note that $U(1; \mb x) = u(\mb x)$ and $U(0; \mb x) = 0$. In the notation of \cite[Section 7.32]{MR4286926}, setting $a = 1$, $f(z) = z$, and
        \begin{equation}
            \phi(\tau;{\bf x}) := \sum_{k=1}^{p} k\binom{2k}{k} x_{2k}\tau^{k},
        \end{equation}
     we conclude that
    \begin{equation}
        U(z; \mb x) := 1 + \sum_{n = 1}^\infty U_N({\bf x}) z^N , \quad \text{ where }  \quad  U_N(\mb x)= \frac{1}{N}[\tau^{N-1}]\phi^{N}(\tau;{\bf x}).
        \label{eq:U-series}
    \end{equation}
    solves \eqref{eq:U} and, for any $\mb t \in \T(T, \gamma)$, is a convergent power series near $z = 0$. In fact, for any contour $C$ encircling $\tau = 1$, we have that \eqref{eq:U-series} converges on 
    \[
        |z| \leq \inf_{\tau \in C} \dfrac{|\tau - 1|}{|\phi(\tau; \mb x)|}.
    \]
    Since $\phi(\tau; \mb x)$ is continuous in $\mb x$ and $\phi(\tau; \mb 0) = 0$, for any given $C$ we may choose $r(C)$ such that for all $|\mb t| < r$, \eqref{eq:U-series} converges on $|z| \leq 1 $. 
    
    Fix a sequence $(n_2,n_4,\cdots,n_{2p})$ and take $N:=n_2 + n_4 + \cdots + n_{2p}$. Since $\phi(\tau; \mb x)$ is linear in the variables ${\bf x}$, the coefficient of $\prod_{k=1}^px_{2k}^{n_{2k}}$ comes exclusively from $W_N({\bf x})$ by choosing  $n_k$ factors contributing $k\binom{2k}{k}x_{2k}(1+\tau)^{k}$ for $k = 1, 2, ..., p$. There are $ {N!}/{\prod_{k=1}^p (n_{2k})!}$ ways to choose such factors. Thus,
        \begin{align*}
        \left[\prod_{k=1}^px_{2k}^{n_{2k}}\right ]U_N(\mb x) &= \frac{1}{N} \prod_{k=1}^p\left(k\binom{2k}{k}x_{2k}\right)^{n_{2k}} \frac{N!}{\prod_{k=1}^p (n_{2k})!} [\tau^{N-1}] (\tau+1)^{n_2 + 2n_4+3n_6+\cdots + p n_{2p}} \\
        &= \frac{1}{N} \prod_{k=1}^p\left(k\binom{2k}{k}\right)^{n_{2k}} \frac{N!}{\prod_{k=1}^p (n_{2k})!}\binom{n_2 + 2n_4+3n_6+\cdots + pn_{2p}}{N-1},\\
        &=  \frac{(n_2 + 2n_{4} + 3n_{6} + \cdots p n_{2p})!}{(n_4 + 2n_6 + \cdots + (p-1)n_{2p}+1)!}\prod_{k=1}^p  \frac{1}{n_{2k}!} \left(k\binom{2k}{k}\right)^{n_{2k}} 
        \end{align*}
    Setting $z = 1$ in \eqref{eq:U-series},  recalling \eqref{eq:x-change}, and grouping  yields 
    \begin{equation}
        \sigma({\bf t}) = \sum_{\mb n \geq 0}  \frac{(n_2 + 2n_{4} + 3n_{6} + \cdots + p n_{2p})!}{(n_4 + 2n_6 + \cdots + (p-1)n_{2p}+1)!} \prod_{k=1}^p \left(-\frac{t_{2k}}{2k}\right)^{n_{2k}} \frac{1}{n_{2k}!}\left(k\binom{2k}{k}\right)^{n_{2k}}
    \end{equation}
    Rearranging this sum using the identity $\frac12\binom{2k}{k} = \binom{2k - 1}{k}$, 
    \begin{equation}
        \sigma({\bf t}) = \sum_{\mb n \geq 0}  \frac{(n_2 + 2n_{4} + 3n_{6} + \cdots + p n_{2p})!}{(n_4 + 2n_6 + \cdots + (p-1)n_{2p}+1)!} \prod_{k=1}^p \dfrac{1}{n_{2k}!}\left(-\binom{2k-1}{k} t_{2k}\right)^{n_{2k}} 
    \end{equation}
    and it remains to observe that, by the definitions of $C_{\mb v}(\mb n)$ in \eqref{eq:catalan-fuss-def}, with $ v = (1, 2, ..., p)$, and the multinomial coefficient, 
    \begin{equation*}
        C_{\mb v}(\mb n) =  \frac{(n_2 + 2n_{4} + 3n_{6} + \cdots + p n_{2p})!}{(n_4 + 2n_6 + \cdots + (p-1)n_{2p}+1)!} \prod_{k=1}^p\dfrac{1}{n_{2k}!}.
    \end{equation*}

\end{proof}

\section{Proof of \texorpdfstring{Theorem \ref{thm:g-1} and Recovery of Tutte's Formula \eqref{eq:count-g-0}}{Explicit formulas for N0, N1}} \label{sec:tutte-section}

In view of Proposition \ref{prop:dF-formulas}, the proof of formula \ref{eq:count-g-0} and Theorem \ref{thm:g-1} amounts to computing the indicated residues and integrating the resulting differential form.

\subsection{Proof of Formula \texorpdfstring{\eqref{eq:count-g-0}}{Tutte's formula}}
\label{sec:tutte-formula}
\begin{lemma}
For any $m = 1, 2, ..., p$,
    \begin{equation}\label{partial-F0-formula}
        \dpd{F_0}{t_{2m}}= -\frac{1}{2m(m+1)}\binom{2m}{m}\sigma^{m+1}(\mb t) - \frac{1}{2m}\binom{2m}{m}\sum_{k=1}^p \frac{k}{k+m} \binom{2k-1}{k}t_{2k}\sigma^{k+m}(\mb t).
    \end{equation}
\end{lemma}
\begin{proof}
    This proof amounts to calculating the residues in \eqref{eq:dF-0}. In a neighborhood of infinity, the following series expansion holds:
        \begin{equation*}
            w(z; \mb t) = \sum_{\ell=0}^{\infty}\binom{2\ell}{\ell}\sigma^{\ell}(\mb t)\frac{z^{1-2\ell}}{1-2\ell}.
        \end{equation*}
    Recalling \eqref{eq:h-formula}, we have
        \begin{multline*}
            h(z; \mb t) w(z;\mb t) =  \sum_{\ell=0}^{\infty}\binom{2\ell}{\ell}\sigma^{\ell}(\mb t)\frac{z^{1-2\ell}}{1-2\ell} + \frac{1}{2}\sum_{\ell=0}^{\infty}\sum_{k=0}^{p-1}\left(\sum_{j=0}^{p-k-1}\binom{2\ell}{\ell}\binom{2j}{j}t_{2(k+j+1)}\sigma^{j+\ell}(\mb t)\right)\frac{z^{1+2k-2\ell}}{1-2\ell}\\
            =\textit{polynomial part} \hfill\\ + \sum_{\ell=0}^{\infty}\binom{2\ell}{\ell}\sigma^{\ell}(\mb t)\frac{z^{1-2\ell}}{1-2\ell} 
            - \frac{1}{2}\sum_{M=0}^{\infty} \left(\sum_{k=0}^{p-1}\sum_{j=0}^{p-k-1}\binom{2M+2k+2}{M+k+1}\binom{2j}{j}t_{2(k+j+1)}\frac{\sigma^{M+j+k+1}(\mb t)}{2M+2k+1}\right) z^{-2M-1},
        \end{multline*}
    where, in the second equality, $M = \ell - k - 1$. Recalling Propositions \ref{prop:JMU-free-energy} and \ref{prop:dF-formulas}, we have
        \begin{align*}
            \frac{\partial F_0}{\partial t_{2m}} &= \frac{1}{4m}\res_{z=\infty}\left(h(z; \mb t) w(z;\mb t)z^{2m}\right) \\
            & = -\dfrac{1}{(4m)(2m + 1)}\binom{2(m + 1)}{m + 1} \sigma^{m + 1}(\mb t) \\
            &- \frac{1}{4m}\sum_{k=0}^{p-1}\sum_{j=0}^{p-k-1}\binom{2j}{j}\binom{2m+2k+2}{m+k+1}t_{2(k+j+1)}\frac{\sigma^{m+j+k+1}(\mb t)}{2m+2k+1}.
        \end{align*}
        Letting $n = j + k + 1$ and interchanging the sum in the second term we find 
        \begin{multline*}
            - \frac{1}{4m}\sum_{k=0}^{p-1}\sum_{j=0}^{p-k-1}\binom{2j}{j}\binom{2m+2k+2}{m+k+1}t_{2(k+j+1)}\frac{\sigma^{m+j+k+1}(\mb t)}{2m+2k+1} \\
            = -\frac{1}{4m}\sum_{n=1}^p \sum_{k = 0}^{n -1}\binom{2(n - k - 1)}{n-k-1}\binom{2m+2k+2}{m+k+1}t_{2n}\frac{\sigma^{m+n}(\mb t)}{2m+2k+1}.
        \end{multline*}
        The result now follows from the identity (whose proof is analogous to Lemma \ref{lemma:sigma-coef} and the many similar identities in Appendix \ref{sec:prop-R-pf})
        \begin{equation*}
            \sum_{k = 0}^{n - 1} \binom{2(n - k - 1)}{n-k-1}\binom{2(m+k+1)}{m+k+1} \dfrac{1}{2m + 2k + 1} = \dfrac{2n}{n + m} \binom{2m}{m} \binom{2n - 1}{n}.
        \end{equation*}
\end{proof}
We are now in a position to find an expression for $F_0(\mb t)$ in terms of $\sigma(\mb t)$.
\begin{lemma} 
For any $m=1, 2,...,p$, the genus $0$ free energy admits the expression
    \begin{align*}
        F_0({\bf t}) &= \frac{3}{4}\frac{m-1}{m} +\frac{1}{2}\log \sigma -\frac{(2m+1)(m-1)}{2m(m+1)}\sigma +\frac{(m-1)^2}{4m(m+1)}\sigma^2\\
                &+\sum_{\ell=1}^p\binom{2\ell-1}{\ell}\frac{\ell-m}{2m(m+\ell)}t_{2\ell}\left(\frac{2m+\ell}{\ell}\sigma^{\ell}-\frac{(m-1)(2m+\ell+1)}{(m+1)(\ell+1)}\sigma^{\ell+1}\right)\\
                &+\sum_{k,\ell = 1}^{p} \binom{2k-1}{k}\binom{2\ell-1}{\ell}\frac{(k-m)(\ell-m)}{2m(m+\ell)(k+\ell)}t_{2k}t_{2\ell}\sigma^{k+\ell}.
    \end{align*}
    \label{prop:F0-explicit}
\end{lemma}
\begin{proof}
    We are going to make the change of variables ${\bf t}\mapsto ({\bf t}\setminus \{t_{2m}\}, \sigma)$, by using the algebraic equation satisfied by $\sigma$ \eqref{eq:sigma-poly} (which can be solved for $t_{2m}$) to give an expression for $t_{2m}$ in terms of the remaining $t_{2k}$, and the variable $\sigma$. Then, we will integrate the quantity
        \begin{equation}
            G({\bf t}\setminus \{t_{2m}\}, \sigma):=\frac{\partial F_0}{\partial t_{2m}}({\bf t}\setminus \{t_{2m}\}, \sigma)\left(\frac{\partial \sigma}{\partial t_{2m}}({\bf t}\setminus \{t_{2m}\}, \sigma)\right)^{-1}
            \label{eq:G-def}
        \end{equation}
    with respect to $\sigma$, to obtain an expression for $F_0(\mb t)$.
    First, note that for each $m=1, 2,...,p$, by the Implicit Function Theorem:
        \begin{align}
            \left(\frac{\partial \sigma}{\partial t_{2m}}\right)^{-1} &=-\binom{2m-1}{m}^{-1}\sigma^{-m}\left(1+\sum_{\substack{k=1\\k\neq m}}^{p}k\binom{2k-1}{k}t_{2k}\sigma^{k-1}\right) -\frac{mt_{2m}}{\sigma},
        \end{align}
    and
        \begin{equation*}
            t_{2m}(\sigma;{\bf t}\setminus \{t_{2m}\}) = -\binom{2m-1}{m}^{-1}\sigma^{-m}\left[-1+\sigma+\sum_{\substack{k=1\\k\neq m}}^p \binom{2k-1}{k}t_{2k}\sigma^k\right].
        \end{equation*}
    So, we can express $\left(\frac{\partial \sigma}{\partial t_{2m}}\right)^{-1}$ as the $t_{2m}$-independent quantity
        \begin{equation}
             \left(\frac{\partial \sigma}{\partial t_{2m}}\right)^{-1} = -\binom{2m-1}{m}^{-1}\sigma^{-m}\left(\frac{m}{\sigma} - (m-1) + \sum_{\substack{k=1\\k\neq m}}^{p}(k-m)\binom{2k-1}{k}t_{2k}\sigma^{k-1}\right).
             \label{eq:sigma-derivative}
        \end{equation}
    Now, we express $\frac{\partial F_0}{\partial t_{2m}}$ as the $t_{2m}$-independent quantity
    \begin{multline}
        \frac{\partial F_0}{\partial t_{2m}} = -\frac{\binom{2m}{m}}{2m(m+1)}\sigma^{m+1} - \frac{1}{2m}\binom{2m}{m}\sum_{\substack{k=1\\k\neq m}}^p \frac{k}{k+m} \binom{2k-1}{k}t_{2k}\sigma^{k+m} - \frac{1}{4m}\binom{2m}{m}\binom{2m-1}{m}t_{2m}\sigma^{2m}\nonumber\\
        =\frac{1}{4m}\binom{2m}{m}\sigma^{m+1}\left[\frac{m-1}{m+1} -\frac{1}{\sigma} - \sum_{\substack{k=1\\k\neq m}}^p \frac{k-m}{k+m}\binom{2k-1}{k}t_{2k}\sigma^{k-1}\right].
        \label{eq:F0-derivative}
    \end{multline}
    Plugging \eqref{eq:sigma-derivative}, \eqref{eq:F0-derivative} into \eqref{eq:G-def} and using the identity $m\binom{2m-1}{m}=\frac{m}{2}\binom{2m}{m}$, we find 
    \begin{multline*}
        G({\bf t}\setminus \{t_{2m}\}, \sigma)
        = -\frac{\sigma}{2m}\left(\frac{m-1}{m+1} -\frac{1}{\sigma} - \sum_{\substack{k=1\\k\neq m}}^p \frac{k-m}{k+m}\binom{2k-1}{k}t_{2k}\sigma^{k-1}\right) \\
        \hfill \times \left(\frac{m}{\sigma} - (m-1) + \sum_{\substack{k=1\\k\neq m}}^{p}(k-m)\binom{2k-1}{k}t_{2k}\sigma^{k-1}\right)
    \end{multline*}
        which simplifies to 
    \begin{multline}
            G({\bf t}\setminus \{t_{2m}\}, \sigma) =\frac{1}{2\sigma} -\frac{(2m+1)(m-1)}{2m(m+1)} +\frac{(m-1)^2}{2m(m+1)}\sigma\\
            +\sum_{\ell=1}^p\binom{2\ell-1}{\ell}\frac{\ell-m}{2m(m+\ell)}t_{2\ell}\left((2m+\ell)\sigma^{\ell-1}-\frac{(m-1)(2m+\ell+1)}{m+1}\sigma^{\ell}\right)\\
            +\sum_{k,\ell = 1}^{p} \binom{2\ell-1}{\ell}\binom{2k-1}{k}\frac{(k-m)(\ell-m)}{2m(m+\ell)}t_{2k}t_{2\ell}\sigma^{k+\ell-1}.
    \end{multline}
        Integrating with respect to $\sigma$, we obtain
            \begin{equation}
            \begin{aligned}
            F_0({\bf t}) &= \frac{1}{2}\log\sigma -\frac{(2m+1)(m-1)}{2m(m+1)}\sigma +\frac{(m-1)^2}{4m(m+1)}\sigma^2\\
            &+\sum_{\ell=1}^p\binom{2\ell-1}{\ell}\frac{\ell-m}{2m(m+\ell)}t_{2\ell}\left(\frac{2m+\ell}{\ell}\sigma^{\ell}-\frac{(m-1)(2m+\ell+1)}{(m+1)(\ell+1)}\sigma^{\ell+1}\right)\\
            &+\sum_{k,\ell = 1}^{p} \binom{2k-1}{k}\binom{2\ell-1}{\ell}\frac{(k-m)(\ell-m)}{2m(m+\ell)(k+\ell)}t_{2k}t_{2\ell}\sigma^{k+\ell} + c({\bf t} \setminus \{t_{2m}\}),
            \end{aligned}
            \label{eq:F0-with-const}
            \end{equation}
        where $c({\bf t} \setminus \{t_{2m}\})$ depends only on the variables ${\bf t} \setminus \{t_{2m}\}$. Note that the above expression is independent of $t_{2m}$. To determine $c({\bf t} \setminus \{t_{2m}\})$, note that if we replace $t_{2\ell}$ in the above expression by solving for $t_{2\ell}$ in \eqref{eq:sigma-poly} in terms of the remaining variables ${\bf t}\setminus \{t_{2\ell}\}$, one obtains the expression that would arise from simply interchanging $m\leftrightarrow \ell$. Therefore, in \eqref{eq:F0-with-const}, $c({\bf t} \setminus \{t_{2m}\}) \equiv const.$ This constant can be determined
        using the identities $F_0({\bf 0}) = 0$, and that $\sigma({\bf 0}) = 1$.
\end{proof}
\begin{remark}
    In the case of a regular map (say, when $t_{2p}\neq 0, t_{2k}=0$ for $k\neq p$, which necessitates $m=p$), Lemma \ref{prop:F0-explicit} reduces to
        \begin{equation}
            F_0(t_{2p}) = \frac{1}{2}\log\sigma +\frac{3}{4}\frac{p-1}{p} - \frac{(p-1)(2p+1)}{2p(p+1)} \sigma + \frac{(p-1)^2}{4p(p+1)}\sigma^2.
            \label{eq:F0-pure}
        \end{equation}
    Using \eqref{eq:sigma-def}, one can verify that
        \begin{align*}
            \sigma(t_{2p}) &=\sum_{n=0}^{\infty}\left(p\binom{2p}{p}\right)^n\frac{(np)!}{((p-1)n+1)!n!}\left(-\frac{t_{2p}}{2p}\right)^n,\\
            \sigma(t_{2p})^2 &=\sum_{n=0}^{\infty}\left(p\binom{2p}{p}\right)^n\frac{2(np+1)!}{((p-1)n+2)!n!}\left(-\frac{t_{2p}}{2p}\right)^n,\\
            \log\sigma(t_{2p}) &= \sum_{n=1}^{\infty}\left(p\binom{2p}{p}\right)^n\frac{(np-1)!}{((p-1)n)!n!}\left(-\frac{t_{2p}}{2p}\right)^n.
        \end{align*}
    Plugging these into \eqref{eq:F0-pure}, we find
        \begin{equation}
            F_0(t_{2p}) = \sum_{n=0}^{\infty}\left(p\binom{2p}{p}\right)^{n} \frac{(np-1)!}{((p-1)n+2)!}\frac{(-t_{2p}/(2p))^{n}}{n!},
        \end{equation}
    which matches \cite[Eqs. (1) and (22)]{GL}.
\end{remark}

\begin{remark}
Since Lemma \ref{prop:F0-explicit} is valid for any $p\geq 1$, $1\leq m \leq p$, one may consider a formal limit as $m\to \infty$ and then restrict to the case where only the first $p$ times are nonzero to obtain a more symmetric expression for the free energy:
    \begin{multline*}
        F_0({\bf t}) = \frac{3}{4} + \frac{1}{2}\log \sigma - \sigma + \frac{1}{4}\sigma^2 - \sum_{\ell = 1}^{p}\binom{2\ell-1}{\ell} t_{2\ell}\left(\frac{\sigma^{\ell}}{\ell} - \frac{\sigma^{\ell+1}}{\ell+1}\right) \\
        + \sum_{k,\ell = 1}^{p}\binom{2k-1}{k}\binom{2\ell-1}{\ell} t_{2k}t_{2\ell}\frac{\sigma^{k+\ell}}{k+\ell}.
    \end{multline*}
\end{remark}

We now record Taylor series of various functions of $\sigma(\mb t)$ that will be useful, and whose proof is analogous to the proof of Proposition \ref{prop:sigma-series}. 
\begin{lemma}\label{Lagrange-Inversion-Lemma}
Recall the notation \eqref{eq:graph-notation}. For any $\ell\geq 1$, and for a small enough neighborhood of $\mb t = \mb 0$, we have the convergent series expansion
    \begin{align}
        \sigma^{\ell}({\bf t}) &= \sum_{\mb {n}\geq \mb 0}\ell\cdot\dfrac{\left(E(\mb n) +\ell-1\right)!}{\left(E(\mb n) - V(\mb n)+ \ell\right)!} \cdot \prod_{k=1}^p\frac{1}{n_{2k}!}\left(k\binom{2k}{k}\right)^{n_{2k}}\left(-\frac{t_{2k}}{2k}\right)^{n_{2k}},\\
        \log(\sigma(\bf t)) &= \sum_{\substack{\mb{n}\geq \mb 0\\\mb{n}\neq\mb{0}}}\dfrac{\left(E(\mb n)-1\right)!}{\left(E(\mb n) - V(\mb n)\right)!} \cdot \prod_{k=1}^p\frac{1}{n_{2k}!}\left(k\binom{2k}{k}\right)^{n_{2k}}\left(-\frac{t_{2k}}{2k}\right)^{n_{2k}}.
    \end{align}
\end{lemma}
For convenience of exposition, let
    \begin{align*}
        \Phi_a({\bf t}) &:= \frac{3}{4}\frac{m-1}{m} + \frac{1}{2}\log\sigma -\frac{(2m+1)(m-1)}{2m(m+1)}\sigma +\frac{(m-1)^2}{4m(m+1)}\sigma^2,\\
        \Phi_b({\bf t}) &:=\sum_{\ell=1}^p\binom{2\ell-1}{\ell}\frac{\ell-m}{2m(m+\ell)}t_{2\ell}\left(\frac{2m+\ell}{\ell}\sigma^{\ell}-\frac{(m-1)(2m+\ell+1)}{(m+1)(\ell+1)}\sigma^{\ell+1}\right),\\
        \Phi_c({\bf t}) &:=\sum_{k,\ell = 1}^{p} \binom{2k-1}{k}\binom{2\ell-1}{\ell}\frac{(k-m)(\ell-m)}{2m(m+\ell)(k+\ell)}t_{2k}t_{2\ell}\sigma^{k+\ell}.
    \end{align*}
and
\begin{equation*}
    A(\mb n) := \dfrac{\left(E(\mb n)-1 \right)! }{\left(2+E(\mb n) - V(\mb n)\right)!} \cdot \prod_{k=1}^p\frac{1}{n_{2k}!}\left(k\binom{2k}{k}\right)^{n_{2k}}\left(-\frac{t_{2k}}{2k}\right)^{n_{2k}}.
\end{equation*}       
In this notation,
\[
    F_0({\bf t}) = \Phi_a({\bf t}) + \Phi_b({\bf t}) + \Phi_c({\bf t}).
\]
Recalling \eqref{eq:Fg}, formula \eqref{eq:count-g-0} amounts to showing 
\begin{equation}
    \label{eq:goal}
    [{\bf t}^{\mb{n}}]\Phi_a({\bf t}) + [{\bf t}^{\mb{n}}]\Phi_b({\bf t}) + [{\bf t}^{\mb{n}}]\Phi_c({\bf t}) = A(\mb{n}), \qquad \mb n \geq 0.
\end{equation}
This can be directly verified for $|\mb n| = \sum_{k = 1}^{p}{n_{2k}} \leq 1$. For higher order terms, we need more lemmas.

\begin{lemma}\label{lemma-phi-B}
    For $|\mb{n}| \geq 1$, and $1 \leq \ell \leq p$,
        \begin{multline}
            [{\bf t}^{\mb{n}}]\bigg\{t_{2\ell}\left(\frac{2m+\ell}{\ell}\sigma^{\ell}({\bf t})-\frac{(m-1)(2m+\ell+1)}{(m+1)(\ell+1)}\sigma^{\ell+1}({\bf t})\right)\bigg\} \\
            =-\frac{2}{\binom{2\ell}{\ell}}n_{2\ell}\left[(2m+\ell)(E(\mb n) - V(\mb n)+2)-\frac{(m-1)(2m+\ell+1)}{m+1}E(\mb n)\right]A(\mb{n}).
                    \label{eq:intermediate-id}
        \end{multline}
    and
    \begin{equation}
    \begin{aligned}
          \relax  [{\bf t}^{\mb{n}}] \log \sigma({\bf t}) &= (E(\mb n) - V(\mb n)+2)(E(\mb n) - V(\mb n)+1)A(\mb{n}),\\
            [{\bf t}^{\mb{n}}] \sigma({\bf t}) &= E(E(\mb n) - V(\mb n)+2)A(\mb{n}),\\
            [{\bf t}^{\mb{n}}] \sigma^2({\bf t}) &= 2E(\mb n)(E(\mb n)+1)A(\mb{n}).
    \end{aligned}
    \label{eq:intermediate-id-1}
    \end{equation}
    Furthermore, for $|\mb{n}| \geq 2$,
    \begin{equation}
        [{\bf t}^{\mb{n}}] t_{2k}t_{2\ell}\sigma^{k+\ell}({\bf t}) = \left[\dfrac{4 (k+\ell)}{\binom{2k}{k}\binom{2\ell}{\ell}} n_{2k}n_{2\ell} - \dfrac{8\ell}{\binom{2\ell}{\ell}^2} n_{2\ell}\delta_{k\ell} \right]A(\mb{n}).
        \label{eq:intermediate-id-2}
    \end{equation} 
    \label{lemma:intermediate-id}
\end{lemma}
\begin{proof}
    Using Lemma \ref{Lagrange-Inversion-Lemma}, we have
        \begin{align*}
            [{\bf t}^{\mb{n}}]t_{2\ell}\sigma^{\ell}({\bf t}) &= -\frac{2\ell(E(\mb n) - V(\mb n) + 2)}{\binom{2\ell}{\ell}} n_{2\ell} A(\mb{n}),\\
            [{\bf t}^{\mb{n}}]t_{2\ell}\sigma^{\ell+1}({\bf t}) &= -\frac{2(\ell+1)E(\mb n)}{\binom{2\ell}{\ell}} n_{2\ell} A(\mb{n}).
        \end{align*}
    Inserting these into the left hand side of \eqref{eq:intermediate-id} proves the first identity. The proofs of \eqref{eq:intermediate-id-1}, \eqref{eq:intermediate-id-2} are analogous. 
\end{proof}
Using Lemma \eqref{lemma:intermediate-id}, one finds 
\begin{equation}
    \begin{aligned}
        \relax [{\bf t}^{\mb{n}}]\Phi_a({\bf t}) &= \left[\frac{1}{2}V^2(\mb n) - \left(\frac{3m+1}{2m(m+1)} E(\mb n)+\frac{3}{2}\right)V(\mb n) + \left(1 + \frac{3}{2m}E(\mb n) + \frac{1}{m(m+1)}E^2(\mb n)\right)\right]A(\mb{n}), \\
        [{\bf t}^{\mb{n}}]\Phi_b({\bf t}) &=-\sum_{\ell=1}^p \frac{\ell-m}{2m(m+\ell)}\left[\frac{2\ell+3m+1}{m+1}E(\mb n)-(2m+\ell)( V(\mb n) - 2)\right]n_{2\ell}A(\mb{n}),\\
        [{\bf t}^{\mb{n}}]\Phi_c({\bf t}) &= \left(\sum_{k,\ell=2}^p \frac{(k-m)(\ell-m)}{2m(m+\ell)}  n_{2k}n_{2\ell} {- \sum_{\ell=1}^{p}\frac{(\ell-m)^2}{2m(m+\ell)}n_{2\ell}}\right)A(\mb{n}).
    \end{aligned}
    \label{eq:Phi-ids}
\end{equation}
We are now ready to prove \eqref{eq:goal} by direct calculation: 
 \begin{multline*}
            [{\bf t}^{\mb{n}}]\Phi_b({\bf t}) + [{\bf t}^{\mb{n}}]\Phi_c({\bf t}) = \bigg(- \frac{3}{2m}E(\mb n) + \frac{3}{2}V(\mb n) + \sum_{k,\ell=1}^p \frac{(k-m)(\ell-m)}{2m(m+\ell)}n_{2k}n_{2\ell} \\
            + V\cdot\sum_{\ell=1}^{p}\frac{(\ell-m)(2m+\ell)}{2m(m+\ell)}n_{2\ell}- E(\mb n)\cdot \sum_{\ell=1}^{p}\frac{(\ell-m)(2\ell+3m+1)}{2m(m+1)(m+\ell)}n_{2\ell} \bigg)A(\mb{n}).
    \end{multline*}
    Replacing $E(\mb n) = \sum_{k=1}^{p}kn_{2k}$, $V=\sum_{k=1}^{p}n_{2k}$ in the above expression, and combining the sums, we obtain 
    that
    \begin{align*}
        &[{\bf t}^{\mb{n}}]\Phi_b({\bf t}) + [{\bf t}^{\mb{n}}]\Phi_c({\bf t}) = \bigg(- \frac{3}{2m}E(\mb n) + \frac{3}{2}V(\mb n) -\sum_{k,\ell=1}^p \frac{(2k-m-1)(\ell-m)}{2(m+1)m} n_{2k}n_{2\ell}\bigg)A(\mb{n})\\
        &=\bigg(- \frac{3}{2m}E(\mb n) + \frac{3}{2}V(\mb n) -\frac{1}{2}\left(\sum_{k=1}^p n_{2k}\right)^2  \\
        & \hfill +\left(\frac{1}{2m} + \frac{1}{m+1}\right)\left(\sum_{k=1}^p n_{2k}\right)\left(\sum_{\ell=1}^p \ell n_{2\ell}\right)- \frac{1}{m(m+1)}\left(\sum_{k=1}^p kn_{2k}\right)^2\bigg)A(\mb{n})\\
        &=\bigg(- \frac{3}{2m}E(\mb n) + \frac{3}{2}V(\mb n)  - \frac{1}{2}V^2(\mb n) + \frac{3m+1}{2m(m+1)}E(\mb n) V(\mb n) - \frac{1}{m(m+1)}E^2(\mb n)\bigg)A(\mb{n})\\
        &= A(\mb{n})- [{\bf t}^{\mb{n}}]\Phi_a({\bf t}).
    \end{align*}
This proves \eqref{eq:goal} and concludes the proof of formula \eqref{eq:count-g-0}.

\subsection{Proof of \texorpdfstring{Theorem \ref{thm:g-1}}{the formula for genus 1}}

\begin{lemma}
    For each $m=1, 2,...,p$, we have that
        \begin{equation}
            \frac{\partial F_1}{\partial t_{2m}} = \frac{1}{24 P_\sigma^2(\sigma(\mb t); \mb t)}\binom{2m}{m}\sigma^{m-1}\left((m-1)P_{\sigma}(\sigma(\mb t); \mb t) -\sigma(\mb t) P_{\sigma \sigma}(\sigma(\mb t); \mb t)\right)
        \end{equation}
        \label{lemma:F1-derivative}
\end{lemma}
\begin{proof}
    Recall that, by \eqref{eq:dF-1}, we have 
    \[
        \dpd{F_1}{t_{2m}} =  -\frac{\sigma(\mb t)}{12P^2_\sigma(\sigma(\mb t); \mb t)} \frac{1}{m } \res_{z=\infty} \left(z^{2m} \cdot \frac{ z^2 P_{\sigma \sigma }(\sigma(\mb t); \mb t) - \left(6 P_\sigma(\sigma(\mb t); \mb t) + 4 \sigma(\mb t) P_{\sigma \sigma}(\sigma(\mb t); \mb t) \right) }{(z^2-4\sigma(\mb t))^{5/2}}  \right).
    \]
    To calculate the residue, note that for $z$ large enough,
        \begin{equation}
            \frac{1}{(z^2-4\sigma)^{5/2}} = \sum_{\ell=0}^{\infty} \frac{(2\ell+3)!}{6\cdot \ell!(\ell+1)!}\frac{\sigma^\ell}{z^{2\ell+5}},
        \end{equation}
    and so
        \begin{align*}
            \res_{z = \infty} z^{2m}\frac{Az^2 + B\sigma}{(z^2-4\sigma)^{5/2}} &= A\frac{m(2m+1)}{6}\binom{2m}{m}\sigma^{m-1}+B \frac{m(m-1)}{12}\binom{2m}{m}\sigma^{m-1}\\
            &=\frac{m}{12}\binom{2m}{m}\sigma^{m-1}\bigg((4A+B)m+(2A-B)\bigg).
        \end{align*}
    To arrive at the result, we set 
    \[
        A = -\frac{\sigma P_{\sigma \sigma }}{12m P_\sigma^2} \qandq B = \frac{3 P_\sigma + 2 \sigma P_{\sigma \sigma }}{6m P_\sigma^2}.
    \] 
    
\end{proof}
\begin{proposition}
    For any $m=1, 2,...,p$,
         \begin{equation} F_1({\bf t}) = -\frac{1}{12}\log \left(\sigma(\mb t) P_\sigma(\sigma(\mb t); \mb t)\right) = -\frac{1}{12}\log\left(\sigma(\mb t) + \sum_{k=1}^p \binom{2k-1}{k}k t_{2k}\sigma^{k}(\mb t)\right) . \end{equation}
         \label{prop:F1-explicit}
\end{proposition}
\begin{proof}
    The proof is in the same spirit as that of Proposition \ref{prop:F0-explicit}. To begin, observe the identity 
    \begin{equation}
       -\dfrac{1}{12} \dpd{}{\sigma} \log (\sigma P_\sigma(\sigma; \mb t) - mP(\sigma; \mb t) + m) = \dfrac{1}{12} \dfrac{(m - 1) P_\sigma(\sigma; \mb t) - \sigma P_{\sigma \sigma}(\sigma; \mb t)}{\sigma P_\sigma(\sigma; \mb t) - mP(\sigma, \mb t) + m},
       \label{eq:intermediate-id-f1}
    \end{equation}
    where we note that the argument of the logarithm does not depend on $t_{2m}$ explicitly. Recall that identity \eqref{eq:sigma-derivative} was obtained by differentiating \eqref{eq:sigma-poly}:  
    \begin{equation}
        \dpd{\sigma}{t_{2m}} = \binom{2m-1}{m} \dfrac{\sigma^m}{P_\sigma}
        \label{eq:sigma-derivative-1}
    \end{equation}
    in which explicit dependence on $t_{2m}$ can be eliminated using \eqref{eq:sigma-poly}. Thus, it follows from Lemma \ref{lemma:F1-derivative} and \eqref{eq:intermediate-id-f1} that 
    \[
        \dpd{F_1}{t_{2m}} \left(\dpd{\sigma}{t_{2m}} \right)^{-1} = \dfrac{1}{12}  \dfrac{(m - 1) P_\sigma(\sigma; \mb t) - \sigma P_{\sigma \sigma}(\sigma; \mb t)}{\sigma P_\sigma(\sigma; \mb t) } = -\dfrac{1}{12} \dpd{}{\sigma} \log (\sigma P_\sigma(\sigma; \mb t) - mP(\sigma; \mb t) + m).
    \]
    Integrating with respect to $\sigma$ and evaluating at $\sigma = \sigma(\mb t)$ (recall that $P(\sigma(\mb t); \mb t) = 1$) yields the result up to a function $c(\mb t\setminus \{t_{2m}\})$. However, after restoring the dependence on $t_{2m}$ in the right hand side of \eqref{prop:F1-explicit}, we note that this expression is independent of $m$, and thus $c(\mb t\setminus \{t_{2m}\}) \equiv const.$. Finally, using $F_1(\mb 0) = 0$ finishes the proof.
\end{proof}

Proposition \ref{prop:F1-explicit} already yields a version of Theorem \ref{thm:g-1}. 

\begin{proposition}
Set $c_{2k}:=k\binom{2k}{k}$ and recall \eqref{eq:graph-notation}. Then:
\begin{multline}
    \mathcal{N}_1(n_2,n_4,...,n_{2p}) =  \frac{1}{12}\left(\prod_{k=1}^{p}n_{2k}!c_{2k}^{n_{2k}}\right)\\
    \times \sum_{r=1}^{V}\frac{1}{r}\sum_{\substack{\mb{k}_i:=(k_{i_2},\cdots,k_{i_{2p}})\in \Z^{p}_+\setminus\{\mb{0}\}\\ \sum_{i_{2\ell}=1}^r k_{i_{2\ell}}= n_{2\ell}}} \prod_{i=1}^r\left(Z(\mb k_i)-E(\mb k_i))\frac{(E(\mb k_i)-1)!}{(E(\mb k_i)-V(\mb k_i) +1)!}\frac{1}{\prod_{\ell=1}^p k_{i_{2\ell}}!}\right).
    \label{eq:n-1-general-even}
\end{multline}
\label{thm:g-1-prelim}
\end{proposition}
\begin{proof}
We start by writing
\[
    F_1(\mb t) = \log (1 - (1-\sigma(\mb t) P_\sigma(\sigma(\mb t); \mb t))),
\]
Since $\sigma(\mb 0) = P_\sigma(\mb 0; \mb 0) = 1$, $F_1(\mb t)$ has an expansion in a small enough neighborhood of $\mb t = \mb 0$ as 
\begin{equation}
    F_1(\mb t) = \frac{1}{12} \sum_{r = 1}^\infty \dfrac{1}{r} (1-\sigma(\mb t) P_\sigma(\sigma(\mb t); \mb t))^r.
    \label{eq:F1-taylor-1}
\end{equation}

Using the definition of $P(\sigma; \mb t)$, \eqref{eq:sigma-def}, and Lemma \ref{Lagrange-Inversion-Lemma}, we have 
\begin{equation}
    1-\sigma(\mb t) P_\sigma(\sigma(\mb t); \mb t) = \sum_{\mb{n}\in\Z^{p}_+\setminus \{\mb{0}\}} \left(\prod_{k=1}^{p}c_{2k}^{n_{2k}}\right)\left(Z(\mb n) - E(\mb n)\right)\frac{(E(\mb n)-1)!}{(E(\mb n)-V(\mb n)+1)!} \dfrac{ \prod_{k=1}^p\left({-t_{2k}}/{2k}\right)^{n_{2k}}}{\prod_{k=1}^p n_{2k}!}.
\end{equation}
Expanding the powers in \eqref{eq:F1-taylor-1} as Cauchy products yields the result.
\end{proof}
Proposition \ref{thm:g-1-prelim} reduces to known results on the number of genus 1 regular maps that appear in \cite{ELT}; this calculation is carried out in Appendix \ref{appendix:reduction}. We are now ready to prove Theorem \ref{thm:g-1}.

\begin{proof}[Proof of Theorem \ref{thm:g-1}]
Begin by noting 
\[
    \log \sigma(\mb t) P_\sigma(\sigma(\mb t); \mb t) = \log \sigma(\mb t) + \log P_\sigma(\sigma(\mb t); \mb t). 
\]
The proof now relies on finding the Taylor expansion of both quantities. Recalling that 
\[
    \sigma(\mb t) = \sigma(\mb t; 1),
\]
where the right hand side was defined in \eqref{eq:sigma-lambda-poly}. In fact, one can relate $\sigma(\mb t)$ and $\sigma(\mb t; \varkappa)$ directly as (cf. Section \ref{sec:t'hooft-parameter}) 
\begin{equation}
    \sigma(\mb t; \varkappa) = \varkappa\sigma(\mb t_{\varkappa}), \qquad  \mb t_\varkappa = (t_2, \varkappa t_4, \varkappa^{2} t_6, \cdots, \varkappa^{p - 1} t_{2p}).
    \label{eq:sigma-rescale}
\end{equation}
In particular, $\sigma(\mb t; 0) \equiv 0$. Plugging \eqref{eq:sigma-rescale} into \eqref{eq:sigma-def} and using \eqref{eq:eq:Chu-number}, we find 
\begin{equation}
    \sigma(\mb t; \varkappa) = \sum_{\mb n \geq 0} \dfrac{E(\mb n)!}{(E(\mb n) - V(\mb n) + 1)!}  \varkappa^{E(\mb n) - V(\mb n) + 1} \prod_{k = 1}^p  c_{2k}^{n_{2k}}  \dfrac{(-t_{2k}/2k)^{n_{2k}}}{n_{2k}!}, \qquad c_{2k} = k \binom{2k}{k}.
    \label{eq:sigma-kappa-expansion}
\end{equation}
Differentiating \eqref{eq:sigma-lambda-poly}, we find
\begin{equation}
    \dpd{\sigma}{\varkappa} = \dfrac{1}{P_\sigma(\sigma(\mb t; \varkappa); \mb t)}, \qquad \dpd[2]{\sigma}{\varkappa} = -\dfrac{P_{\sigma \sigma}(\sigma(\mb t; \varkappa); \mb t)}{P^3_{\sigma}(\sigma(\mb t; \varkappa); \mb t)}.
    \label{eq:sigma-kappa-derivative}
\end{equation}
Our task now is to arrive at a Taylor expansion of $\log P_\sigma(\sigma(\mb t); \mb t)$ which can be done using the following calculation: 
\begin{align*}
    \log P_\sigma(\sigma(\mb t); \mb t) &= \int_0^{\sigma(\mb t)} \dpd{}{s} \log P_\sigma(s; \mb t) \dd s + \log P_\sigma (0; \mb t)\\
    &= \int_0^{\sigma(\mb t)} \dfrac{P_{\sigma\sigma}(s; \mb t)}{P_\sigma(s; \mb t)} \dd s + \log (1 + t_2).
\end{align*}
Making the change of variables $s  \equiv s(\varkappa) = \sigma(\mb t; \varkappa)$, we have 
\[
    \log P_\sigma(\sigma(\mb t); \mb t) = \int_0^1 \dfrac{P_{\sigma \sigma}(\sigma(\mb t; \varkappa); \mb t)}{P_{\sigma}(\sigma(\mb t; \varkappa); \mb t)} \cdot \dpd{\sigma}{\varkappa} \dd \varkappa + \log(1 + t_2).
\]
Using formulas \eqref{eq:sigma-kappa-derivative}, we have 
\begin{align*}
    \log P_\sigma(\sigma(\mb t); \mb t) &= - \int_0^1 P_\sigma(\sigma(\mb t; \varkappa); \mb t) \dpd[2]{\sigma}{\varkappa} \dd \varkappa + \log(1 + t_2)\\
    &=  \int_0^1 (-1 + 1 - P_\sigma(\sigma(\mb t; \varkappa); \mb t) )\dpd[2]{\sigma}{\varkappa} \dd \varkappa + \log(1 + t_2)\\
    &= 1- \dfrac{1}{P_\sigma(\sigma(\mb t); \mb t)} + \int_0^1 (1 - P_\sigma(\sigma(\mb t; \varkappa); \mb t) )\dpd[2]{\sigma}{\varkappa} \dd \varkappa + \log(1 + t_2)\\
    &= 1- \dpd{\sigma}{\varkappa} \biggl|_{\varkappa = 1}+ \int_0^1 (1 - P_\sigma(\sigma(\mb t; \varkappa); \mb t) )\dpd[2]{\sigma}{\varkappa} \dd \varkappa + \log(1 + t_2).
\end{align*}
By differentiating \eqref{eq:sigma-kappa-expansion}, we find 
\begin{equation}
    \dpd[2]{\sigma}{\varkappa} = \sum_{\mb n > 0} (E(\mb n) - V(\mb n))\dfrac{E(\mb n)!}{(E(\mb n) - V(\mb n))!}  \varkappa^{E(\mb n) - V(\mb n) - 1} \prod_{k = 1}^p  c_{2k}^{n_{2k}}  \dfrac{(-t_{2k}/2k)^{n_{2k}}}{n_{2k}!}.
    \label{eq:sigma-derivative-2-expand}
\end{equation}
Furthermore, using Lemma \ref{Lagrange-Inversion-Lemma} and \eqref{eq:sigma-rescale}, we have 
\begin{equation}
    1 - P_\sigma(\sigma(\mb t; \varkappa); \mb t) = \sum_{\mb n > 0} \dfrac{Z(\mb n) - E(\mb n)}{E(\mb n) (E(\mb n) - 1)}\dfrac{E(\mb n)!}{(E(\mb n) - V(\mb n))!}  \varkappa^{E(\mb n) - V(\mb n) - 1} \prod_{k = 1}^p  c_{2k}^{n_{2k}}  \dfrac{(-t_{2k}/2k)^{n_{2k}}}{n_{2k}!}.
    \label{eq:P-sigma-kappa-expand}
\end{equation}
Using \eqref{eq:sigma-derivative-2-expand}, \eqref{eq:P-sigma-kappa-expand}, we find 
\begin{multline}
    \int_0^1 (1 - P_\sigma(\sigma(\mb t; \varkappa); \mb t) )\dpd[2]{\sigma}{\varkappa} \dd \varkappa  = \\
    \sum_{\mb n > 0} \dfrac{V(\mb n) !}{E(\mb n) - V(\mb n)}\left(\sum_{0 < \mb r < \mb n}  \dfrac{Z(\mb r) - E(\mb r)}{E(\mb r) (E(\mb r) - 1)}\dfrac{\dbinom{E(\mb r)}{V(\mb r)}\dbinom{E(\mb n -\mb r)}{V(\mb n - \mb r)}}{ \dbinom{V(\mb n)} {V(\mb r)}} \prod_{j = 1}^p \binom{n_{2j}}{r_{2j}} \right)\prod_{k = 1}^p   c_{2k}^{n_{2k}}  \dfrac{(-t_{2k}/2k)^{n_{2k}}}{n_{2k}!}.
\end{multline}
Finally, plugging in the expansion for $\log \sigma(\mb t)$ from Lemma \ref{Lagrange-Inversion-Lemma} and $\pd{\sigma}{\kappa}$ from \eqref{eq:sigma-kappa-expansion}, and noting that $\log(1 + t_{2})$ only depends on $t_2$, the result follows. 
\end{proof}

\section{Proof of Theorem \ref{main-theorem}} 
\label{sec:main-thm-pf}
In this section, we prove the main result of this work, Theorem \ref{main-theorem}. We briefly outline the strategy of the proof. 
    \begin{enumerate}[(a)]
        \item We first establish a proposition about the `singular' structure of the recurrence coefficients $r_g(\mb t, \varkappa)$. This is based on an analysis of the \textit{discrete string equation} \cite{MR835728, MR1705232}: 
            \begin{equation}\label{eq:discrete-string}
                \varkappa = R_{n,N}(\mb t ) + \sum_{k=1}^{p} t_{2k}\Phi_{2k}(\{R_{n,N}(\mb t )\}),
            \end{equation}
        which is satisfied by the recurrence coefficients $R_{n,N}(\mb t )$.
        The only input we need from the Riemann-Hilbert analysis is the already established fact that the recurrence coefficients $R_{n,N}(\mb t, \varkappa)$ admit a $1/N^2$ asymptotic expansion:
            \begin{equation}
                R_{n,N}(\mb t )\sim \sum_{g=0}^{\infty} \frac{r_g(\mb t; \varkappa)}{N^{2g}},
            \end{equation}
        and the explicit form of $r_1(\mb t, \varkappa)$, which we have computed in Proposition \ref{prop:recurrence-expansion}.
        
        \item Having established the structure of the recurrence coefficients, we specialize to $\varkappa = 1$ and use the program of analytic combinatorics in several variables (ACSV) \cite{PW} to compute the asymptotics of combinatorial generating functions such as $[{\bf t}^{\bf n}]r_g(\mb t; \varkappa)$, where ${\bf n} = V\boldsymbol{\alpha}$, and $V\to \infty$. Here, $\boldsymbol{\alpha}$ is as described in the statement of Theorem \ref{main-theorem}.
        
        \item By applying a version of an identity of Bleher and Its \cite[Theorem 3.4]{MR2187941}, we can relate (at finite ${\bf n}$) $[{\bf t}^{\bf n}]r_g({\bf t})$ to $[{\bf t}^{\bf n}]F_{g}({\bf t})$, which are essentially the constants $\mathcal{N}_g({\bf n})$. Thus, we obtain our result.
    \end{enumerate}
For the remainder of this section, we will omit the dependence of various quantities on $\varkappa, \mb t$ where there is no risk of confusion.
\subsection{Singular Structure of the Recurrence Coefficients} 
Our starting point is to analyze the discrete string equation \eqref{eq:discrete-string}.
The functions $\Phi_{2k}(\{R_{n,N}\})$ appearing in this equation are called \textit{Freud functions} \cite{GL}, and are polynomials in the recurrence coefficients. For instance,
    \begin{equation}
        \Phi_{4}(\{R_{n,N}\}) = R_{n,N}(R_{n+1,N}+R_{n,N}+R_{n-1,N}).
    \end{equation}
Freud functions have some combinatorial structure:
    \begin{lemma} (cf. \cite{GL}).
    For any fixed $q\geq1$, the Freud function $\Phi_{2q}(\{R_{n,N}\})$ has the following form:
    \begin{equation}
        \Phi_{2q}(\{R_{n,N}\}) = \sum_{j=1}^{\binom{2q-1}{q}}\prod_{s\in I_{j,q}} R_{n+s,N},
    \end{equation}
where $|I_{j,q}| = q$, $I_{j,q}\subset \left\{ -q+1,-q+2,\cdots,q-2,q-1\right\}$. Moreover, as $N\to \infty$, if we write
    \begin{equation}
        R_{n,N}\sim\sum_{g=0}^{\infty} \frac{r_g}{N^{2g}},
    \end{equation}
then
    \begin{equation}
        \Phi_{2q}(\{R_{n,N}\}) \sim \sum_{g=0}^{\infty}\frac{\phi_g^{(2q)}}{N^{2g}},
    \end{equation}
where $\phi_g^{(2q)}$ are polynomials in $r_g,r_{g-1},r_{g-1}',r_{g-2},r_{g-2}',r_{g-2}'',...,r_{0},r_0',r_0^{(2q)}$, where $' = \frac{\partial}{\partial \varkappa}$ denotes the derivative in the t'Hooft parameter $\varkappa$.
    \end{lemma}
Structural analysis of the discrete string equation is essentially equivalent to structural analysis of the Freud functions. The next lemma establishes the relevant properties of these functions that we will need in our study of the coefficients $r_g$.
\begin{lemma} \label{lemma:weight-rj}
    Define the following grading functionals on monomials $\left\{\prod_{m}r_{j_m}^{(\ell_m)}\right\}$ by setting
        \begin{align*}
            \rho\left(\prod_{m}r_{j_m}^{(\ell_m)}\right) &= 2\left(\sum_m j_m\right)+\left(\sum_m\ell_m\right),\\
            \psi\left(\prod_{m}r_{j_m}^{(\ell_m)}\right) &= \sum_m (1-\delta_{j_m,0}\delta_{\ell_m,0})\left(5j_m+2\ell_m-1\right).
        \end{align*}
    Fix $q\geq 1$. The Freud function $\Phi_{2q}(\{R_{n,N}\})$ admits an asymptotic expansion of the following form:
        \begin{equation}
            \Phi_{2q}(\{R_{n,N}\}) \sim \sum_{g=0}^{\infty} \frac{\phi^{(2q)}_g}{N^{2g}},
        \end{equation}
    where the functions $\phi^{(2q)}_g$ have the following properties:
        \begin{enumerate}[(a)]
            \item $\phi^{(2q)}_g$ is a homogeneous polynomial of degree $q$ in the variables $\{r_{k}^{(j)}\mid j=0,...,2g-2k,k=0,...,g\}$,
            and carries the grading
                \begin{equation}
                    \rho(\varphi) = 2g.
                \end{equation}
            \item For any $g\geq 2$,
                \begin{align}\label{psi-functional-expansion}
                    \phi^{(2q)}_g
                    &= q\binom{2q-1}{q}r_0^{q-1}\left(r_g + \frac{(q-1)}{6}r_{g-1}'' + \frac{(q-1)}{2}r_0^{-1}\left(\frac{1}{2}\sum_{\ell=1}^{g-1} r_{\ell} r_{g-\ell}\right)\right)+ M(r_0,r_1,\cdots,r_{g-1}),
                \end{align}
            where $M(r_0,r_1,\cdots,r_{g-1})$ is a differential polynomial in $r_0,r_1,\cdots,r_{g-1}$ in which each monomial $m$ contributing to $M$ satisfies
                \begin{align*}
                    \psi(m) <5g-2.
                \end{align*}
        \end{enumerate}
\end{lemma}
\begin{remark}
    Before proceeding to the proof of the lemma, we make a few heuristics which should motivate the functionals $\rho,\psi$ to the reader. When one expands the Freud function \sloppy
    $\Phi_{2q}(\{R_{n,N}\})$ as $N\to \infty$, it is clear that the coefficient $\phi_{g}^{(2q)}$ of $N^{-2g}$ will be the sum of many terms of the form $\left\{\prod_{m}r_{j_m}^{(\ell_m)}\right\}$. The functional $\rho$ determines which terms appear at order $N^{-2g}$; more precisely, $\phi_{g}^{(2q)}$ is a sum of terms
    of $\rho$-weight $2g$. Among these terms, some are more singular than others as we approach the `critical point' (in the language of ACSV, the unique minimal point contributing to the Cauchy integral). The functional $\psi$ measures how `singular' these terms are: the higher the $\psi$-weight, the more singular the contribution of a term to $\phi_{g}^{(2q)}$.
\end{remark}
\begin{proof}
    Recall that $\Phi_{2q}(\{R_{n,N}\})$ has the following form:
        \begin{equation}
            \Phi_{2q}(\{R_{n,N}\}) = \sum_{t=1}^{\binom{2q-1}{q}} \prod_{\ell=1}^{q} R_{n+s_{t,\ell},N},
        \end{equation}
    where the integers $s_{t,\ell}\in\{-q+1,...,q-1\}$. Furthermore, recall that
        \begin{equation}
            R_{n+s,N} \sim \sum_{j=0}^{\infty} \frac{1}{N^{2j}} \left(B_{j} + \frac{1}{N}C_{j}\right),
        \end{equation}
    where
        \begin{align}
            B_j &= \sum_{v=0}^j \frac{r_{j-v}^{(2v)}(\varkappa)}{(2v)!}s^{2v} = r_j + [\textit{terms involving $r_i$ with $i<j$}],\\
            C_j &= \sum_{v=0}^j \frac{r_{j-v}^{(2v+1)}(\varkappa)}{(2v+1)!}s^{2v+1}.
        \end{align}
    From the above expressions, it follows that (a) holds.
    To see that (b) also holds, we will calculate the coefficient of each of the explicit terms appearing in \eqref{psi-functional-expansion} (which are all easily seen to be of $\rho$-weight $2g$); then, we will show that the remaining terms in $M$ have $\psi$-weight strictly less than
    $5g-2$.
    
    \underline{\textit{The coefficient of the term $r_0^{q-1}r_g$.}} Let us focus on a single one of the $\binom{2q-1}{q}$ terms in the expression for $\Phi_{2q}(\{R_{n,N}\})$:
        \begin{align*}
            \prod_{\ell=1}^q R_{n+s_{t,\ell},N} &= \prod_{\ell=1}^q\left(\sum_{j=0}^{\infty} \frac{1}{N^{2j}} \left(B_{j} + \frac{1}{N}C_{j}\right)\right)\\
            &=\sum_{j_1,\cdots, j_{q}=0}^{\infty}\left(\frac{1}{N^{2J}} \prod_{\ell=1}^{q}\left(B_{j_{\ell}} + \frac{1}{N}C_{j_{\ell}}\right)\right),
        \end{align*}
    where $J=j_1+j_2+\cdots+j_{q}$. So, the contribution of this term at order $1/N^{2g}$ to the asymptotic expansion of 
    $\Phi_{2q}(\{R_{n,N}\})$ is
        \begin{equation}\label{N2g-rprod}
            [N^{-2g}]\left(\prod_{\ell=1}^q R_{n+s_{t,\ell},N}\right) = \left(\sum_{j_1+\cdots+j_q =g} \prod_{\ell=1}^{q}B_{j_{\ell}}\right) + \cdots + \left(\sum_{2j_1+\cdots+2j_q =2g-q}\prod_{\ell=1}^{q} C_{j_{\ell}}\right).
        \end{equation}
    Note that only terms in this expression that can contain a factor of $r_g$ appear in the first set of parentheses. A factor of $r_g$ only appears if one of the $j_{\ell}'s$ is equal to $g$, and the rest are zero. There are $q$ possible sets of indices that achieve this, and each contributes to the above term the amount
        \begin{equation*}
            \left(\prod_{\ell=1}^{q} B_{j_{\ell}}\right)\bigg|_{j_1=g,j_2=0,\cdots,j_q=0} = r_0^{q-1}r_g.
        \end{equation*}
    Thus, we see that 
        \begin{equation*}
            [N^{-2g}]\left(\prod_{\ell=1}^q R_{n+s_{t,\ell},N}\right) = qr_0^{q-1}r_g + [\textit{terms independent of }r_g].
        \end{equation*}
    Since the Freud function is a sum of $\binom{2q-1}{q}$ such terms, we see that
        \begin{equation*}
             \phi_{g}^{(2q)} = [N^{-2g}]\Phi_{2q} = q\binom{2q-1}{q}r_0^{q-1}r_g + [\textit{terms independent of }r_g].
        \end{equation*}
        
    \underline{\textit{The coefficient of the term $r_0^{q-2}r_{\ell}r_{g-\ell}$.}} Now, note that, for given $\ell$, by taking $j_1 =\ell, j_2 = g-\ell$, and $j_3=j_4=\cdots=j_q = 0$,
        \begin{equation*}
            \left(\prod_{\ell=1}^{q} B_{j_{\ell}}\right)\bigg|_{j_1 =\ell, j_2 = g-\ell,j_3=0,\cdots,j_q=0} = r_0^{q-2}r_{\ell}r_{g-\ell}.
        \end{equation*}
    Including this choice of indices, there are $q(q-1)$ possible choices of indices which achieve the above form.  Since the Freud function is a sum of $\binom{2q-1}{q}$ such terms, we see that
        \begin{equation*}
             \phi_{g}^{(2q)} = q\binom{2q-1}{q}r_0^{q-1}r_g + \frac{q(q-1)}{2}\binom{2q-1}{q}r_0^{q-2}\sum_{\ell=1}^{g-1}r_{\ell}r_{g-\ell} +\cdots,
        \end{equation*}
    where the factor of $\frac{1}{2}$ comes from the correction due to double counting (note the symmetry of the sum upon exchanging $\ell\to g-\ell$).
    
    \underline{\textit{The coefficient of the term $r_0^{q-1}r_{g-1}''$.}} We have that
    \begin{equation*}
        [N^{-2g}]\left(\prod_{\ell=1}^q R_{n+s_{t,\ell},N}\right) = qr_0^{q-1}r_g + \frac{q(q-1)}{2}r_0^{q-2}\sum_{\ell=1}^{g-1}r_{\ell}r_{g-\ell}  + c_1r_{g-1}''r_0^{q-1} + \tilde{M}(r_0,...,r_{g-1}),
    \end{equation*}
    where $c_1 =\frac{1}{2}\sum_{\ell=1}^q s_{t,\ell}^2$.
    If we sum up the $\binom{2q-1}{q}$ contributions defining the Freud function, we see that
    \begin{multline}
            \phi_{g}^{(2q)} = [N^{-2g}]\Phi_{2q} \\= q\binom{2q-1}{q}r_0^{q-1}r_g + C_1r_0^{q-1}r_{g-1}'' + \frac{q(q-1)}{2}\binom{2q-1}{q}r_0^{q-2}\sum_{\ell=1}^{g-1}r_{\ell}r_{g-\ell} + M(r_0,...,r_{g-1}),
    \end{multline}
    where
        \begin{equation}\label{C1-def}
            C_1 = \frac{1}{2}\sum_{t=1}^{\binom{2q-1}{q}}\sum_{\ell=1}^q s_{t,\ell}^2. 
        \end{equation}
    Our goal is to determine $C_1$ explicitly. Observe that, at order $N^{-2}$,
    \begin{equation*}
        [N^{-2}]\left(\prod_{\ell=1}^q R_{n+s_{t,\ell},N}\right) = q\binom{2q-1}{q}r_0^{q-1}r_1+ c_1r_0^{q-1}r_0'' + c_2 r_0^{q-2}(r_0')^2,
    \end{equation*}
    where the constant $c_1$ is the same as above. Again summing the $\binom{2q-1}{q}$ contributions defining the Freud function,
    \begin{equation*}
        \phi_{1}^{(2q)} = q\binom{2q-1}{q}r_0^{q-1}r_1 + C_1r_0^{q-1}r_{0}'' + C_2r_0^{q-2}(r_0')^2,
    \end{equation*}
    where $C_1$ is as in \eqref{C1-def}, and $C_2$ is some other constant. We now appeal to the fact that the equation
        \begin{equation*}
            \varkappa = R_{n,N} + \sum_{k=1}^{p} t_{2k}\Phi_{2k}(\{R_{n,N}\})
        \end{equation*}
    must be satisfied. Using the above, one finds that, at order $N^{-2}$, if we set $t_{2k} = 0$ for $k\neq q$,
        \begin{equation*}
            0 = \bigg(r_1+t_{2q}q\binom{2q-1}{q}r_0^{q-1}r_1 + t_{2q} C_1r_{0}'' + t_{2q} C_2(r_0')^2\bigg)\bigg|_{t_{2k}=0,k\neq q},
        \end{equation*}
    which, after rearranging, implies that
        \begin{equation*}
            r_1\bigg|_{t_{2k}=0,k\neq q} = - \frac{C_1r_0^{q-1}r_{0}'' + C_2r_0^{q-2}(r_0')^2}{P_{\sigma}}\bigg|_{t_{2k}=0,k\neq q}.
        \end{equation*}
    On the other hand, by Proposition \ref{prop:recurrence-expansion}, we have that
        \begin{equation*}
            r_1\bigg|_{t_{2k}=0,k\neq q} = -\frac{\frac{1}{6}q(q-1)\binom{2q-1}{q} t_{2q} r_0^{q-1}r_0'' + \frac{1}{12}\binom{2q-1}{q}q(q-1)(q-2)t_{2q} r_0^{q-2} (r_0')^2}{P_{\sigma}}\bigg|_{t_{2k}=0,k\neq q};
        \end{equation*}
    this determines the constants $C_1,C_2$ to be
        \begin{equation}
            C_1 = \frac{1}{6}q(q-1)\binom{2q-1}{q}, \qquad\qquad C_2 = \frac{1}{12}q(q-1)\binom{2q-1}{q}.
        \end{equation}
    This determines the coefficient of $r_{g-1}''$ in \eqref{psi-functional-expansion}.

    \underline{\textit{Calculation of the $\psi$-weight of the monomials in $M$.}} It is easy to see that the terms we have studied (terms of the form $r_0^{q-1}r_g, r_0^{q-1}r_{g-1}''$, and $r_0^{q-2}r_{\ell}r_{g-\ell}$) are of $\psi$-weight greater than or equal to\footnote{More precisely, all terms except $r_0^{q-1}r_g$ have $\psi$-weight $5g-2$, and $\psi(r_0^{q-1}r_g) = 5g-1$.} $5g-2$. It remains to see that all other monomials contributing to $\phi_{g}^{(2q)}$ are of $\psi$-weight strictly less that $5g-2$. Observe that, for any monomial $\prod_m r_{j_m}^{(\ell_m)}$,
        \begin{equation*}
            \psi\left(\prod_m r_{j_m}^{(\ell_m)}\right) = 5\sum_{m}j_m + 2\sum_{m}\ell_m - L,
        \end{equation*}
    where $L$ counts the number of terms in $\prod_m r_{j_m}^{(\ell_m)}$ which are not identically $r_0$. Now, if $\prod_m r_{j_m}^{(\ell_m)}$ is one of the terms
    appearing in $[N^{-2g}]\left(\prod_{\ell=1}^q R_{n+s_{t,\ell},N}\right)$, then
        \begin{equation*}
            2g = \rho\left(\prod_m r_{j_m}^{(\ell_m)}\right) = 2\sum_{m}j_m + \sum_{m}\ell_m \Longleftrightarrow \sum_{m}\ell_m = 2g-2\sum_{m}j_m,
        \end{equation*}
    and so we must have that
        \begin{equation}\label{psi-formula}
            \psi\left(\prod_m r_{j_m}^{(\ell_m)}\right) = 5\sum_{m}j_m + 2\left(2g-2\sum_{m}j_m\right) - L = 4g + \sum_{m}j_m - L.
        \end{equation}
    On the other hand, if $2g = \rho\left(\prod_m r_{j_m}^{(\ell_m)}\right)$, we also have the inequality
        \begin{equation*}
        \sum_{m}j_m \leq \sum_{m}j_m + \frac{1}{2}\sum_{m}\ell_m  = g.
        \end{equation*}
    Combining this inequality with \eqref{psi-formula}, we see that
        \begin{equation*}
            \psi\left(\prod_m r_{j_m}^{(\ell_m)}\right) \leq 5g - L.
        \end{equation*}
    For fixed $g\geq 2$, $L = 0$ is not possible, since the only monomial with $L=0$ is $r_0^q$, which has $\rho$-weight $0$ and thus does not contribute at order $N^{-2g}$. For $L = 1$, the only possible monomials of $\rho$-weight $2g$ are $r_{g-\ell}^{(2\ell)}r_0^{q-1}$, which have $\psi$-weight
        \begin{equation*}
            \psi\left(r_{g-\ell}^{(2\ell)}r_0^{q-1}\right) = 5g-\ell-1,
        \end{equation*}
    so only $\ell=0,1$ have $\psi$-weight greater than $5g-2$, and these are terms we have already studied. For $L=2$, a moments thought reveals that the only terms of $\psi$-weight greater than $5g-2$ are of the form $r_0^{q-2}r_{\ell}r_{g-\ell}$.
    This completes the proof of the Proposition.
\end{proof}

\begin{remark}
    In our calculation of the coefficient of $r_0^{q-1}r_{g-1}''$, we relied on the fact that $R_{n,N}$ solve the discrete string equation. We remark that this method is admittedly very brute force; there is likely a more intrinsic way to calculate this coefficient by analyzing more carefully the Freud weights.
\end{remark}
We also need to describe how terms in the above sum change under differentiation with respect to $\varkappa$. This is accomplished in the following Lemma:
\begin{lemma}\label{general-derivative-lemma}
    Let $Q := Q(r_0;{\bf t})$ be a polynomial, and $\Delta\geq 1$ a fixed positive integer. Then, for any $k\geq 0$,
    \begin{equation}
        \frac{\partial^k}{\partial \varkappa^k}\frac{Q(r_0;{\bf t})}{P_{\sigma}(r_0;{\bf t})^{\Delta}} = \frac{\tilde{Q}(r_0;{\bf t})}{P_{\sigma}(r_0;{\bf t})^{\Delta + 2k}},
    \end{equation}
    where $\tilde{Q}(r_0;{\bf t})$ is some other polynomial in $r_0,{\bf t}$.
\end{lemma}
\begin{proof}
    This follows easily from an induction argument, \eqref{eq:r0-r1-lambda}, and \eqref{eq:sigma-kappa-derivative}. 
\end{proof}
As an immediate corollary of the previous two lemmas, we obtain that:
\begin{corollary}\label{string-corollary}
    For $g\geq 2$, at order $N^{-2g}$, the string equation reads
        \begin{equation}\label{string-corollary-eq}
            0 = P_{\sigma} r_g + \frac{1}{2}P_{\sigma\sigma} \left(\sum_{\ell=1}^{g-1}r_{\ell}r_{g-\ell}+ \frac{1}{3}r_{g-1}''r_0\right) + H(r_0,r_1,...,r_{g-1}),
        \end{equation}
    where $H$ is a differential polynomial in $r_0,r_1,...,r_{g-1}$ such that each monomial $h$ contributing to $H$ satisfies
        \begin{equation}
            \psi(h) \leq 5g-3.
        \end{equation}
    Furthermore, for each $g\geq 1$, there exists a polynomial $L_g(r_0;{\bf t})$ such that,
        \begin{equation}\label{r_g-structure}
            r_{g} = \frac{L_g(r_0;{\bf t})}{P_{\sigma}^{5g-1}(r_0;{\bf t})}.
        \end{equation}
\end{corollary}
\begin{proof}
    The first identity follows from applying Lemma \ref{lemma:weight-rj}(b) to the string equation, as well as the definition of $P_{\sigma}(r_0;{\bf t})$, $P_{\sigma\sigma}(r_0;{\bf t})$. 
    
    We prove that $r_{g}$ has the form \eqref{r_g-structure} inductively. The base case of the induction is established by Proposition \ref{prop:recurrence-expansion}. Assuming the inductive hypothesis, and by applying Lemma \ref{general-derivative-lemma}, it follows that for any $k \geq 0$, and any $1\leq \ell \leq g-1$,
        \begin{equation*}
            \frac{\partial^k}{\partial \varkappa^{k}}r_{\ell}(\varkappa) = \frac{Q_{k,\ell}(r_0;{\bf t})}{P_{\sigma}(r_0;{\bf t})^{5\ell+2k-1}}.
        \end{equation*}
     From this formula, it is immediate that, if we rewrite the monomial $m = m(r_0,r_1,\cdots r_{g-1})$ coming from Equation \eqref{psi-functional-expansion} in terms of $r_0$,  
        \begin{equation*}
            m = \frac{f(r_0;{\bf t})}{P_{\sigma}(r_0;{\bf t})^{\psi(m)}}.
        \end{equation*}
     
     Thus, one sees that the terms which contribute the most factors of $P_{\sigma}(r_0,{\bf t})^{-1}$ arise from monomials $m$ with the largest 
     $\psi$-weight, i.e. monomials such that $\psi(m) = 5g-2$. 
\end{proof}
\begin{remark}
    Observe that a similar formula in the regular, even-degree case was derived in \cite{MR2754408} (see also the recent work \cite{GL}). We emphasize here the importance of the power of $5g-1$ in the denominator.
\end{remark}
We will make a slightly more precise statement of this Lemma in the next section, but the general form of $r_g$ we have arrived at here is enough to proceed to the methods of \cite{PW}.

\subsection{Analytic Combinatorics in Several Variables}
With this structural lemma in place, we are in a position to apply the techniques of the theory of analytic combinatorics in several variables (ACSV) developed by Pemantle and Wilson \cite{PW1,PW2} (see also the book \cite{PW}). Observe that all of the functions $r_g({\bf t})$ are of the form
    \begin{equation}
        \frac{Q(\sigma({\bf t});{\bf t})}{P_{\sigma}(\sigma({\bf t});{\bf t})^{\Delta}},
    \end{equation}
where $Q$ is a polynomial, and $\Delta$ is a positive integer. In the interest of generality, we will study Cauchy-type integrals of the form
    \begin{equation}
        I_V := I_V(Q,\boldsymbol{\alpha},\Delta) = \frac{C(\boldsymbol{\alpha})}{(2\pi \ii)^{p}}\oint_{\mathbb{T}_{\delta}} \frac{Q(\sigma({\bf t});{\bf t})}{P_{\sigma}(\sigma({\bf t});{\bf t})^{\Delta}} \ee^{V\Phi({\bf t})} \prod_{k=1}^{p}\frac{\dd t_{2k}}{t_{2k}},
        \label{eq:def-IV-INTEGRAL}
    \end{equation}
where: 
    \begin{itemize}
        \item \begin{equation*}C(\boldsymbol{\alpha}) := \frac{1}{V!}\prod_{k=1}^p(-2k)^{\alpha_{2k}V}(\alpha_{2k}V)!,\end{equation*}
        \item $\Phi({\bf t}) = -\sum_{k=1}^{p}\alpha_{2k} \log t_{2k}$,
        \item $\mathbb{T}_{\delta} := \{{\bf t} \mid |t_{2k}|=\delta\, ,k=1,...,p\}$, where $\delta>0$ is sufficiently small (observe that $I_V$ is independent of $\delta$),
        \item $\Delta\geq 1$ is a positive integer,
        \item $Q(\sigma({\bf t});{\bf t})$ is a polynomial in $\sigma({\bf t}), {\bf t}$, and the coefficients $f({\bf n}) = f({\bf n};Q,\Delta)$ of the series
            \begin{equation}\label{eq:combinatorial-generating-function}
                \frac{Q(\sigma({\bf t});{\bf t})}{P_{\sigma}(\sigma({\bf t});{\bf t})^{\Delta}} = \sum_{\boldsymbol{n}\geq 0} f({\bf n})\prod_{k=1}^p \left(-\frac{t_{2k}}{2k}\right)^{n_{2k}} \frac{1}{n_{2k}!}
            \end{equation}
        are non-negative $f({\bf n})\geq 0$ for all but finitely many ${\mb n}$. 
    \end{itemize}
Observe that the last property is the most restrictive. However, this property comes for free from the fact that we will work exclusively with \textit{combinatorial} generating functions, i.e. functions where $f({\bf n})$ in \eqref{eq:combinatorial-generating-function} count some combinatorial object.
For example:
    \begin{itemize}
        \item $r_g({\bf t})$ is the generating function for the number of connected, labeled, $2$-legged, genus $g$ maps (see \cite{MR2367197}, \S6 for the definition of $2$-legged),
        \item $F_g({\bf t})$ counts the number of connected, labeled, genus $g$ maps,
        \item $\mathfrak{Q}(\mathfrak{Q}+1)F_g({\bf t})$ counts the number of connected, labeled, genus $g$ maps multiplied by $E(E+1)$, where $E$ is the number of edges (see Lemma \eqref{Lemma:Bleher-Its} below).
    \end{itemize}
The basic principle of ACSV is that the singularity set $\mathcal{V}$ of the integrand (here, the places where $P_{\sigma}(\sigma({\bf t}),{\bf t}) = 0$) control the asymptotics of $I_V$. Before studying the $V\to \infty$ asymptotics of $I_V$, it will be useful to first study an associated \textit{critical surface}.

\subsubsection{ACSV: Geometry of the critical surface} 
    Recall that
        \begin{equation}
           - \varkappa + P(\sigma;{\bf t}) = -\varkappa + \sigma + \sum_{k=1}^{p}\binom{2k-1}{k}t_{2k}\sigma^k \qandq P_{\sigma}(\sigma;{\bf t}) = 1 + \sum_{k=1}^{p}k\binom{2k-1}{k}t_{2k}\sigma^{k-1}.
        \end{equation}
    Let $\{\beta_{2k}\}_{k=1}^{p}$ be a collection of strictly positive real numbers satisfying $\sum_{k=1}^{p}\beta_{2k} = 1$. We define the derived quantities $\mathfrak{e},\mathfrak{z}$ by 
    \begin{equation}
        \mathfrak{e} = \mathfrak{e}(\boldsymbol{\beta}) := \sum_{k=1}^{p}k\beta_{2k} >1, \qquad\qquad  \mathfrak{z} = \mathfrak{z}(\boldsymbol{\beta}):=\sum_{k=1}^{p}k^2\beta_{2k} >\mathfrak{e}.
    \end{equation}
    Observe that the stated inequalities are trivial to prove.
We will define the critical surface in terms of a collection of auxiliary parameters $\{\beta_{2k}\}_{k=1}^{p}$, which one should think of as \textit{the proportion of vertices which have degree $2k$}.

\begin{definition}
    Let $\boldsymbol{\beta} :=\left( \beta_{2},\cdots,\beta_{2p}\right)$, and 
        \begin{equation}\label{t2k}
            t_{2k}(\boldsymbol{\beta}) := -\frac{\beta_{2k}}{\binom{2k-1}{k} \mathfrak{e}}\left(\frac{\mathfrak{e}-1}{\mathfrak{e}\varkappa}\right)^{k-1},
        \end{equation}
${\bf t}(\boldsymbol{\beta}) := \left( t_{2}(\boldsymbol{\beta}),\cdots, t_{2p}(\boldsymbol{\beta})\right)$. We define the hypersurface
        \begin{equation}\label{eq:critical-surface}
            \Sigma_{cr} := \left\{ {\bf t}(\boldsymbol{\beta}) \bigg| \beta_{2k}> 0,\quad \sum_{k=1}^{p}\beta_{2k} = 1\right\} \subset \mathbb{R}^{p}.
        \end{equation}
Observe that $t_{2k}(\boldsymbol{\beta})<0$.
\end{definition}

\begin{lemma}
    At the point ${\bf t}(\boldsymbol{\beta})\in \Sigma_{cr}$, the normal vector is given by
        \begin{equation}\label{vec nu}
            \boldsymbol{\nu}(\boldsymbol{\beta}) := \left(\nu_{2},\nu_{4},\cdots,\nu_{2p}\right),
        \end{equation}
    where
        \begin{equation}\label{nu2k}
            \nu_{2k}(\boldsymbol{\beta}) = \binom{2k-1}{k}\left(\frac{\mathfrak{e} \varkappa}{\mathfrak{e} -1}\right)^{k-1}.
        \end{equation}
\end{lemma}
\begin{proof}
    Observe that
        \begin{align*}
            \boldsymbol{\nu}(\boldsymbol{\beta})\cdot {\bf t}(\boldsymbol{\beta}) = -\frac{1}{\mathfrak{e}}\sum_{k=1}^{p}\beta_{2k} = -\frac{1}{\mathfrak{e}}.
        \end{align*}
    On the other hand, since for any fixed $k=1,...,p$,
        \begin{equation*}
            \frac{\partial}{\partial \beta_{2k}} \mathfrak{e} = k,
        \end{equation*}
    one can show that
        \begin{equation*}
            -\frac{\partial}{\partial \beta_{2k}}\frac{1}{\mathfrak{e}} = \frac{k}{\mathfrak{e}^2} = {\bf t}(\boldsymbol{\beta}) \cdot \frac{\partial}{\partial \beta_{2k}} \boldsymbol{\nu}(\boldsymbol{\beta}),
        \end{equation*}
    so that
        \begin{equation*}
            0 = \frac{\partial}{\partial \beta_{2k}}\left[\boldsymbol{\nu}(\boldsymbol{\beta})\cdot {\bf t}(\boldsymbol{\beta})\right]-\frac{k}{\mathfrak{e}^2} = {\bf t}(\boldsymbol{\beta}) \cdot \frac{\partial}{\partial \beta_{2k}} \boldsymbol{\nu}(\boldsymbol{\beta}) + \boldsymbol{\nu}(\boldsymbol{\beta}) \cdot \frac{\partial}{\partial \beta_{2k}} {\bf t}(\boldsymbol{\beta}) - \frac{k}{\mathfrak{e}^2} = \boldsymbol{\nu}(\boldsymbol{\beta}) \cdot \frac{\partial}{\partial \beta_{2k}} {\bf t}(\boldsymbol{\beta}),
        \end{equation*}
    which implies that $\boldsymbol{\nu}(\boldsymbol{\beta})$ is orthogonal to the vectors
        \begin{equation*}
            \left\{\frac{\partial}{\partial \beta_{2k}}{\bf t}(\boldsymbol{\beta})\bigg| k=1,...,p\right\}.
        \end{equation*}
    Since these vectors span the tangent space of the $\Sigma_{cr}$ at ${\bf t}(\boldsymbol{\beta})$, it follows that $\boldsymbol{\nu}(\boldsymbol{\beta})$ is indeed the normal to the surface there.
\end{proof}

\begin{proposition}
    $\Sigma_{cr}$ is a subset of the algebraic variety defined by the resultant of $P_{\sigma}$ and $P$, i.e. the discriminant of $P$ in $\sigma$. Equivalently, if we denote this discriminant by $\xi({\bf t})$, and ${\bf t}(\boldsymbol{\beta})\in \Sigma_{cr}$, then
        \begin{equation}
            \xi({\bf t}(\boldsymbol{\beta})) = 0.
        \end{equation}
\end{proposition}
\begin{proof}
    We have only to show that there exists $\sigma(\boldsymbol{\beta})$ such that $P(\sigma(\boldsymbol{\beta});{\bf t}(\boldsymbol{\beta})) = P_{\sigma}(\sigma(\boldsymbol{\beta});{\bf t}(\boldsymbol{\beta})) = 0$, for every ${\bf t}(\boldsymbol{\beta}) \in \Sigma_{cr}$. Indeed, take
        \begin{equation}
    \sigma(\boldsymbol{\beta}) := \frac{\mathfrak{e} \varkappa}{\mathfrak{e}-1}.
        \end{equation}
        Then:
    \begin{align*}
        P(\sigma(\boldsymbol{\beta});{\bf t}(\boldsymbol{\beta})) &= -\varkappa + \frac{\mathfrak{e} \varkappa}{\mathfrak{e}-1} - \sum_{k=1}^{p}\binom{2k-1}{k} \frac{\beta_{2k}}{\binom{2k-1}{k} \mathfrak{e}}\left(\frac{\mathfrak{e}-1}{\mathfrak{e}\varkappa}\right)^{k-1}\left(\frac{\mathfrak{e} \varkappa}{\mathfrak{e}-1}\right)^{k}\\
        &= -\varkappa + \frac{\mathfrak{e} \varkappa}{\mathfrak{e}-1} -\frac{\varkappa}{\mathfrak{e}-1}\sum_{k=1}^{p}\beta_{2k}\\
        &= -\varkappa + \frac{\mathfrak{e}\varkappa }{\mathfrak{e}-1} -\frac{\varkappa}{\mathfrak{e}-1} = 0,
    \end{align*}
and 
    \begin{align*}
        P_{\sigma}(\sigma(\boldsymbol{\beta});{\bf t}(\boldsymbol{\beta})) &= 1 - \sum_{k=1}^{p}k\binom{2k-1}{k} \frac{\beta_{2k}}{\binom{2k-1}{k} \mathfrak{e}}\left(\frac{\mathfrak{e}-1}{\mathfrak{e}\varkappa}\right)^{k-1}\left(\frac{\mathfrak{e} \varkappa}{\mathfrak{e}-1}\right)^{k-1}\\
        &=1 -\frac{1}{\mathfrak{e}} \sum_{k=1}^{p}k\beta_{2k} = 0.
    \end{align*}
\end{proof}
\begin{remark}
    Note that we have not parameterized the full discriminant $\xi({\bf t}) = 0$, which also contains higher-degeneracy points where, for example,  $P=P_{\sigma}=P_{\sigma \sigma} = 0$. Indeed, for ${\bf t}(\boldsymbol{\beta}) \in \Sigma_{cr}$ we have
    \begin{align*}
        P_{\sigma\sigma}(\sigma(\boldsymbol{\beta}),{\bf t}(\boldsymbol{\beta})) &= -\sum_{k=1}^pk(k-1)\binom{2k-1}{k}\frac{\beta_{2k}}{\binom{2k-1}{k} \mathfrak{e}}\left(\frac{\mathfrak{e}-1}{\mathfrak{e}\varkappa}\right)^{k-1}\left(\frac{\mathfrak{e} \varkappa}{\mathfrak{e}-1}\right)^{k-2}\\
        &= -\frac{(\mathfrak{e}-1)(\mathfrak{z}-\mathfrak{e})}{\mathfrak{e}^2 \varkappa} <0,
    \end{align*}
which is evidently never zero, whereas (provided that more than two of the $t_{2k}$ are nonzero) we can find ${\bf t}_c$ such that $P_{\sigma}(\sigma({\bf t}_c);{\bf t}_c) =P_{\sigma\sigma}(\sigma({\bf t}_c);{\bf t}_c) =0$. For instance, if we consider the situation when only $t_4$ and $t_6$ are nonzero, then $P(\sigma;{\bf t}) = -\varkappa + \sigma + 3t_4\sigma^2+10t_6\sigma^3$. If we take ${\bf t}_c = \left(-\frac{1}{9\varkappa},\frac{1}{270\varkappa^2}\right)$, then with $\sigma_c = 3\varkappa$, $P(\sigma_c,{\bf t}_c) = P_{\sigma}(\sigma_c,{\bf t}_c) = P_{\sigma\sigma}(\sigma_c,{\bf t}_c) = 0$. Obviously, this point does not lie on the critical surface $\Sigma_{cr}$.
\end{remark}
\begin{remark}
    Let us consider a \emph{fixed} ${\bf t}(\boldsymbol{\alpha})\in\Sigma_{cr}$. Since $P_{\sigma\sigma}(\sigma(\boldsymbol{\alpha}),{\bf t}(\boldsymbol{\alpha}))\neq 0$ for any fixed ${\bf t}(\boldsymbol{\alpha}) \in \Sigma_{cr}$, it follows that $\nabla_{{\bf t}} \xi ({\bf t}(\boldsymbol{\alpha})) \neq 0$. Obviously, $\xi({\bf t})$ and $K\cdot \xi({\bf t})$ both define the same algebraic variety, where $K$ is any nonzero constant. We can uniquely specify a normalization ($=$ value of $K = K(\boldsymbol{\alpha})$) of $\xi({\bf t})$ such that
        \begin{equation}\label{discrim-normalization}
            \nabla_{{\bf t}} \xi ({\bf t})\bigg|_{{\bf t} = {\bf t}(\boldsymbol{\alpha})} = \frac{\boldsymbol{\nu}(\boldsymbol{\alpha})}{||\boldsymbol{\nu}(\boldsymbol{\alpha})||}.
        \end{equation}
    

\noindent    From here on, we assume our choice of $\xi({\bf t})$ is endowed with the normalization \eqref{discrim-normalization}.
\end{remark}

\begin{proposition}\label{prop ift}
    Let $\boldsymbol{\tau}:=\left( \tau_1,\cdots, \tau_{p-1}\right)$ be a real vector\footnote{Caution: the vector $\boldsymbol{\tau}$ is of length $p-1$, whereas all other vectors in this section are of length $p$.}, and suppose $\boldsymbol{\alpha}_{\boldsymbol{\tau}}$ is a collection of vectors in $\mathbb{R}^{p}$ which satisfy:
        \begin{enumerate}[(a)]
            \item $\boldsymbol{\alpha}_{\boldsymbol{0}} = \boldsymbol{\alpha}$, and for $\boldsymbol{\tau}$ sufficiently small, $(\boldsymbol{\alpha}_{\boldsymbol{\tau}})_{2k}>0$, and $\sum_{k=1}^p(\boldsymbol{\alpha}_{\boldsymbol{\tau}})_{2k} = 1$, and the dependence on ${\boldsymbol{\tau}}$ is analytic,
            \item The vectors
                \begin{equation}\label{Ta}
                    {\bf T}_a:=\frac{\partial}{\partial \tau_a} {\bf t}(\boldsymbol{\alpha}_{\boldsymbol{\tau}})\bigg|_{\boldsymbol{\tau} = 0}
                \end{equation}
            are orthonormal.
        \end{enumerate}
    Then, there exists a unique $f = f(\boldsymbol{\tau};\xi)$, analytic in a neighborhood of $(\boldsymbol{\tau};\xi) = \boldsymbol{0}$, such that $f(\boldsymbol{0},0) = 0$, $f(\boldsymbol{\tau};0) \equiv 0$, and if we define
        \begin{equation}
            {\bf t}(\boldsymbol{\tau};\xi) := {\bf t}(\boldsymbol{\alpha}_{\boldsymbol{\tau}}) + \boldsymbol{\nu}(\boldsymbol{\alpha}) f(\boldsymbol{\tau};\xi),
        \end{equation}
    then
        \begin{equation}\label{implicit-xi-equation}
            \xi({\bf t}(\boldsymbol{\tau};\xi_0)) = \xi_0.
        \end{equation}
\end{proposition}
\begin{proof}
    Let us first observe that we can always $\boldsymbol{\alpha}_{\boldsymbol{\tau}}$ satisfying conditions (a) and (b) (one can even take the dependence of $\boldsymbol{\alpha}_{\boldsymbol{\tau}}$ on $\boldsymbol{\tau}$ to be linear, and then conditions (a) and (b) are essentially a straightforward application of the Gram-Schmidt method). Furthermore observe that ${\bf t}(\boldsymbol{\alpha}_{\boldsymbol{\tau}})$ lies on zero locus of the discriminant, since if we take $\sigma_0(\boldsymbol{\tau}) := \sigma({\bf t}(\boldsymbol{\alpha}_{\boldsymbol{\tau}}))$, then
        \begin{equation*}
            P(\sigma_0(\boldsymbol{\tau}),{\bf t}(\boldsymbol{\alpha}_{\boldsymbol{\tau}})) = P_{\sigma}(\sigma_0(\boldsymbol{\tau}),{\bf t}(\boldsymbol{\alpha}_{\boldsymbol{\tau}})) = 0.
        \end{equation*}
        
    We show that, for any fixed $\boldsymbol{\tau}$, the implicit function $f\mapsto f(\boldsymbol{\tau};\cdot)$ exists. Analyticity in $\boldsymbol{\tau}$ then follows from the fact that 
    ${\bf t}(\boldsymbol{\alpha}_{\boldsymbol{\tau}})$ is analytic for sufficiently small $\boldsymbol{\tau}$.
    
    Observe that $f(\boldsymbol{\tau};0) = 0$, $\xi = 0$ indeed solves the equation \eqref{implicit-xi-equation}, since ${\bf t}(\boldsymbol{\alpha}_{\boldsymbol{\tau}})$ lies on zero locus of the discriminant. Differentiating $\eqref{implicit-xi-equation}$ with respect to $f$, we obtain that
        \begin{equation*}
           \dod{}{f}\xi({\bf t}(\boldsymbol{\tau};\xi))\bigg|_{f=0}  = \sum_{k=1}^p \frac{\partial \xi}{\partial t_{2k}} \frac{\partial t_{2k}}{\partial f} \bigg|_{f=0} = \nabla_{{\bf t}} \xi({\bf t})\bigg|_{{\bf t}= {\bf t}(\boldsymbol{\alpha}_{\boldsymbol{\tau}})}\cdot  \boldsymbol{\nu}(\boldsymbol{\alpha})= 
           ||\boldsymbol{\nu}(\boldsymbol{\alpha})|| + o(1),
        \end{equation*}
    which is non-vanishing for sufficiently small $\boldsymbol{\tau}$. Thus, the implicit function theorem guarantees the existence of $f$; by construction $f(\boldsymbol{\tau},0)\equiv 0$.
\end{proof}

\begin{corollary}
    The change of variables ${\bf t}\mapsto {\bf t}(\boldsymbol{\tau}, \xi)$ is well-defined in a neighborhood of ${\bf t} = {\bf t}(\boldsymbol{\alpha})$. Furthermore, ${\bf t}(\boldsymbol{\tau}, 0) = {\bf t}(\boldsymbol{\alpha}_{\boldsymbol{\tau}})$.
\end{corollary}
\begin{proof}
    Let 
        \begin{equation}
            \mathcal{J}:= \bigg|\frac{\partial (\boldsymbol{\tau}, \xi)}{\partial {\bf t}} \bigg|_{{\bf t} = {\bf t}(\boldsymbol{\alpha})}
        \end{equation}
    denote the Jacobian of this transformation. We have arranged by the previous theorem that the Jacobian matrix is an orthogonal matrix; it thus follows that $\mathcal{J} = 1$,
    and the transformation is ${\bf t}\mapsto {\bf t}(\boldsymbol{\tau}, \xi)$ is well defined locally.
\end{proof}
\begin{remark}
    Observe that, for fixed $\xi_0$, ${\bf t}(\boldsymbol{\tau},\xi_0)$ parameterizes the level hypersurface
        \begin{equation}
            \xi({\bf t}) = \xi_0.
        \end{equation}
    Since all of these hypersurfaces carry the same normal vector at $\boldsymbol{\tau} = 0$, it follows that, for any $a=1,...,p-1$,
        \begin{equation}
            \frac{\partial}{\partial \tau_a}{\bf t}(\boldsymbol{\tau},\xi_0) \perp \boldsymbol{\nu}(\boldsymbol{\alpha}).
        \end{equation}
\end{remark}

\begin{proposition}
    In a neighborhood of ${\bf t} = {\bf t}(\boldsymbol{\alpha})$,
        \begin{equation}
            \sigma({\bf t}(\boldsymbol{\tau}, \xi)) = \sum_{k=0}^{\infty} \sigma_{k}(\boldsymbol{\tau}) \xi^{k/2},
        \end{equation}
    where $\sigma_{k}(\boldsymbol{\tau})$ are analytic functions of $\boldsymbol{\tau}$. Equivalently, there exist analytic functions $a({\bf t})$, $b({\bf t})$ in a neighborhood of ${\bf t} = {\bf t}(\boldsymbol{\alpha})$ such that
        \begin{equation}
            \sigma({\bf t}) = a({\bf t}) + b({\bf t})\xi({\bf t})^{1/2}.
        \end{equation}
    Explicitly,
        \begin{equation}
            a({\bf t}(\boldsymbol{\alpha})) = \frac{\mathfrak{e}\varkappa}{\mathfrak{e} -1}, \qquad b({\bf t}(\boldsymbol{\alpha})) = \sigma_{1}(0) = -\left(\frac{2\mathfrak{e}^3 \varkappa^2}{(\mathfrak{z}-\mathfrak{e})(\mathfrak{e}-1)^2}||\boldsymbol{\nu}(\boldsymbol{\alpha})||\right)^{1/2}.
        \end{equation}
\end{proposition}
\begin{proof} 
    Make the ansatz
    \begin{equation*}
        \sigma({\bf t}(\boldsymbol{\tau}, \xi)) = \sigma_0(\boldsymbol{\tau}) +\sigma_1(\boldsymbol{\tau})\xi^{1/2} + \sum_{k=2}^{\infty}\sigma_k(\boldsymbol{\tau})\xi^{k/2},
    \end{equation*}
    where $\sigma_0(\boldsymbol{\tau}) := \sigma({\bf t}(\boldsymbol{\alpha}_{\boldsymbol{\tau}}))$ is analytic and satisfies $P(\sigma_0;{\bf t}(\boldsymbol{\alpha}_{\boldsymbol{\tau}})) = P_{\sigma}(\sigma_0;{\bf t}(\boldsymbol{\alpha}_{\boldsymbol{\tau}}))=0$. Write
    \begin{equation*}
        {\bf t}(\boldsymbol{\tau}, \xi) = {\bf t}(\boldsymbol{\alpha}_{\boldsymbol{\tau}}) + {\bf t}_{\xi}(\boldsymbol{\tau}, 0)\xi + \mathcal{O}(\xi^2).
    \end{equation*}
    Expanding the equation $0 = P(\sigma;{\bf t})$ in $\xi$, we obtain that
        \begin{align*}
            0 &= P(\sigma_0;{\bf t}(\boldsymbol{\alpha}_{\boldsymbol{\tau}})) + P_{\sigma}(\sigma_0;{\bf t}(\boldsymbol{\alpha}_{\boldsymbol{\tau}}))\cdot\sigma_1\xi^{1/2}\\
            &+ \left(\frac{1}{2}P_{\sigma\sigma}(\sigma_0;{\bf t}(\boldsymbol{\alpha}_{\boldsymbol{\tau}})) \sigma_{1}^2 + \sum_{k=1}^{p}\frac{\partial P}{\partial t_{2k}}(\sigma_0;{\bf t}(\boldsymbol{\alpha}_{\boldsymbol{\tau}})) \frac{\partial t_{2k}}{\partial \xi}(\boldsymbol{\tau}, 0) + P_{\sigma}(\sigma_{0};{\bf t}(\boldsymbol{\alpha}_{\boldsymbol{\tau}}))\sigma_2\right)\xi + \mathcal{O}(\xi^{3/2}).
        \end{align*}
    Since $P(\sigma_0;{\bf t}(\boldsymbol{\alpha}_{\boldsymbol{\tau}})) = P_{\sigma}(\sigma_0;{\bf t}(\boldsymbol{\alpha}_{\boldsymbol{\tau}})) = 0$, the $\mathcal{O}(1)$ and $\mathcal{O}(\xi^{1/2})$ terms vanish identically, and the equation at order $\xi$ then reads
        \begin{equation*}
            0 = \frac{1}{2}P_{\sigma\sigma}(\sigma_0;{\bf t}(\boldsymbol{\alpha}_{\boldsymbol{\tau}})) \sigma_{1}^2 + \sum_{k=1}^{p}\frac{\partial P}{\partial t_{2k}}(\sigma_0;{\bf t}(\boldsymbol{\alpha}_{\boldsymbol{\tau}})) \frac{\partial t_{2k}}{\partial \xi}(\boldsymbol{\tau}, 0).
        \end{equation*}
    Since $P_{\sigma\sigma}(\sigma_0;{\bf t}(\boldsymbol{\alpha}_{\boldsymbol{\tau}})) \neq 0$, we can solve for $\sigma_1$ explicitly:
        \begin{equation*}
            \sigma_1(\boldsymbol{\tau}) = \pm \left(-\frac{2}{P_{\sigma\sigma}(\sigma_0;{\bf t}(\boldsymbol{\alpha}_{\boldsymbol{\tau}}))}\sum_{k=1}^{p}\frac{\partial P}{\partial t_{2k}}(\sigma_0;{\bf t}(\boldsymbol{\alpha}_{\boldsymbol{\tau}})) \frac{\partial t_{2k}}{\partial \xi}(\boldsymbol{\tau}, 0)\right)^{1/2}.
        \end{equation*}
    The function under the square root is nonzero when $\boldsymbol{\tau} = 0$, and so either choice of $\sigma_1(\boldsymbol{\tau})$ defines an analytic function; we take the $-$ sign to coincide with its value in regular cases. All remaining terms are then explicit polynomial expressions in the coefficients of the series expansion of ${\bf t}(\boldsymbol{\alpha}_{\boldsymbol{\tau}})$,
    $\sigma_0$, and $\sigma_1$, and are thus analytic.

    Finally, observe that
        \begin{equation*}
            \nabla_{{\bf t}}P(\sigma(\boldsymbol{\alpha});{\bf t}(\boldsymbol{\alpha}))= \frac{\mathfrak{e} \varkappa }{\mathfrak{e}-1}\boldsymbol{\nu}(\boldsymbol{\alpha}),
        \end{equation*}
    and further that
        \begin{equation*}
            \frac{\partial {\bf t}}{\partial \xi}(0,0) = \frac{\boldsymbol{\nu}(\boldsymbol{\alpha})}{||\boldsymbol{\nu}(\boldsymbol{\alpha})||}.
        \end{equation*}
    Thus, since $P_{\sigma\sigma}(\sigma(\boldsymbol{\alpha}),{\bf t}(\boldsymbol{\alpha})) = -\frac{(\mathfrak{z}-\mathfrak{e})(\mathfrak{e}-1)}{\mathfrak{e}^2\varkappa}$,
        \begin{equation*}
            \sigma_1(0) = -\left(\frac{2\mathfrak{e}^3 \varkappa^2}{(\mathfrak{z}-\mathfrak{e})(\mathfrak{e}-1)^2}||\boldsymbol{\nu}(\boldsymbol{\alpha})||\right)^{1/2}.
        \end{equation*}
\end{proof}
\begin{remark}
    Observe that the functions $a({\bf t})$, $b({\bf t})$ are explicit in terms of the analytic coefficients $\sigma_k(\boldsymbol{\tau})$ and $\xi$:
        \begin{equation}
            a({\bf t}(\boldsymbol{\tau},\xi)) = \sum_{k=0}^{\infty} \sigma_{2k}(\boldsymbol{\tau})\xi^k,\qquad\qquad b({\bf t}(\boldsymbol{\tau},\xi)) = \sum_{k=0}^{\infty} \sigma_{2k+1}(\boldsymbol{\tau})\xi^k.
        \end{equation}
\end{remark}

\begin{corollary}
    With the notations of the previous proposition, we can write
        \begin{equation}
            P_{\sigma}(\sigma({\bf t});{\bf t}) = f({\bf t}) + g({\bf t})\xi({\bf t})^{1/2},
        \end{equation}
    where $f({\bf t}(\boldsymbol{\alpha})) = 0$, 
    \begin{equation}\label{eq:g-var}
        g({\bf t}(\boldsymbol{\alpha})) = -\frac{(\mathfrak{e}-1)(\mathfrak{z}-\mathfrak{e})}{\mathfrak{e}^2 \varkappa}\sigma_1(0) = \left(\frac{2(\mathfrak{z}-\mathfrak{e})}{\mathfrak{e}}||\boldsymbol{\nu}(\boldsymbol{\alpha})||\right)^{1/2}.
    \end{equation}
\end{corollary}

\begin{proposition}\label{Phi-structure-prop}
For $\xi,\boldsymbol{\tau}$ sufficiently small,
    \begin{equation}
        \prod_{k=1}^{p}\left(\frac{t_{2k}(\xi;\boldsymbol{\tau})}{t_{2k}(\boldsymbol{\alpha})}\right)^{-V\alpha_{2k}} = \left(1 - \mathfrak{e}||\boldsymbol{\nu}(\boldsymbol{\alpha})|| \cdot \xi + \frac{1}{2}\sum_{a,b=1}^{p-1}\mathcal{H}_{ab}\tau_a\tau_b + \sum_{a=1}^{p-1}C_a\tau_a\xi + \frac{1}{2}D\xi^2 + \mathcal{O}(3)\right)^{-V},
    \end{equation}
where $\mathcal{O}(3)$ denotes terms in the Taylor series of degree $3$ or higher. Moreover, 
    \begin{equation}
        \prod_{k=1}^{p}\left(\frac{t_{2k}( V^{-1}\xi; V^{-1/2}\boldsymbol{\tau})}{t_{2k}(\boldsymbol{\alpha})}\right)^{-V\alpha_{2k}} = \exp\left(\mathfrak{e}||\boldsymbol{\nu}(\boldsymbol{\alpha})|| \cdot \xi-\frac{1}{2}\langle \boldsymbol{\tau},\mathcal{H}\boldsymbol{\tau}\rangle\right)(1 + \mathcal{O}(V^{-1/2})),
    \end{equation}
where the error in the above is locally uniform in $\xi,\boldsymbol{\tau}$.
\end{proposition}
\begin{proof}
    Observe that
        \begin{equation*}
             \left( \frac{\alpha_{2}}{t_{2}(\boldsymbol{\alpha})},\frac{\alpha_{4}}{t_{4}(\boldsymbol{\alpha})},\cdots, \frac{\alpha_{2p}}{t_{2p}(\boldsymbol{\alpha})}\right) = -\mathfrak{e}\boldsymbol{\nu}(\boldsymbol{\alpha}).
        \end{equation*}
    So,
    \begin{equation*}
        \frac{\partial}{\partial \tau_a}\prod_{k=1}^{p}\left(\frac{t_{2k}(\xi;\boldsymbol{\tau})}{t_{2k}(\boldsymbol{\alpha})}\right)^{-V\alpha_{2k}}\bigg|_{\boldsymbol{\tau}=0,\xi=0} = \sum_{k=1}^{p} \frac{\alpha_{2k}}{t_{2k}(\boldsymbol{\alpha})} \frac{\partial t_{2k}}{\partial \tau_a}(0,0) = -\mathfrak{e} \boldsymbol{\nu}(\boldsymbol{\alpha})\cdot {\bf T}_a = 0.
    \end{equation*}
    On the other hand,
    \begin{equation*}
        \frac{\partial}{\partial \xi}\prod_{k=1}^{p}\left(\frac{t_{2k}(\xi;\boldsymbol{\tau})}{t_{2k}(\boldsymbol{\alpha})}\right)^{-V\alpha_{2k}}\bigg|_{\boldsymbol{\tau}=0,\xi=0} = \sum_{k=1}^{p} \frac{\alpha_{2k}}{t_{2k}(\boldsymbol{\alpha})} \frac{\partial t_{2k}}{\partial \xi}(0,0) = -\mathfrak{e}\boldsymbol{\nu}(\boldsymbol{\alpha})\cdot \frac{\boldsymbol{\nu}(\boldsymbol{\alpha})}{||\boldsymbol{\nu}(\boldsymbol{\alpha})||} = -\mathfrak{e} ||\boldsymbol{\nu}(\boldsymbol{\alpha})||.
    \end{equation*}
    We do not specify the form of the higher-order Taylor coefficients, and have only introduced them in the statement of the proposition for later convenience (see Proposition \ref{prop:asymptotic-formula}).
\end{proof}

Let us return our attention to the local structure of the recurrence coefficients, which are an important ingredient in the proof of our main theorem. We now state a more precise version of Lemma \ref{general-derivative-lemma}, which will allow us to study the effect of differentiation with respect to the t'Hooft parameter $\varkappa$.
\begin{proposition}\label{general-derivative-lemma-precise}
        Let $Q := Q(r_0;{\bf t})$ be a polynomial, $\Delta\geq 1$ a fixed positive integer, and suppose that 
        \begin{equation}
            Q(r_0;{\bf t}) = q_a({\bf t}) + q_b({\bf t})\xi({\bf t})^{1/2},
        \end{equation}
        where $q_a$, $q_b$ are analytic in a neighborhood of ${\bf t} = {\bf t}(\boldsymbol{\alpha})$, and $q_a({\bf t}(\boldsymbol{\alpha})) = Q_0$. Then, for any $k\geq 0$,
    \begin{equation}
        \frac{\partial^k}{\partial \varkappa^k}\frac{Q(r_0;{\bf t})}{P_{\sigma}(r_0;{\bf t})^{\Delta}} = \frac{\tilde{Q}(r_0;{\bf t})}{P_{\sigma}(r_0;{\bf t})^{\Delta + 2k}},
    \end{equation}
    where $\tilde{Q}(r_0;{\bf t})$ is some other polynomial in $r_0,{\bf t}$, which satisfies 
        \begin{equation}
            \tilde{Q}(r_0;{\bf t}) = \tilde{q}_a({\bf t}) + \tilde{q}_b({\bf t})\xi({\bf t})^{1/2},
        \end{equation}
    and
        \begin{equation}
            \tilde{q}_a({\bf t}(\boldsymbol{\alpha})) = c(\boldsymbol{\alpha})^{\Delta}\prod_{j=1}^{m}(\Delta +2j -2) Q_0,
        \end{equation}
    with $c(\boldsymbol{\alpha}):=-P_{\sigma\sigma}(\sigma(\boldsymbol{\alpha}),{\bf t}(\boldsymbol{\alpha})) = \frac{(\mathfrak{e}-1)(\mathfrak{z}-\mathfrak{e})}{\mathfrak{e}^2 \varkappa}$.
\end{proposition}
\begin{proof}
    The proof is essentially a refinement of the proof of Lemma \ref{general-derivative-lemma}. Observe that if
    $Q(r_0;{\bf t}) = q_a({\bf t}) + q_b({\bf t})\xi({\bf t})^{1/2}$, differentiating in $\varkappa$, and recalling Equation \eqref{eq:sigma-kappa-derivative}, we find that
        \begin{align*}
            \frac{\partial}{\partial \varkappa}\frac{Q(r_0;{\bf t})}{P_{\sigma}(r_0;{\bf t})^{\Delta}} = \frac{P_{\sigma}Q_{\sigma}-\Delta Q P_{\sigma\sigma}}{P_{\sigma}^{\Delta+2}},
        \end{align*}
    and so we see that the numerator is of the form 
    \begin{equation*}
        P_{\sigma}Q_{\sigma}-\Delta Q P_{\sigma\sigma} = \tilde{q}_a({\bf t}) + \tilde{q}_b({\bf t})\xi({\bf t})^{1/2},
    \end{equation*}
    where $\tilde{q}_a({\bf t})$, $\tilde{q}_b({\bf t})$ are analytic in a neighborhood of ${\bf t}= {\bf t}(\boldsymbol{\alpha})$, and $\tilde{q}_a({\bf t}(\boldsymbol{\alpha})) = c \Delta Q_0$. The remainder of the proof is a standard application of induction.
\end{proof}
Consequentially, we have the following more precise statement about the singular structure of $r_g({\bf t})$:
\begin{proposition}\label{recurrence-coefficient-local-prop}
    For each $g\geq 1$, near ${\bf t} = {\bf t}(\boldsymbol{\alpha})$ we can represent $r_g({\bf t})$ as
        \begin{equation}
            r_g({\bf t}) = \frac{\mathcal{C}_g(\boldsymbol{\alpha})}{\xi({\bf t})^{\frac{1}{2}(5g-1)}}\left(a_g({\bf t}) + b_g({\bf t})\xi({\bf t})^{1/2} \right),
        \end{equation}
    where $a_g(\mb t),b_g(\mb t)$ are analytic in a neighborhood of ${\bf t} = {\bf t}(\boldsymbol{\alpha})$, $a_g({\bf t}(\boldsymbol{\alpha})) = 1$, and where
            \begin{equation}
                \mathcal{C}_0 = -\left(\frac{2\mathfrak{e}^3\varkappa^2}{(\mathfrak{z}-\mathfrak{e})(\mathfrak{e}-1)^2}||\boldsymbol{\nu}(\boldsymbol{\alpha})||\right)^{1/2},\qquad\qquad \mathcal{C}_1 = \frac{1}{24}\frac{\mathfrak{e}-1}{\mathfrak{e}\varkappa }\frac{1}{||\boldsymbol{\nu}(\boldsymbol{\alpha})||^2},
            \end{equation}
        and for $g\geq 2$,
            \begin{equation}\label{eq:C-recursion}
        \mathcal{C}_g = -\frac{1}{2\mathcal{C}_0}\sum_{\ell=1}^{g-1}\mathcal{C}_{\ell}\mathcal{C}_{g-\ell} - (5g-4)(5g-6)\frac{\mathcal{C}_1}{\mathcal{C}_0}\mathcal{C}_{g-1}.
    \end{equation}
\end{proposition}
\begin{proof}
    Observe that:
        \begin{itemize}
            \item For any $\Delta\geq 1$, there exist $a_{\Delta}({\bf t})$, $b_{\Delta}({\bf t})$, analytic in a neighborhood of ${\bf t} = {\bf t}(\boldsymbol{\alpha})$, such that
                \begin{equation*}
                    P_{\sigma}(\sigma({\bf t}),{\bf t})^{\Delta} = \xi({\bf t})^{\Delta/2}(a_{\Delta}({\bf t}) + b_{\Delta}({\bf t})\xi({\bf t})^{1/2}),
                \end{equation*}
            where $a_{\Delta}({\bf t}(\boldsymbol{\alpha})) = g({\bf t}(\boldsymbol{\alpha}))^{\Delta}$, as defined in Equation \eqref{eq:g-var},
            \item There exist $\hat{a}({\bf t}), \hat{b}({\bf t})$, analytic in a neighborhood of ${\bf t} = {\bf t}(\boldsymbol{\alpha})$, such that
                \begin{equation*}
                    P_{\sigma\sigma}(\sigma({\bf t}),{\bf t}) = \hat{a}({\bf t}) + \hat{b}({\bf t})\xi({\bf t})^{1/2},
                \end{equation*}
            and $\hat{a}({\bf t}(\boldsymbol{\alpha})) = -\frac{(\mathfrak{e}-1)(\mathfrak{z}-\mathfrak{e})}{\mathfrak{e}^2 \varkappa}$, 
            \item As a consequence of the previous two observations, and Proposition \ref{prop:recurrence-expansion}, there exist \sloppy $a_1({\bf t}), b_1({\bf t})$, analytic in a neighborhood of ${\bf t} = {\bf t}(\boldsymbol{\alpha})$, such that $a_1({\bf t}(\boldsymbol{\alpha}))=1$, and
                \begin{equation*}
                    r_1({\bf t}) = \frac{\mathcal{C}_1}{\xi({\bf t})^{2}}\left(a_1({\bf t}) + b_1({\bf t})\xi({\bf t})^{1/2}\right).
                \end{equation*}
        \end{itemize}
        We now proceed by induction. The above observations establish the base case of the induction. By Corollary \ref{string-corollary}, we have that
        \begin{equation*}
            r_g = \underbrace{-\frac{1}{2}\frac{P_{\sigma\sigma}}{P_{\sigma}} \left(\sum_{\ell=1}^{g-1}r_{\ell}r_{g-\ell}+\frac{1}{3}r_{g-1}''r_0\right)}_{(1)} \underbrace{- \frac{1}{P_{\sigma}}H(r_0,r_1,...,r_{g-1})}_{(2)}.
        \end{equation*}
        where each monomial $m$ contributing to $H$ satisfies $\psi(m)<5g-2$. Assuming the inductive hypothesis, and using Proposition \ref{general-derivative-lemma-precise} to rewrite $r_{g-1}''$ in terms of $r_g$, by direct calculation one can see that the term (1) in the above can be written as
            \begin{equation*}
                \text{(1)} = \frac{\mathcal{C}_g}{\xi({\bf t})^{\frac{1}{2}(5g-1) }}\left(\tilde{a}_g({\bf t}) + \tilde{b}_g({\bf t})\xi({\bf t})^{1/2}\right),
            \end{equation*}
        where $\tilde{a}_g({\bf t}), \tilde{b}_g({\bf t})$ are analytic in a neighborhood of ${\bf t} = {\bf t}(\boldsymbol{\alpha})$, $\tilde{a}_g({\bf t}(\boldsymbol{\alpha})) = 1$, and
        $\mathcal{C}_g$ are as in the statement of the proposition. On the other hand, because each monomial $m$ in $H$ satisfies
        $\psi(m)<5g-2$, there exist $\tilde{c}({\bf t}), \tilde{d}({\bf t})$, analytic in a neighborhood of ${\bf t} = {\bf t}(\boldsymbol{\alpha})$, such that
            \begin{equation*}
                \text{(2)} = \frac{1}{\xi({\bf t})^{K/2}}\left(\tilde{c}({\bf t}) + \tilde{d}({\bf t})\xi({\bf t})^{1/2}\right),
            \end{equation*}
        and $K<5g-2$, i.e. the singularity of (2) is of strictly lower order than that of (1). Combining these results yields the result of the proposition.
\end{proof}

\subsubsection{Calculation of the Hessian determinant}
In our main asymptotic result, a closed form expression for the determinant of $\boldsymbol{\mathcal{H}}$ is necessary. We now proceed to derive this expression.
Observe that, by definition,
    \begin{equation}
            \mathcal{H}_{ab} = \frac{\partial^2}{\partial \tau_a\partial \tau_b} \left(\prod_{k=1}^p \left(\frac{t_{2k}(\xi;\boldsymbol{\tau})}{t_{2k}(\boldsymbol{\alpha})}\right)^{\alpha_{2k}}\right)\bigg|_{\xi = 0, \boldsymbol{\tau} = 0}.
        \end{equation}
Since $f(\boldsymbol{\tau},0) \equiv 0$ (see Proposition \ref{prop ift}), if we first evaluate at $\xi = 0$, we see that the above expression is equivalent to
        \begin{equation}
            \mathcal{H}_{ab} = \frac{\partial^2}{\partial \tau_a\partial \tau_b} \left(\prod_{k=1}^p \left(\frac{t_{2k}(\boldsymbol{\alpha}_{\boldsymbol{\tau}})}{t_{2k}(\boldsymbol{\alpha})}\right)^{\alpha_{2k}}\right)\bigg|_{\boldsymbol{\tau} = 0}.
        \end{equation}
To this end, we define
    \begin{equation}\label{Hab G}
        \mathcal{G}(\boldsymbol{\tau}) := \prod_{k=1}^p \left(\frac{t_{2k}(\boldsymbol{\alpha}_{\boldsymbol{\tau}})}{t_{2k}(\boldsymbol{\alpha})}\right)^{\alpha_{2k}}.
    \end{equation}
        It is convenient to consider a holomorphic branch of $\mathcal{L}(\boldsymbol{\tau}):= \log \mathcal{G}(\boldsymbol{\tau})$ near $\boldsymbol{\tau} = \boldsymbol{0}$, which is possible since $\mathcal{G}(\boldsymbol{0}) = 1$ and since the dependence on $\boldsymbol{\tau}$ is analytic, $\mathcal{G}(\boldsymbol{\tau}) \neq 0$ in a neighborhood of $\boldsymbol{\tau} = \boldsymbol{0}$. 

        \begin{lemma} Let $\mathcal{L}(\boldsymbol{\tau})$ denote a holomorphic branch of $\log \mathcal{G}(\boldsymbol{\tau})$. It holds that
                    \begin{equation}\label{Hab d^2 log}
            \mathcal{H}_{ab} = \frac{\partial^2}{\partial \tau_a\partial \tau_b} \mathcal{L}(\boldsymbol{\tau})\bigg|_{\boldsymbol{\tau} = \boldsymbol{0}}.
        \end{equation}
        \end{lemma}
        \begin{proof}
            Differentiating $\mathcal{G}(\boldsymbol{\tau}) = \ee^{\mathcal{L}(\boldsymbol{\tau})}$ twice yields \begin{equation}
                  \frac{\partial^2 \mathcal{G}}{\partial \tau_a\partial\tau_b}
    =
    \mathcal{G}
    \left(
        \frac{\partial^2 \mathcal{L}}{\partial \tau_a\partial\tau_b}
        +
        \frac{\partial \mathcal{L}}{\partial \tau_a}
        \frac{\partial \mathcal{L}}{\partial \tau_b}
    \right).
            \end{equation}
            Since $\mathcal{G}(\boldsymbol{0}) = 1$, in order to prove the lemma it remains to show that
            \begin{equation}\label{nts0}
                  \frac{\partial \mathcal{L}}{\partial \tau_a} \bigg|_{\boldsymbol{\tau}=\boldsymbol{0}}  = 0.
            \end{equation}
            By chain rule
            \[ \frac{\partial}{\partial \tau_a}  \mathcal{L} ({\bf t}(\boldsymbol{\alpha}_{\boldsymbol{\tau}})) = \sum_{k=1}^p
    \frac{\partial}{\partial t_{2k}} \mathcal{L} ({\bf t})
    \frac{\partial}{\partial \tau_a}
    t_{2k}(\boldsymbol{\alpha}_{\boldsymbol{\tau}}). \] Evaluating at $\boldsymbol{\tau}=\boldsymbol{0}$ and recalling \eqref{Ta}, we have
    \[ \frac{\partial}{\partial \tau_a}  \mathcal{L} ({\bf t}(\boldsymbol{\alpha}_{\boldsymbol{\tau}})) = \nabla_{\bf t}\mathcal{L}({\bf t}) \ \cdot \ {\bf T}_a.\] A direct calculation shows $\partial \mathcal{L}/\partial t_{2k} = \alpha_{2k}/t_{2k}$ and, recalling \eqref{t2k} and \eqref{nu2k} we have $$t_{2k} = - \frac{\alpha_{2k}}{\mathfrak{e}{\nu}_{2k}(\boldsymbol{\alpha})}.$$ Therefore, recalling the notation in \eqref{vec nu}, we have $\nabla_{\bf t}\mathcal{L}({\bf t}) = - \mathfrak{e}  \boldsymbol{\nu}(\boldsymbol{\alpha}) $. Since $\boldsymbol{\nu}(\boldsymbol{\alpha})$ is normal to ${\bf T}_a$, \eqref{nts0} holds.
        \end{proof}

        \begin{lemma}
It holds that
\begin{equation}\label{d^2 log}
    \begin{split}
        \frac{\partial^2}{\partial \tau_a\partial \tau_b} \mathcal{L}(\boldsymbol{\tau})\bigg|_{\boldsymbol{\tau}=\boldsymbol{0}} 
    &=
    -
    \sum_{k=1}^p
    \frac{1}{\alpha_{2k}}
    \frac{\partial}{\partial \tau_a}
    (\boldsymbol{\alpha}_{\boldsymbol{\tau}})_{2k}
    \bigg|_{\boldsymbol{\tau}=\boldsymbol{0}}
    \frac{\partial}{\partial \tau_b}
    (\boldsymbol{\alpha}_{\boldsymbol{\tau}})_{2k}
    \bigg|_{\boldsymbol{\tau}=\boldsymbol{0}}
    \\
    &\quad
    -
    \frac{1}{\mathfrak{e}(\mathfrak{e}-1)}
    \left(
     \sum_{j=1}^p   j
        \frac{\partial}{\partial \tau_a}
    (\boldsymbol{\alpha}_{\boldsymbol{\tau}})_{2j}
        \bigg|_{\boldsymbol{\tau}=\boldsymbol{0}}
    \right)
    \left(
        \sum_{\ell=1}^p    \ell
    \frac{\partial}{\partial \tau_b}
(\boldsymbol{\alpha}_{\boldsymbol{\tau}})_{2\ell}
\bigg|_{\boldsymbol{\tau}=\boldsymbol{0}}
    \right).
    \end{split}
\end{equation}

        \end{lemma}
        \begin{proof}
            Recall that
\[
    t_{2k}(\boldsymbol{\alpha}_{\boldsymbol{\tau}})
    =
    -\frac{(\boldsymbol{\alpha}_{\boldsymbol{\tau}})_{2k}}
    {
        \binom{2k-1}{k}\mathfrak{e}(\boldsymbol{\tau})
    }
    \left(
        \frac{\mathfrak{e}(\boldsymbol{\tau})-1}
        {\mathfrak{e}(\boldsymbol{\tau})\varkappa}
    \right)^{k-1},
\qquad
    \mathfrak{e}(\boldsymbol{\tau})
    =
    \sum_{j=1}^p j(\boldsymbol{\alpha}_{\boldsymbol{\tau}})_{2j}.
\]
Taking the logarithm and expanding yields
\begin{equation}\label{log t2k main terms}
    \log t_{2k}(\boldsymbol{\alpha}_{\boldsymbol{\tau}})
    =
    \log(\boldsymbol{\alpha}_{\boldsymbol{\tau}})_{2k}
    -
    k\log\mathfrak{e}(\boldsymbol{\tau})
    +
    (k-1)\log(\mathfrak{e}(\boldsymbol{\tau})-1)
    +
    C,
\end{equation}
up to $\boldsymbol{\tau}$-independent terms which are all absorbed into $C$. 
Differentiating \eqref{log t2k main terms} twice and evaluating at $\boldsymbol{\tau}= \boldsymbol{0}$ we obtain
\[
\begin{aligned}
    \frac{\partial^2}{\partial\tau_a\partial\tau_b}
    \log t_{2k} (\boldsymbol{\alpha}_{\boldsymbol{\tau}})
    \bigg|_{\boldsymbol{\tau}=\boldsymbol{0}} &=  \frac{1}{\alpha_{2k}}
    \frac{\partial^2}{\partial \tau_a\partial\tau_b}
    (\boldsymbol{\alpha}_{\boldsymbol{\tau}})_{2k} \bigg|_{\boldsymbol{\tau}=\boldsymbol{0}}
     -
    \frac{1}{\alpha_{2k}^2}
    \frac{\partial}{\partial \tau_a}
    (\boldsymbol{\alpha}_{\boldsymbol{\tau}})_{2k}
    \bigg|_{\boldsymbol{\tau}=\boldsymbol{0}}
    \frac{\partial}{\partial \tau_b}
    (\boldsymbol{\alpha}_{\boldsymbol{\tau}})_{2k}
    \bigg|_{\boldsymbol{\tau}=\boldsymbol{0}}
    \\
    &
    -
    k
    \left[    \frac{1}{\mathfrak{e}}
        \frac{\partial^2 \mathfrak{e}(\boldsymbol{\tau})}
        {\partial \tau_a\partial\tau_b}
        \bigg|_{\boldsymbol{\tau}=\boldsymbol{0}}     -
        \frac{1}{\mathfrak{e}^2}       \frac{\partial\mathfrak{e}(\boldsymbol{\tau})}{\partial\tau_a}
        \bigg|_{\boldsymbol{\tau}=\boldsymbol{0}}
        \frac{\partial\mathfrak{e}(\boldsymbol{\tau})}{\partial\tau_b}
        \bigg|_{\boldsymbol{\tau}=\boldsymbol{0}} \right]
    \\
    &\quad
    +
    (k-1)
    \left[
        \frac{1}{\mathfrak{e}-1}
        \frac{\partial^2 \mathfrak{e}(\boldsymbol{\tau})}
        {\partial \tau_a\partial\tau_b}
        \bigg|_{\boldsymbol{\tau}=\boldsymbol{0}}
        -
        \frac{1}{(\mathfrak{e}-1)^2}
        \frac{\partial\mathfrak{e}(\boldsymbol{\tau})}{\partial\tau_a}
        \bigg|_{\boldsymbol{\tau}=\boldsymbol{0}}
        \frac{\partial\mathfrak{e}(\boldsymbol{\tau})}{\partial\tau_b}
        \bigg|_{\boldsymbol{\tau}=\boldsymbol{0}}
    \right].
\end{aligned}
\]
Recalling \eqref{Hab G}, note that $\displaystyle     \frac{\partial^2 \mathcal{L}}{\partial \tau_a\partial\tau_b}(\boldsymbol{0})
    =
    \sum_{k=1}^p
    \alpha_{2k}
    \frac{\partial^2}{\partial \tau_a\partial\tau_b}
    \log t_{2k}(\boldsymbol{\alpha}_{\boldsymbol{\tau}})
    \bigg|_{\boldsymbol{\tau}=\boldsymbol{0}}$. Multiplying the preceding equation by \(\alpha_{2k}\) and summing over
\(k\), we observe that the first term on the right-hand side vanishes,
while the third and fifth terms cancel. The remaining terms give the
right-hand side of \eqref{d^2 log}.
\end{proof}

We will also make use of the following identity from elementary linear algebra:
\begin{lemma}\label{lin alg}
Let \(\mb E\) be an invertible symmetric \(n\times n\) matrix, let
\(\boldsymbol \nu\in\mathbb R^n\setminus\{\mb 0\}\), and let \(\mb T\) be an
\(n\times(n-1)\) matrix whose columns form an orthonormal basis of
\(\boldsymbol \nu^\perp\). Then
\[
    \det(\mb T^T\mb E\mb T)
    =
    \det \mb E\,
    \frac{\boldsymbol \nu^T \mb E^{-1} \boldsymbol \nu}{\|\boldsymbol \nu\|^2}.
\]
\end{lemma}

\begin{proof}
Although the proof is standard, we present the result for completeness and the convenience of the reader. Let
\(
    \mb e:=\frac{\boldsymbol \nu}{\|\boldsymbol \nu\|}
\) be the normal unit vector. Since the columns of \(\mb T\) form an orthonormal basis of \(\boldsymbol \nu^\perp\), and
\(\mb e\) is a unit vector perpendicular to every column of \(\mb T\), the \(n\times n\)
matrix
\[
    \mb Q:=\begin{bmatrix} \mb T & \mb e \end{bmatrix}
\]
is orthogonal. In other words,
\(
    \mb Q^T\mb Q=\I_n,
\)
and hence \(
    \det \mb Q=\pm 1.
\) Now consider the matrix
\[
   \mb M := \mb Q^T \mb E \mb Q.
\]
Since \(\mb Q\) is orthogonal,
\(
    \det(\mb M)
    =
    \det \mb E.
\) On the other hand, using the block form of \(\mb Q\), we have
\[
    \mb M
    =
    \begin{bmatrix}
        \mb T^T \mb E \mb T & \mb T^T \mb E \mb e \\
        \mb e^T \mb E\mb T & \mb e^T\mb E\mb e
    \end{bmatrix}.
\]

Now use the cofactor formula for the inverse of an invertible matrix:
\[
    \mb M^{-1}
    =
    \frac{\operatorname{adj}(\mb M)}{\det\mb M}.
\]
In particular, the cofactor of the $(n,n)$ entry of $\mb M$ gives
\[
    [\mb M^{-1}]_{nn}
    =
    \frac{\det(\mb T^T\mb E\mb T)}{\det \mb M}.
\]
Equivalently,
\[
    \det(\mb T^T \mb E \mb T)
    =
    \det \mb M\,[\mb M^{-1}]_{nn} =
    \det \mb E\,[\mb M^{-1}]_{nn} = (\det \mb E )    \mb e^T\mb E^{-1}\mb e.
\]
Recalling
\(
    \mb e=\frac{\boldsymbol \nu}{\|\boldsymbol \nu\|},
\)
we get the desired identity.
\end{proof}  

\begin{lemma}
   Let $\boldsymbol{\mathcal{E}}$ be a  \(p\times p\) symmetric matrix defined by
\[
    \mathcal E_{k\ell}
    :=
    \frac{\alpha_{2k}}{t_{2k}(\boldsymbol{\alpha})^2}\delta_{k\ell}
    -
    \frac{k\ell\,\alpha_{2k}\alpha_{2\ell}}
    {
        (\mathfrak{z}-\mathfrak{e})
        t_{2k}(\boldsymbol{\alpha})t_{2\ell}(\boldsymbol{\alpha})
    },
    \qquad
    k,\ell=1,\ldots,p.
\]  We have
  \begin{equation}
       \det\boldsymbol{\mathcal H}
    =
    (-1)^{p-1}
    \det\left(({\bf T}_a)^T\boldsymbol{\mathcal E}{\bf T}_b\right)_{a,b=1}^{p-1},
  \end{equation}
 where ${\bf T}_a$ and ${\bf T}_b$ are given by \eqref{Ta}.
\end{lemma}
 \begin{proof}
            We first rewrite the formula \eqref{d^2 log} in terms of the tangent vectors \({\bf T}_a\) and \({\bf T}_b\). Differentiating \eqref{log t2k main terms} once with respect to \(\tau_a\), and then evaluating at
\(\boldsymbol{\tau}=0\), gives

\begin{equation}
    \frac{\partial}{\partial \tau_a}
    \log t_{2k}(\boldsymbol{\alpha}_{\boldsymbol{\tau}})
    \bigg|_{\boldsymbol{\tau}=\boldsymbol{0}}
    =
    \frac{1}{\alpha_{2k}}
    \frac{\partial}{\partial \tau_a}
    [\boldsymbol{\alpha}_{\boldsymbol{\tau}}]_{2k}
    \bigg|_{\boldsymbol{\tau}=\boldsymbol{0}}
     +
    \frac{k-\mathfrak{e}}{\mathfrak{e}(\mathfrak{e}-1)}
    \sum_{j=1}^p
    j
    \frac{\partial}{\partial \tau_a}
    [\boldsymbol{\alpha}_{\boldsymbol{\tau}}]_{2j}
    \bigg|_{\boldsymbol{\tau}=\boldsymbol{0}}.
\end{equation}
Notice that the left-hand-side is precisely $[{\bf T}_a]_k/t_{2k}(\boldsymbol{\alpha})$, so we have
\begin{equation}\label{Ta over t2k}
    \frac{[{\bf T}_a]_k}{t_{2k}(\boldsymbol{\alpha})} = 
    \frac{1}{\alpha_{2k}}
    \frac{\partial}{\partial \tau_a}
    [\boldsymbol{\alpha}_{\boldsymbol{\tau}}]_{2k}
    \bigg|_{\boldsymbol{\tau}=\boldsymbol{0}}
     +
    \frac{k-\mathfrak{e}}{\mathfrak{e}(\mathfrak{e}-1)}
    \sum_{j=1}^p
    j
    \frac{\partial}{\partial \tau_a}
    [\boldsymbol{\alpha}_{\boldsymbol{\tau}}]_{2j}
    \bigg|_{\boldsymbol{\tau}=\boldsymbol{0}}.
\end{equation} and
\begin{equation}\label{Tb over t2k}
    \frac{[{\bf T}_b]_k}{t_{2k}(\boldsymbol{\alpha})} = 
    \frac{1}{\alpha_{2k}}
    \frac{\partial}{\partial \tau_b}
    [\boldsymbol{\alpha}_{\boldsymbol{\tau}}]_{2k}
    \bigg|_{\boldsymbol{\tau}=\boldsymbol{0}}
     +
    \frac{k-\mathfrak{e}}{\mathfrak{e}(\mathfrak{e}-1)}
    \sum_{j=1}^p
    j
    \frac{\partial}{\partial \tau_b}
    [\boldsymbol{\alpha}_{\boldsymbol{\tau}}]_{2j}
    \bigg|_{\boldsymbol{\tau}=\boldsymbol{0}}.
\end{equation}

Multiplying \eqref{Ta over t2k} by $k\alpha_{2k}$ and summing over $k$ yields
\begin{equation}
     \sum_{k=1}^p
    k\alpha_{2k}
    \frac{[{\bf T}_a]_k}{t_{2k}(\boldsymbol{\alpha})}
    =
    \frac{\mathfrak{z}-\mathfrak{e}}{\mathfrak{e}(\mathfrak{e}-1)}
    \sum_{k=1}^p
    k
    \frac{\partial}{\partial \tau_a}
    [\boldsymbol{\alpha}_{\boldsymbol{\tau}}]_{2k}
    \bigg|_{\boldsymbol{\tau}=\boldsymbol{0}},
\end{equation}
where we have used $1+\frac{\mathfrak{z}-\mathfrak{e}^2}{\mathfrak{e}(\mathfrak{e}-1)} = \frac{\mathfrak{z}-\mathfrak{e}}{\mathfrak{e}(\mathfrak{e}-1)}$. Similarly \begin{equation}
    \sum_{k=1}^p
    k\alpha_{2k}
    \frac{[{\bf T}_b]_k}{t_{2k}(\boldsymbol{\alpha})}
    =
    \frac{\mathfrak{z}-\mathfrak{e}}{\mathfrak{e}(\mathfrak{e}-1)}
    \sum_{k=1}^p
    k
    \frac{\partial}{\partial \tau_b}
    [\boldsymbol{\alpha}_{\boldsymbol{\tau}}]_{2k}
    \bigg|_{\boldsymbol{\tau}=\boldsymbol{0}}.
\end{equation}
The product of the last two expressions resembles the second term on the right hand side of \eqref{d^2 log}. To this end consider
\[
\begin{aligned}
    &\frac{1}{\mathfrak{z}-\mathfrak{e}}
    \left(
        \sum_{k=1}^p
        k\alpha_{2k}
        \frac{[{\bf T}_a]_k}{t_{2k}(\boldsymbol{\alpha})}
    \right)
    \left(
        \sum_{k=1}^p
        k\alpha_{2k}
        \frac{[{\bf T}_b]_k}{t_{2k}(\boldsymbol{\alpha})}
    \right)
    \\
    &\qquad
    =
    \frac{\mathfrak{z}-\mathfrak{e}}{\mathfrak{e}^2(\mathfrak{e}-1)^2}
    \left(
        \sum_{j=1}^p
        j
        \frac{\partial}{\partial \tau_a}
        [\boldsymbol{\alpha}_{\boldsymbol{\tau}}]_{2j}
        \bigg|_{\boldsymbol{\tau}=\boldsymbol{0}}
    \right)
    \left(
        \sum_{\ell=1}^p
        \ell
        \frac{\partial}{\partial \tau_b}
        [\boldsymbol{\alpha}_{\boldsymbol{\tau}}]_{2\ell}
        \bigg|_{\boldsymbol{\tau}=\boldsymbol{0}}
    \right).
\end{aligned}
\] 

On the other hand, the product of \eqref{Ta over t2k} and \eqref{Tb over t2k} also produces terms similar to those on the right hand side of \eqref{d^2 log}. Indeed, multiplying the right-hand sides of \eqref{Ta over t2k} and
\eqref{Tb over t2k}, weighting by \(\alpha_{2k}\), and summing over \(k\)
gives, after simplification,
\[
\begin{aligned} \sum_{k=1}^p \alpha_{2k}
    \frac{[{\bf T}_a]_k}{t_{2k}(\boldsymbol{\alpha})}
    \frac{[{\bf T}_b]_k}{t_{2k}(\boldsymbol{\alpha})} &=
    \sum_{k=1}^p    \frac{1}{\alpha_{2k}} \frac{\partial}{\partial \tau_a}    [\boldsymbol{\alpha}_{\boldsymbol{\tau}}]_{2k}
    \bigg|_{\boldsymbol{\tau}=\boldsymbol{0}}
    \frac{\partial}{\partial \tau_b}
    [\boldsymbol{\alpha}_{\boldsymbol{\tau}}]_{2k}
    \bigg|_{\boldsymbol{\tau}=\boldsymbol{0}}
    \\
    &    + \frac{\mathfrak{e}^2-2\mathfrak{e}+\mathfrak{z}}{\mathfrak{e}^2(\mathfrak{e}-1)^2} \left(  \sum_{j=1}^p j  \frac{\partial}{\partial \tau_a}
        [\boldsymbol{\alpha}_{\boldsymbol{\tau}}]_{2j}
        \bigg|_{\boldsymbol{\tau}=\boldsymbol{0}}    \right)
    \left(
        \sum_{\ell=1}^p  \ell  \frac{\partial}{\partial \tau_b}   [\boldsymbol{\alpha}_{\boldsymbol{\tau}}]_{2\ell}   \bigg|_{\boldsymbol{\tau}=\boldsymbol{0}}    \right).
\end{aligned}
\]
Now, subtract the previous two equations and simplify to obtain
\[
\begin{aligned}
    &\sum_{k=1}^p
    \alpha_{2k}
    \frac{[{\bf T}_a]_k}{t_{2k}(\boldsymbol{\alpha})}
    \frac{[{\bf T}_b]_k}{t_{2k}(\boldsymbol{\alpha})}
     -    \frac{1}{\mathfrak{z}-\mathfrak{e}} \left(
        \sum_{k=1}^p   k\alpha_{2k}  \frac{[{\bf T}_a]_k}{t_{2k}(\boldsymbol{\alpha})} \right)
    \left( \sum_{k=1}^p k\alpha_{2k}  \frac{[{\bf T}_b]_k}{t_{2k}(\boldsymbol{\alpha})} \right) \\ &= \sum_{k=1}^p
 \frac{1}{\alpha_{2k}}    \frac{\partial}{\partial \tau_a} [\boldsymbol{\alpha}_{\boldsymbol{\tau}}]_{2k}
  \bigg|_{\boldsymbol{\tau}=\boldsymbol{0}}    \frac{\partial}{\partial \tau_b}  [\boldsymbol{\alpha}_{\boldsymbol{\tau}}]_{2k}
 \bigg|_{\boldsymbol{\tau}=\boldsymbol{0}}     +\frac{1}{\mathfrak{e}(\mathfrak{e}-1)} \left(   \sum_{j=1}^p  j \frac{\partial}{\partial \tau_a}  [\boldsymbol{\alpha}_{\boldsymbol{\tau}}]_{2j}  \bigg|_{\boldsymbol{\tau}=\boldsymbol{0}}  \right)  \left( \sum_{\ell=1}^p   \ell \frac{\partial}{\partial \tau_b}   [\boldsymbol{\alpha}_{\boldsymbol{\tau}}]_{2\ell}
  \bigg|_{\boldsymbol{\tau}=\boldsymbol{0}}\right).
\end{aligned}
\] Hence, comparing with \eqref{Hab d^2 log} and \eqref{d^2 log}, we have arrived at an expression for $\mathcal{H}_{ab}$ in terms of \({\bf T}_a\) and \({\bf T}_b\):
\begin{equation}
        \mathcal H_{ab} =    -   \sum_{k=1}^p \alpha_{2k} \frac{[{\bf T}_a]_k}{t_{2k}(\boldsymbol{\alpha})}   \frac{[{\bf T}_b]_k}{t_{2k}(\boldsymbol{\alpha})} +    \frac{1}{\mathfrak{z}-\mathfrak{e}}\left( \sum_{k=1}^p k\alpha_{2k} \frac{[{\bf T}_a]_k}{t_{2k}(\boldsymbol{\alpha})} \right)
    \left( \sum_{k=1}^p   k\alpha_{2k}  \frac{[{\bf T}_b]_k}{t_{2k}(\boldsymbol{\alpha})}    \right).
\end{equation}
To this end, introduce the \(p\times p\) symmetric matrix
\[
    \mathcal E_{k\ell}
    :=    \frac{\alpha_{2k}}{t_{2k} (\boldsymbol{\alpha})^2}\delta_{k\ell}    -   \frac{k\ell\,\alpha_{2k}\alpha_{2\ell}}{ (\mathfrak{z}-\mathfrak{e}) t_{2k}(\boldsymbol{\alpha})t_{2\ell}(\boldsymbol{\alpha}) },    \qquad
    k,\ell=1,\ldots,p.
\] It is a straightforward calculation to observe that
\begin{equation}\label{Hab Ta E Tb}
        \mathcal H_{ab}   = -({\bf T}_a)^T\boldsymbol{\mathcal E}{\bf T}_b,
\end{equation}
from which the result follows.
\end{proof}
\begin{proposition}\label{prop:determinant-formula} We have that
    \begin{equation}
        \det\boldsymbol{\mathcal{H}} = (-1)^{p+1}\frac{(\mathfrak{e}-1)\prod_{k=1}^p\alpha_{2k}}{\mathfrak{e} (\mathfrak{z}-\mathfrak{e})||\boldsymbol{\nu}(\boldsymbol{\alpha})||^2\prod_{k=1}^{p}t_{2k}(\boldsymbol{\alpha})^2}.
    \end{equation} 
\end{proposition}
\begin{proof}
 If
\(\boldsymbol{\mathcal E}\) is invertible (to be checked below) and symmetric, and if \({\bf T}_1,\ldots,{\bf T}_{p-1}\) form an orthonormal basis of \(\boldsymbol{\nu}(\boldsymbol{\alpha})^\perp\), then by Lemma \ref{lin alg}

\begin{equation}\label{TatETb}
 \det\left(({\bf T}_a)^T\boldsymbol{\mathcal E}{\bf T}_b\right)_{a,b=1}^{p-1}
    =    \det\boldsymbol{\mathcal E}\, \frac{ \boldsymbol{\nu}(\boldsymbol{\alpha})^T   \boldsymbol{\mathcal E}^{-1} \boldsymbol{\nu}(\boldsymbol{\alpha})}
    {\|\boldsymbol{\nu}(\boldsymbol{\alpha})\|^2}.   
\end{equation}

We now compute $\det \boldsymbol{\mathcal{E}}$ and justify that $\boldsymbol{\mathcal{E}}$ is invertible. The matrix $\boldsymbol{\mathcal{E}}$ can be written as
\[     \boldsymbol{\mathcal E}
    =
    \operatorname{diag}
    \left(
        \frac{\alpha_{2k}}{t_{2k}(\boldsymbol{\alpha})^2}
    \right)_{k=1}^p -   \frac{1}{\mathfrak{z}-\mathfrak{e}}
    \left(    \frac{k\alpha_{2k}}{t_{2k}(\boldsymbol{\alpha})}
    \right)_{k=1}^p \left(    \frac{k\alpha_{2k}}{t_{2k}(\boldsymbol{\alpha})}   \right)_{k=1}^{p,T}. \]  
    Recall that if \(\mb A\) is an invertible square matrix
and \(\mb u,\mb v\) are column vectors, then
\[
    \det(\mb A+\mb u \mb v^T) = \det(\mb A)\left(1+\mb v^T \mb A^{-1} \mb u\right).
\] Applying this identity to $\boldsymbol{\mathcal{E}}$, after straightforward simplification, gives
\begin{equation}\label{det E}
    \det \boldsymbol{\mathcal E} = - \frac{\mathfrak{e}}{\mathfrak{z}-\mathfrak{e}} \prod_{k=1}^{p} \frac{\alpha_{2k}}{t_{2k}(\boldsymbol{\alpha})^2}.
\end{equation}
This quantity is nonzero, so recalling \eqref{TatETb}, it only remains to compute $\boldsymbol{\mathcal{E}}^{-1}$. To that end, we use the Sherman--Morrison formula \cite{ShermanMorrison1950}: if \(\mb A\) is an invertible
\(n\times n\) matrix and \(\mb u, \mb v\in\mathbb R^n\) satisfy \(
    1+\mb v^T \mb A^{-1} \mb u\neq 0,
\)
then
\[
    (\mb A+ \mb u \mb v^T)^{-1}
    =
    \mb A^{-1}
    -
    \frac{\mb A^{-1} \mb u \mb v^T \mb A^{-1}}{1+\mb v^T \mb A^{-1}\mb u}.
\] 
Applying this identity to $\boldsymbol{\mathcal{E}}$ gives
\begin{equation}
        \boldsymbol{\mathcal E}^{-1}
    =
    \operatorname{diag}
    \left(
        \frac{t_{2k}(\boldsymbol{\alpha})^2}{\alpha_{2k}}
    \right)_{k=1}^p
    -
    \frac{1}{\mathfrak{e}}
    \operatorname{diag}
    \left(
        \frac{t_{2k}(\boldsymbol{\alpha})^2}{\alpha_{2k}}
    \right)_{k=1}^p
    \left(
        \frac{k\alpha_{2k}}{t_{2k}(\boldsymbol{\alpha})}
    \right)_{k=1}^p
    \left(
        \frac{k\alpha_{2k}}{t_{2k}(\boldsymbol{\alpha})}
    \right)_{k=1}^{p,T}
    \operatorname{diag}
    \left(
        \frac{t_{2k}(\boldsymbol{\alpha})^2}{\alpha_{2k}}
    \right)_{k=1}^p
\end{equation}
Now we compute $\boldsymbol{\nu}(\boldsymbol{\alpha})^T   \mathcal E^{-1} \boldsymbol{\nu}(\boldsymbol{\alpha})$. The contribution from the first term in $\boldsymbol{\mathcal{E}}^{-1}$ is
\begin{equation}
    \boldsymbol{\nu}(\boldsymbol{\alpha})^T
    \operatorname{diag}
    \left(
        \frac{t_{2k}(\boldsymbol{\alpha})^2}{\alpha_{2k}}
    \right)_{k=1}^p
    \boldsymbol{\nu}(\boldsymbol{\alpha})
    =
    \sum_{k=1}^p
    \nu_{2k}(\boldsymbol{\alpha})^2
    \frac{t_{2k}(\boldsymbol{\alpha})^2}{\alpha_{2k}}
    =
    \sum_{k=1}^p
    \frac{\alpha_{2k}}{\mathfrak{e}^2}
    =
    \frac{1}{\mathfrak{e}^2}.
\end{equation}
Regarding the second term observe that
\[   \boldsymbol{\nu}(\boldsymbol{\alpha})^T
        \operatorname{diag}
    \left(
        \frac{t_{2k}(\boldsymbol{\alpha})^2}{\alpha_{2k}}
    \right)_{k=1}^p
    \left(
        \frac{k\alpha_{2k}}{t_{2k}(\boldsymbol{\alpha})}
    \right)_{k=1}^p
    \left(
        \frac{k\alpha_{2k}}{t_{2k}(\boldsymbol{\alpha})}
    \right)_{k=1}^{p,T}
    \operatorname{diag}
    \left(
        \frac{t_{2k}(\boldsymbol{\alpha})^2}{\alpha_{2k}}
    \right)_{k=1}^p
    \boldsymbol{\nu}(\boldsymbol{\alpha}) \equiv \mb V \mb V^T, \] where
    \begin{equation}
        \begin{split}
           \mb V&:= \boldsymbol{\nu}(\boldsymbol{\alpha})^T
        \operatorname{diag}
    \left(
        \frac{t_{2k}(\boldsymbol{\alpha})^2}{\alpha_{2k}}
    \right)_{k=1}^p
    \left(
        \frac{k\alpha_{2k}}{t_{2k}(\boldsymbol{\alpha})}
    \right)_{k=1}^p
    =
    \sum_{k=1}^p
    k\nu_{2k}(\boldsymbol{\alpha})t_{2k}(\boldsymbol{\alpha})
    =
    -\frac{1}{\mathfrak{e}}
    \sum_{k=1}^p k\alpha_{2k}
    =
    -1. 
        \end{split}
    \end{equation}
    Therefore 
    \begin{equation}
        \boldsymbol{\nu}(\boldsymbol{\alpha})^T   \mathcal E^{-1} \boldsymbol{\nu}(\boldsymbol{\alpha})= \frac{1}{\mathfrak{e}^2} - \frac{1}{\mathfrak{e}}.
    \end{equation}
Combining this with \eqref{Hab Ta E Tb}, \eqref{TatETb}, \eqref{det E} yields:
  \begin{equation}
        \det\boldsymbol{\mathcal{H}} = (-1)^{p+1}\frac{(\mathfrak{e}-1)\prod_{k=1}^p\alpha_{2k}}{\mathfrak{e} (\mathfrak{z}-\mathfrak{e})||\boldsymbol{\nu}(\boldsymbol{\alpha})||^2\prod_{k=1}^{p}t_{2k}(\boldsymbol{\alpha})^2}.
    \end{equation}        \end{proof}


\subsubsection{ACSV: Contour deformation and asymptotics}
Let us recall the methods of analytic combinatorics in several variables (ACSV). The basic principle is that the singularity set $\mathcal{V}$ of the integrand (here, the points where $P_{\sigma}(\sigma({\bf t}),{\bf t}) = 0$) control the asymptotics of $I_V$ \eqref{eq:def-IV-INTEGRAL}. More concretely, let
    \begin{equation}
        \mathcal{V}:=\{{\bf t}\mid P_{\sigma}(\sigma({\bf t}),{\bf t}) = 0\}.
    \end{equation}
Following \cite[Definition 6.34]{PW}, we say a point ${\bf s}\in \mathcal{V}$ is \textit{minimal} if there does not exist ${\bf u} \in \mathbb{C}^p$ such that $P_{\sigma}(\sigma({\bf u}),{\bf u}) = 0$ with $|u_{2k}|\leq |s_{2k}|$, $k=1,...,p$, with one of these inequalities being strict. Minimal points are the points which lie on the closure of the domain of convergence of $\frac{Q}{P^{\Delta}}$. We call ${\bf s}\in \mathcal{V}$ \textit{strictly minimal} if it is minimal, and there are no other singularities with the same modulus (taken coordinate-wise). Observe that $\Sigma_{cr}$ as defined in \eqref{eq:critical-surface} is a subset of $\mathcal{V}$; we shall soon see that it is precisely the component of $\mathcal{V}$ relevant to our analysis.

The following proposition follows immediately from the results of \cite[Section 3]{PW3} (see also \cite[Lemma 6.41]{PW}).
\begin{proposition}
    Every minimal point ${\bf s}\in \mathcal{V}$ is strictly minimal, and lies in the negative orthant: i.e. every minimal point ${\bf s}$ satisfies $s_{2k}<0$.
    \label{prop:minimal-unique}
\end{proposition}
\begin{proof}
    See \cite[Proposition 3.17]{PW3}. We remark here that \textit{positive orthant} in their proposition has been replaced by \textit{negative orthant} here due to the fact that our generating function has positive coefficients in the variables $-\frac{s_{2k}}{2k}$.
\end{proof}

\begin{proposition}
    For given $\boldsymbol{\alpha}$ be a vector of real numbers satisfying $\sum_{k=1}^p\alpha_{2k} = 1$. Then there is a unique minimal point ${\bf t}(\boldsymbol{\alpha}) \in \mathcal{V}$ which is given by Equation \eqref{t2k} with $\boldsymbol{\beta} = \boldsymbol{\alpha}$.
\end{proposition}
\begin{proof}
    Uniqueness of this minimal point follows from Proposition \ref{prop:minimal-unique}, and so we only have to show that the point ${\bf t}(\boldsymbol{\alpha})$ is minimal. Indeed, observe that $\Sigma_{cr}$ is a subset of the boundary of the Reinhardt domain of convergence for functions of the form $Q/P_{\sigma}^{\Delta}$, since $\sigma$ is a solution to the algebraic equation $P(\sigma;{\bf t}) = \varkappa$. Furthermore, one can show that the height function
        \begin{equation}
            h_{\boldsymbol{\alpha}}(\boldsymbol{\beta})\bigg|_{\mathcal{V}} := - \sum_{k=1}^p \alpha_{2k}\log|t_{2k}(\beta_{2k})|
        \end{equation}
    has an extremal point at $\boldsymbol{\beta} = \boldsymbol{\alpha}$, which implies $\boldsymbol{\alpha}$ is minimal.
\end{proof}

We now proceed to our main asymptotic proposition.
In what follows, will make use of the following well-known formula for the $\Gamma$-function (see \cite[Eq. 5.9.2]{DLMF} ):
        \begin{equation}\label{eq:Gamma-Hankel}
            \frac{1}{\Gamma(z)} = \frac{1}{2\pi \ii}\int_{L}s^{-z} \ee^s \dd s, 
        \end{equation}
    where $L$ is a contour starting at $-\infty + \ii\delta$, traversing clockwise around the origin, and returning to $\infty$ along the direction $-\infty -\ii \delta$, which is valid for all $z\in \mathbb{C}$.

\begin{proposition}\label{prop:asymptotic-formula}
    Recall the definition of $I_V(Q,\boldsymbol{\alpha},\Delta)$ in \eqref{eq:def-IV-INTEGRAL}. Suppose $Q(\sigma({\bf t});{\bf t})$ satisfies $Q(\sigma({\bf t}),{\bf t}) = q_a({\bf t}) + q_b({\bf t})\xi({\bf t})^{1/2}$, where $q_a(\mb t),q_b(\mb t)$ are analytic in a neighborhood of ${\bf t} = {\bf t}(\boldsymbol{\alpha})$, and $q_a({\bf t}(\boldsymbol{\alpha})) = Q_0$. Then:
    \begin{equation}
        I_V(Q,\boldsymbol{\alpha},\Delta) = \frac{Q_0(\mathfrak{e}||\boldsymbol{\nu}(\boldsymbol{\alpha})||)^{\Delta/2-1}}{\Gamma(\frac{\Delta}{2})g({\bf t}(\boldsymbol{\alpha}))^{\Delta}}||\boldsymbol{\nu}(\boldsymbol{\alpha})||\sqrt{\frac{\mathfrak{e}(\mathfrak{z}-\mathfrak{e})}{\mathfrak{e}-1} }V^{\frac{\Delta}{2}-1} \ee^{V\Omega(\boldsymbol{\alpha};\varkappa)}(1 + \mathcal{O}(V^{-1/2})),
    \end{equation}
where $g({\bf t}(\boldsymbol{\alpha}))$ is as in Equation \eqref{eq:g-var}, $\boldsymbol{\mathcal{H}}$ is the Hessian matrix as defined in Proposition \ref{Phi-structure-prop}, and
    \begin{equation}
        \Omega(\boldsymbol{\alpha};\varkappa) := \mathfrak{e}\log\mathfrak{e} -(\mathfrak{e}-1)\log(\mathfrak{e}-1) + \sum_{k=1}^p\alpha_{2k}\log \left( k\binom{2k}{k} \right)+(\mathfrak{e}-1)\log\varkappa.
    \end{equation}
\end{proposition}
\begin{proof}
    Let us study $\ee^{-V\Omega(\boldsymbol{\alpha};\varkappa)}I_V$.
    Define the following surface $S$: in the coordinates $t_{2k}$ outside of a small neighborhood of ${\bf t} = {\bf t}(\boldsymbol{\alpha})$, we define $S := S_{0}\cup S_{loc}$, where $S_{0}$ is the part of the torus
        \begin{equation*}
            |t_{2k}| = \rho_{2k}, \qquad |\arg \,t_{2k}| <\pi - \delta_1,
        \end{equation*}
    where $\rho_{2k}>|t_{2k}(\boldsymbol{\alpha})|^{\alpha_{2k}}$,
    and $S_{loc}$ is defined in a small neighborhood of ${\bf t} = {\bf t}(\boldsymbol{\alpha})$ in the local coordinates $(\xi,\boldsymbol{\tau})$ to be homotopic to
        \begin{align*}
            \xi:&\qquad\qquad \{t+\ii\delta\mid -L<t<\delta'\} \cup \{ \delta' + \ii t \delta\mid -1<t<1\} \cup \{t-\ii\delta\mid -L<t<\delta'\},\\
            \tau_a:& \qquad\qquad \{\ii x_a\mid -L<x_a<L\},
        \end{align*}
    where $L,\delta,\delta'>0$. By our freedom in our choice of homotopy, we can arrange so that the surface $S$ is connected. Because ${\bf t}(\boldsymbol{\alpha})$ is minimal, we have that the integrand of $I_V$ is analytic on $S_0$, and so
        \begin{align*}
            |\ee^{-V\Omega(\boldsymbol{\alpha};\varkappa)}I_V| &= \left|\ee^{-V\Omega(\boldsymbol{\alpha};\varkappa)}\frac{C(\boldsymbol{\alpha})}{(2\pi \ii)^{p}}\oint_{\mathbb{T}_{\delta}} \frac{Q(\sigma({\bf t});{\bf t})}{P_{\sigma}(\sigma({\bf t});{\bf t})^{\Delta}} \ee^{V\Phi({\bf t})} \prod_{k=1}^{p}\frac{\dd t_{2k}}{t_{2k}}\right|\\
            &=\left|\ee^{-V\Omega(\boldsymbol{\alpha};\varkappa)}\frac{C(\boldsymbol{\alpha})}{(2\pi \ii)^{p}}\oint_{S_0\cup S_{loc}} \frac{Q(\sigma({\bf t});{\bf t})}{P_{\sigma}(\sigma({\bf t});{\bf t})^{\Delta}} \ee^{V\Phi({\bf t})} \prod_{k=1}^{p}\frac{\dd t_{2k}}{t_{2k}}\right|\\
            &\leq\left|\ee^{-V\Omega(\boldsymbol{\alpha};\varkappa)}\frac{C(\boldsymbol{\alpha})}{(2\pi \ii)^{p}}\oint_{S_{loc}} \frac{Q(\sigma({\bf t});{\bf t})}{P_{\sigma}(\sigma({\bf t});{\bf t})^{\Delta}} \ee^{V\Phi({\bf t})} \prod_{k=1}^{p}\frac{\dd t_{2k}}{t_{2k}}\right|\\
            & \qquad +\left|\ee^{-V\Omega(\boldsymbol{\alpha};\varkappa)}\frac{C(\boldsymbol{\alpha})}{(2\pi \ii)^{p}}\oint_{S_{0}} \frac{Q(\sigma({\bf t});{\bf t})}{P_{\sigma}(\sigma({\bf t});{\bf t})^{\Delta}} \ee^{V\Phi({\bf t})} \prod_{k=1}^{p}\frac{\dd t_{2k}}{t_{2k}}\right|\\
            &=\left|\ee^{-V\Omega(\boldsymbol{\alpha};\varkappa)}\frac{C(\boldsymbol{\alpha})}{(2\pi \ii)^{p}}\oint_{S_{loc}} \frac{Q(\sigma({\bf t});{\bf t})}{P_{\sigma}(\sigma({\bf t});{\bf t})^{\Delta}} \ee^{V\Phi({\bf t})} \prod_{k=1}^{p}\frac{\dd t_{2k}}{t_{2k}}\right| + \mathcal{O}(\ee^{-cV}),
        \end{align*}
    for some $c>0$, since the integrand is continuous on $S_0$, $\rho_{2k}>|t_{2k}(\boldsymbol{\alpha})|^{\alpha_{2k}}$, and by the definition of $\Omega(\boldsymbol{\alpha};\varkappa)$. Since we are only interested in the leading order asymptotics of $I_V$, from now on we by a slight abuse of notation
    we will simply write
        \begin{equation*}
            \ee^{-V\Omega(\boldsymbol{\alpha};\varkappa)}I_V = \ee^{-V\Omega(\boldsymbol{\alpha};\varkappa)}\frac{C(\boldsymbol{\alpha})}{(2\pi \ii)^{p}}\oint_{S_{loc}} \frac{Q(\sigma({\bf t});{\bf t})}{P_{\sigma}(\sigma({\bf t});{\bf t})^{\Delta}} e^{V\Phi({\bf t})} \prod_{k=1}^{p}\frac{\dd t_{2k}}{t_{2k}},
        \end{equation*}
    where it is understood that the equality in the above is equality up to exponentially small terms. Now, on $S_{loc}$, we can use the local coordinates $(\xi,\boldsymbol{\tau})$. We make the rescaling $\tau_a=\ii V^{-1/2}x_a$, $\xi=V^{-1}s$. Then, by our assumption on the form of $Q$,
    we have that
        \begin{align*}
            \frac{Q(\sigma({\bf t});{\bf t})}{P_{\sigma}(\sigma({\bf t});{\bf t})^{\Delta}} &= V^{\Delta/2}\frac{Q_0}{g({\bf t}(\boldsymbol{\alpha}))^{\Delta} s^{\Delta/2}}\left(1 + \mc O(V^{-1/2})\right),\\
            \prod_{k=1}^{p}\frac{\dd t_{2k}}{t_{2k}} &= V^{-\frac{p-1}{2}-1} \ii^{p-1}\frac{\dd s\prod_{a=1}^{p-1}\dd x_{2k}}{\prod_{k=1}^pt_{2k}(\boldsymbol{\alpha})}\left(1 + \mc O(V^{-1/2})\right),
        \end{align*}
    by Proposition \ref{Phi-structure-prop},
        \begin{align*}
            \ee^{V\left[\Phi({\bf t}) - \Phi({\bf t}(\boldsymbol{\alpha}))\right]} = \ee^{\mathfrak{e}||\boldsymbol{\nu}(\boldsymbol{\alpha})|| s + \frac{1}{2} \langle {\bf x}, \boldsymbol{\mathcal{H}} {\bf x} \rangle}(1 + \mc O(V^{-1/2})).
        \end{align*}
    Finally, observe that by Stirling's formula,
    \begin{equation*}
        \frac{\prod_{k=1}^p(\alpha_{2k}V)!}{V!} = C_0 V^{\frac{p-1}{2}} \ee^{V\psi(\boldsymbol{\alpha})}(1 + \mathcal{O}(V^{-1/2})),
    \end{equation*}
    where
    \begin{equation*}
        C_0 = (2\pi)^{\frac{p-1}{2}}\left(\prod_{k=1}^{p} \alpha_{2k}\right)^{1/2},\qquad \psi(\boldsymbol{\alpha}) = \sum_{k=1}^{p}\alpha_{2k}\log \alpha_{2k}.
    \end{equation*}
    Combining these observations, we obtain that
    \begin{align*}
        \ee^{-V\Omega(\boldsymbol{\alpha};\varkappa)}I_V &= \frac{Q_0V^{{\Delta/2}-1}(2\pi)^{\frac{1-p}{2}}\left(\prod_{k=1}^p\alpha_{2k}\right)^{1/2}}{g({\bf t}(\boldsymbol{\alpha}))^{\Delta} \prod_{k=1}^{p}t_{2k}(\boldsymbol{\alpha})}\\
        &\times \int_{\mathbb{R}^{p-1}}\frac{1}{2\pi i}\int_{L}s^{-\Delta/2}\ee^{\mathfrak{e}||\boldsymbol{\nu}(\boldsymbol{\alpha})|| s + \frac{1}{2} \langle {\bf x}, \boldsymbol{\mathcal{H}} {\bf x} \rangle} \dd s\prod_{a=1}^p \dd x_a\left(1+\mc O(V^{-1/2})\right),
    \end{align*}
    where we have replaced the finite integrals by integrals over $\mathbb{R}^{p-1}$ and the Hankel-type contour $L$ without loss of generality, as the integrals over these contours differ from the original ones by exponentially small terms. Using Equation \eqref{eq:Gamma-Hankel}, as well as the fact that $(-1)^p\boldsymbol{\mathcal{H}}$ is positive-definite, we obtain that
        \begin{equation*}
            \ee^{-V\Omega(\boldsymbol{\alpha};\varkappa)}I_V = \frac{Q_0V^{{\Delta/2}-1}(\mathfrak{e} ||\boldsymbol{\nu}(\boldsymbol{\alpha})||)^{\Delta/2-1}\left(\prod_{k=1}^p\alpha_{2k}\right)^{1/2}}{\Gamma\left(\frac{\Delta}{2}\right)g({\bf t}(\boldsymbol{\alpha}))^{\Delta}\sqrt{|\det\boldsymbol{\mathcal{H}}|} \prod_{k=1}^{p}t_{2k}(\boldsymbol{\alpha})}\left(1+\mc O(V^{-1/2})\right).
        \end{equation*}
    Finally, applying the result of Proposition \ref{prop:determinant-formula} yields the result.
\end{proof}

\begin{remark}
    Observe that $I_V = I_V(Q,\boldsymbol{\alpha},\Delta)$ depends on $Q$ only through the value $Q_0$. We therefore make the slight abuse of notation in what follows $I_V = I_V(Q_0,\boldsymbol{\alpha},\Delta)$, when there is no cause for ambiguity.
\end{remark}

\subsection{Extraction of free energy from recurrence coefficients}
We now must relate our results about asymptotics of Taylor coefficients of functions of the form studied in the previous section to the asymptotics of $F_g({\bf t})$. This is essentially accomplished by appealing to the following Lemma, which essentially follows directly from the work of Bleher and Its \cite{MR2187941}:

\begin{lemma}\label{Lemma:Bleher-Its}
    Define the differential operator
    \begin{equation}
        \mathfrak{Q} := \sum_{k=1}^{p} k t_{2k}\frac{\partial}{\partial t_{2k}}.
    \end{equation}
then
    \begin{equation}
        \mathfrak{Q}(\mathfrak{Q}+1)F_{n}({\bf t};N) + \frac{1}{2} = \frac{1}{4\varkappa^2}R_{n,N}(R_{n+1,N}+R_{n-1,N}).
    \end{equation}
\end{lemma}
\begin{proof}
    Let
    \begin{equation*}
        V(z;{\bf t}) := \frac{1}{2}z^2 + \sum_{k=1}^{p}\frac{t_{2k}}{2k}z^{2k},
    \end{equation*}
and consider the partition function
    \begin{equation*}
        \mc Z_{n, N}({\bf t}) := \int \prod_{1\leq i<j \leq n}(z_i-z_j)^2 \prod_{i=1}^n \exp\left(-NV(z_i;{\bf t})\right) \dd z_i.
    \end{equation*}
We make the change of variables
    \begin{equation*}
        t_{2k} = \mu^{-k}s_{2k},\qquad\qquad z_i = \sqrt{\mu}w_i.
    \end{equation*}
Under this change of variables, we find that
    \begin{equation*}
        \mc Z_{n, N}({\bf t}) =\mu^{n^2/2}\int \prod_{1\leq i<j \leq n}(w_i-w_j)^2 \prod_{i=1}^n \exp\left(-NV(w_i;{\bf s}) - N V_0(w_i;\mu)\right) \dd w_i,
    \end{equation*}
where $V_0(w;\mu) :=\frac{\mu-1}{2}w^2$. Observe that when $\mu = 1$, we return to the original partition function. The parameter $\mu$ varies $t_2$. Consider the family of orthogonal polynomials defined by
    \begin{equation*}
        \int p_n(z)p_m(z)\exp\left(-NV(z;{\bf s}) - NV_0(z;\mu)\right) \dd z = h_{n,N}\delta_{nm},
    \end{equation*}
which satisfy the recurrence $zp_{n}=p_{n+1}+R_{n,N}p_{n-1}$, $R_{n,N} = \frac{h_{n,N}}{h_{n-1,N}}$.
Then, up to an overall parameter-independent constant,
    \begin{equation*}
        \mc Z_{n, N}({\bf t}) = C_n\mu^{n^2/2}\prod_{k=0}^{n-1}h_{k,N}(\mb t).
    \end{equation*}
Let us calculate $\frac{\partial^2}{\partial \mu^2}\log {\mc Z}_{n,N}({\bf t})$:
    \begin{align*}
        \frac{\partial^2}{\partial \mu^2}\log \mc Z_{n, N}({\bf t}) &= -\frac{n^2}{2\mu^2} + \sum_{k=0}^{n-1}\frac{\partial^2}{\partial\mu^2}\log h_{k,N}(\mb t).
    \end{align*}
Observe the following identities:
    \begin{align*}
        \frac{\partial}{\partial \mu} \log h_{n,N}&= -\frac{N}{2}\frac{1}{h_{n,N}}\int z^2p_np_n \exp\left(-NV(z;{\bf s}) - NV_0(z;\mu)\right)\dd z = -\frac{n}{2\varkappa}(R_{n+1,N}+R_{n,N}),\\
        \frac{\partial}{\partial \mu}R_{n,N} &= \frac{h_{n-1}\frac{\partial}{\partial \mu} h_{n,N} - h_{n,N}\frac{\partial}{\partial \mu} h_{n-1,N}}{h_{n-1,N}^2}
        =-\frac{N}{2}R_{n,N}(R_{n+1,N}-R_{n-1,N}),
    \end{align*}
so that 
    \begin{align*}
        \sum_{k=0}^{n-1}\frac{\partial^2}{\partial\mu^2}\log h_{k,N} = \frac{N^2}{4}R_{n,N}(R_{n+1,N}+R_{n-1,N}),
    \end{align*}
and
    \begin{equation*}
        \frac{\partial^2}{\partial \mu^2}\log \mc Z_{n, N}({\bf t}) = -\frac{n^2}{2\mu^2} +  \frac{N^2}{4}R_{n,N}(\mb t)(R_{n+1,N}(\mb t)+R_{n-1,N}(\mb t)).
    \end{equation*}
If we convert back to the variables ${\bf t}$, and set $\mu = 1$, defining ${\mc F}_{n,N}({\bf t}) := \frac{1}{n^2}\log {\mc Z}_{n,N}({\bf t})$, we obtain the identity (recalling that $\varkappa = n /N$):
    \begin{equation*}
        \mathfrak{D} {\mc F}_{n,N}({\bf t}) + \frac{1}{2} =  \frac{1}{4\varkappa^2}R_{n,N}(\mb t)(R_{n+1,N}(\mb t)+R_{n-1,N}(\mb t)),
    \end{equation*}
where $\mathfrak{D}$ is the differential operator
    \begin{align*}
        \mathfrak{D} &:= \sum_{k,j=1}^p\left(kj t_{2k}t_{2j}\frac{\partial^2}{\partial t_{2k}\partial t_{2j}}\right) + \sum_{k=1}^{p} k(k+1)t_{2k}\frac{\partial}{\partial t_{2k}}\\
        &= \left(\sum_{k=1}^pk t_{2k}\frac{\partial}{\partial t_{2k}}\right)^2 - \sum_{k=1}^{p}k^2t_{2k}\frac{\partial}{\partial t_{2k}} + \sum_{k=1}^{p} k(k+1)t_{2k}\frac{\partial}{\partial t_{2k}}\\
        &=\left(\sum_{k=1}^pk t_{2k}\frac{\partial}{\partial t_{2k}}\right)^2+ \sum_{k=1}^{p} k t_{2k}\frac{\partial}{\partial t_{2k}}.
    \end{align*}
This completes the proof.
\end{proof}

The following Proposition is then immediate:
\begin{proposition}\label{prop:AtoF}
    Let 
        \begin{equation}
            \frac{1}{4\varkappa^2}R_{n,N}(\mb t)(R_{n+1,N}(\mb t)+R_{n-1,N}(\mb t))-\frac{1}{2} \sim \sum_{g=0}^{\infty} \frac{\mathcal{A}_g({\bf t})}{N^{2g}},
        \end{equation}
    and denote the genus $g$ free energy by $F_g({\bf t})$. 
    Then, for any fixed $g\geq 0$, and for $\boldsymbol{n}\neq 0$,
        \begin{equation}
            [{\bf t}^{\bf n}]F_g({\bf t}) = \frac{[{\bf t}^{\bf n}]\mathcal{A}_g({\bf t})}{E(E+1)},
        \end{equation}
    where $E = E({\bf n}) = \sum_{k=1}^{p}kn_{2k}$.
\end{proposition}
\begin{proof}
    We have that
    \begin{equation}
        \mathfrak{Q}(\mathfrak{Q}+1) \mc F_{n,N}({\bf t}) + \frac{1}{2} = \frac{1}{4\varkappa^2}R_{n,N}(\mb t)(R_{n+1,N}(\mb t)+R_{n-1,N}(\mb t)),
    \end{equation}
Where $\mathfrak{Q}$ is the differential operator
    \begin{equation}
        \mathfrak{Q} := \sum_{k=1}^{p} k t_{2k}\frac{\partial}{\partial t_{2k}}.
    \end{equation}
Since 
    \begin{equation*}
        \mc F_{n,N}({\bf t}) \sim \sum_{g=0}^{\infty} \frac{F_g({\bf t})}{N^{2g}},
    \end{equation*}
we have the equality, for each $g\geq 0$,
    \begin{equation}
        \mathfrak{Q}(\mathfrak{Q}+1)F_g({\bf t}) = \mathcal{A}_g({\bf t}).
    \end{equation}
Now, for any ${\bf n} \neq 0$,
    \begin{equation*}
        [{\bf t}^{{\bf n}}] (\mathfrak{Q} F_g)({\bf t}) = E({\bf n})[{\bf t}^{{\bf n}}]F_g({\bf t}).
    \end{equation*}
Thus, it follows that
    \begin{equation*}
       [{\bf t}^{{\bf n}}]\mathcal{A}_g({\bf t}) =  [{\bf t}^{{\bf n}}]\mathfrak{Q}(\mathfrak{Q}+1)F_g({\bf t}) = E({\bf n})(E({\bf n})+1) [{\bf t}^{{\bf n}}]F_g({\bf t}),
    \end{equation*}
which completes the proof.
\end{proof}
So, if we can study the singular structure of the functions $\mathcal{A}_g({\bf t})$, by applying the results of the previous section, we will be done. Indeed, we have that
\begin{proposition}\label{prop:A-structure}
    Let $\mathcal{A}_g({\bf t})$ be as in Proposition \ref{prop:AtoF}. There exists a polynomial $\tilde{q}(\sigma({\bf t}),{\bf t})$, such that
    \begin{equation}
        \mathcal{A}_g({\bf t}) = \frac{\tilde{q}(\sigma({\bf t}),{\bf t})}{{P}_{\sigma}(\sigma({\bf t});{\bf t})^{5g-1}}.
    \end{equation}
    Furthermore, for ${\bf t}$ in a sufficiently small neighborhood of ${\bf t}(\boldsymbol{\alpha})$,
    \begin{equation}
        \mathcal{A}_g({\bf t}) = \frac{\mathfrak{e}}{(\mathfrak{e}-1)\varkappa}\frac{\mathcal{C}_g}{\xi({\bf t})^{\frac{1}{2}(5g-1)}}(\hat{a}_g({\bf t}) + \hat{b}({\bf t})\xi({\bf t})^{1/2}),
    \end{equation}
    where $\hat{a}_g,\hat{b}_g$ are analytic in a neighborhood of ${\bf t} = {\bf t}(\boldsymbol{\alpha})$, $\hat{a}_g({\bf t}(\boldsymbol{\alpha})) = 1$, and the constants $\mathcal{C}_g$ are as in Proposition \ref{recurrence-coefficient-local-prop}.
\end{proposition}
\begin{proof}
    Observe that we have the following explicit formula for $\mathcal{A}_g({\bf t})$:
        \begin{equation*}
            \mathcal{A}_g({\bf t}) = \frac{1}{2\varkappa^2}\sum_{k=0}^{g}r_{g-k}({\bf t})\sum_{\ell=0}^k \frac{r_{k-\ell}^{(2\ell)}({\bf t})}{(2\ell)!},
        \end{equation*}
    where derivatives in the above are with respect to $\varkappa$. 
    Observe that, by Lemma \ref{general-derivative-lemma-precise}, the most singular term (the one with the highest $\psi$-weight) in the $\mathcal{A}_g({\bf t})$ is
        \begin{equation*}
            \frac{1}{2\varkappa^2}\sum_{k=0}^{g}r_{g-k}({\bf t};\varkappa)\sum_{\ell=0}^k \frac{r_{k-\ell}^{(2\ell)}({\bf t};\varkappa)}{(2\ell)!} = \frac{1}{\varkappa^2}r_0r_g + \textit{[less singular terms]},
        \end{equation*}
    and so we can write that
        \begin{equation*}
            \mathcal{A}_{g}({\bf t}) := \frac{1}{2\varkappa^2}\sum_{k=0}^{g}r_{g-k}({\bf t};\varkappa)\sum_{\ell=0}^k \frac{r_{k-\ell}^{(2\ell)}({\bf t};\varkappa)}{(2\ell)!} = \frac{\mathfrak{e}}{(\mathfrak{e}-1)\varkappa}\frac{\mathcal{C}_g}{\xi({\bf t})^{\frac{1}{2}(5g-1)}}(\hat{a}_g({\bf t}) + \hat{b}({\bf t})\xi({\bf t})^{1/2}).
        \end{equation*}
\end{proof}

Finally, we can prove Theorem \ref{main-theorem}:
    \begin{proof}
        \textit{(of Theorem \ref{main-theorem})}. Since
            \begin{equation*}
                [{\bf t}^{\boldsymbol{\alpha}V}] F_g({\bf t}) = \frac{[{\bf t}^{\boldsymbol{\alpha}V}]\mathcal{A}_g({\bf t})\big|_{\varkappa=1}}{\mathfrak{e}^2 V^2}(1+\mc O(V^{-1}))
            \end{equation*}
        as $V\to \infty$ by Proposition \ref{prop:AtoF}, we have only to evaluate the asymptotics of $[{\bf t}^{\boldsymbol{\alpha}V}]\mathcal{A}_g({\bf t})$. The result follows immediately from applying the result of Proposition \ref{prop:asymptotic-formula} to $\mathcal{A}_g({\bf t})$ (in light of Proposition \ref{prop:A-structure}), and evaluating at $\varkappa = 1$. Indeed, we have the explicit relation
            \begin{equation}
                \mathcal{N}_g(\boldsymbol{\alpha}V) = V!\dfrac{I_V\left(\frac{\mathfrak{e} g({\bf t}(\boldsymbol{\alpha}))^{5g-1}}{\varkappa(\mathfrak{e}-1)} \mathcal{C}_g;\boldsymbol{\alpha},5g-1\right)\big|_{\varkappa = 1}}{\mathfrak{e}^2V^2}(1+\mc O(V^{-1/2})).
            \end{equation}
        Observe that $\mathcal{K}_g$ and $\mathcal{C}_g$ are related by the formula
            \begin{equation}
                \mathcal{K}_g := \frac{\mathcal{C}_g\big|_{\varkappa=1}}{\mathfrak{e}(\mathfrak{e}-1)} \sqrt{\frac{\mathfrak{z}-\mathfrak{e}}{\mathfrak{e} (\mathfrak{e}-1)}} (\mathfrak{e}||\boldsymbol{\nu}(\boldsymbol{\alpha})||)^{\frac{1}{2}(5g-1)}.
            \end{equation}
        It can be readily checked that $\mathcal{K}_g$ satisfy the recurrence relation in the statement of the theorem, since $\mathcal{C}_g$ satisfy the recurrence \eqref{eq:C-recursion}. This completes the proof.
    \end{proof}

\appendix

\section{\texorpdfstring{Recovering \eqref{eq:g-1-pure}}{Reduction of Main Theorem to Regular Case}}
\label{appendix:reduction}

In this section, we work out the reduction of Equation \eqref{eq:n-1-general-even} to the regular case where for some $p \geq 2$, $n_{2p} \neq 0$ and $n_{2k} = 0$ for all $1 \leq k \leq p - 1$. In this setting, \eqref{eq:n-1-general-even} reads 
\begin{equation}
    \mathcal{N}_1(0,0,...,0, n_{2p}) = \dfrac{1}{12} n_{2p}! \binom{2p}{p}^{n_{2p}} p^{n_{2p}} \sum_{r = 1}^n \dfrac1r \sum_{\substack{k_1, k_2, ..., k_r\\k_1 + \cdots + k_r = n_{2p}}} \prod_{i = 1}^r p(p - 1) k_i \dfrac{(pk_i - 1)!}{((p - 1)k_i + 1)!} \dfrac{1}{k_i!}.
\end{equation}
Denoting $c_{2p} := p \binom{2p}{p} = 2p \binom{2p - 1}{p - 1}$, we find 
\begin{equation}
    \mathcal{N}_1(0,0,...,0, n_{2p}) = \dfrac{n_{2p}!}{12} c_{p}^{n_{2p}} \sum_{r = 1}^n \dfrac{(p - 1)^r}{r} \sum_{\substack{k_1, k_2, ..., k_r\\k_1 + \cdots + k_r = n_{2p}}} \prod_{i = 1}^r \dfrac{(pk_i)!}{((p - 1)k_i + 1)!} \dfrac{1}{k_i!}.
\end{equation}
The numbers 
\begin{equation}
    A_m(p, r) := \dfrac{r}{pm + r} \binom{pm + r}{m} = r\dfrac{(mp + r - 1)!}{(m(p-1) + r)! m!}
\end{equation}
are known as the Fuss-Catalan numbers (cf. \cite[Chapter 7.5, Example 5]{MR1397498}). In this notation, 
\begin{equation}
    \mathcal{N}_1(0,0,...,0, n_{2p}) = \dfrac{n_{2p}!}{12} c_{p}^{n_{2p}} \sum_{r = 1}^n \dfrac{(p - 1)^r}{r} \sum_{\substack{k_1, k_2, ..., k_r\\k_1 + \cdots + k_r = n_{2p}}} \prod_{i = 1}^rA_{k_i}(p, 1).
\end{equation}
The Fuss-Catalan numbers enjoy many algebraic and analytic properties; for one, their generating function $B_{p, r} (z) := \sum_{m \geq 0} A_m(p, r) z^m$ satisfies 
\[
    B_{p, r} (z) = \left[1 + z (B_{p, r}(z))^{\frac{p}{r}} \right]^r,
\]
which in our case reduces to 
\begin{equation}
    B_{p, 1} (z) =1 + z (B_{p, 1}(z))^p.
    \label{eq:gen-fun-fc-relation}
\end{equation}
Letting $\mathfrak{B}_p(z) = B_{p, 1}(z) - 1$ and using the usual Taylor expansion of the logarithm, we observe that 
\begin{equation}
    -\log(1 - (p - 1)\mathfrak{B}_p(z)) = \sum_{r = 1}^\infty \dfrac{(p - 1)^r}{r} \left( \sum_{m \geq 1} A_m(p, 1) \right)^r.
\end{equation}
Using this and the Cauchy product formula, we find 
\begin{equation}
    [z^n] (-\log(1 - (p - 1)\mathfrak{B}_p(z))) = \sum_{r = 1}^n \dfrac{(p - 1)^r}{r} \sum_{\substack{k_1, k_2, ..., k_r\\k_1 + \cdots + k_r = n}} \prod_{i = 1}^rA_{k_i}(p, 1)
\end{equation}
or 
\begin{equation}
    \mathcal{N}_1(0,0,...,0, n_{2p}) = \dfrac{n_{2p}!}{12} c_{p}^{n_{2p}} [z^{n_{2p}}] (-\log(1 - (p - 1)\mathfrak{B}_p(z))).
\end{equation}
It follows from \eqref{eq:gen-fun-fc-relation} that 
\begin{equation}
    \mathfrak{B}_p(z) = z \Psi(\mathfrak{B}_p(z)) \quad \text{ where } \quad \Psi(\tau) := (\tau + 1)^p
\end{equation}
We now apply Lagrange inversion with
\begin{equation}
    \phi(\tau) := -\log(1 - (p - 1)\tau)
\end{equation}
to compute: 
\begin{equation}
    \begin{aligned}
        \relax [z^n](-\log(1 - (p - 1)\mathfrak{B}_p(z))) &= \dfrac{1}{n} [\tau^{n-1}] \phi'(\tau) \Psi^n(\tau)\\
        &= \dfrac{1}{n} [\tau^{n-1}] \dfrac{p - 1}{1 - (p - 1)\tau} (\tau + 1)^{n p}\\
        &= \dfrac{p - 1}{n} [\tau^{n - 1}] \left( \sum_{k = 0}^\infty (p - 1)^k \tau^k\right) \left( \sum_{j = 0}^{np} \binom{np}{j} \tau^j \right) \\
        & = \dfrac{1}{n} \sum_{k = 0}^{n - 1} \binom{np }{n - 1 - k}(p -1)^{k + 1}
    \end{aligned}
\end{equation}
Altogether, we find 
\begin{equation}
    \mathcal{N}_1(0, ..., n_{2p}) = \dfrac{(n_{2p}-1)!}{12} c_p^{n_{2p}} \sum_{k = 0}^{n_{2p} - 1} \binom{n_{2p} p }{n_{2p} - 1 - k}(p -1)^{k + 1}.
\end{equation}
This number is known, so let's match it with existing formulas; It was shown in \cite{ELT} that 
\begin{multline}
    \mathcal{N}_1(0, ..., n_{2p}) = \dfrac{n_{2p}! }{12} c_p^{n_{2p}} \left( (p - 1) \binom{pn_{2p} - 1}{n_{2p} - 1} \pFq{3}{2}{1, 1, 1-n_{2p}}{2, (p - 1)n_{2p} + 1}{1 - p} \right.\\
    - \left. (p - 1)^2 \binom{pn_{2p} - 1}{n_{2p} - 2} \pFq{3}{2}{1, 1, 2-n_{2p}}{2, (p - 1)n_{2p} + 2}{1 - p} \right)
\end{multline}
To prove the equality of the two expressions, we first use the Pochhammer symbol identities 
\[
     (m)_k = \dfrac{(m + k- 1)!}{(m - 1)!}\quad \text{ and } \quad (-x)_k = (-1)^k (x - k + 1)_k
\]
to write 
\begin{multline*}
   S_1:= (p - 1) \binom{pn_{2p} - 1}{n_{2p} - 1} \pFq{3}{2}{1, 1, 1-n_{2p}}{2, (p - 1)n_{2p} + 1}{1 - p} \\= \sum_{k = 0}^{n_{2p} - 1} \dfrac{(n_{2p} p - 1)!}{(n_{2p} - 1 - k)! (n_{2p}(p - 1) + k)!} \dfrac{(p - 1)^{k + 1}}{k+1}
\end{multline*} 
\begin{multline*}
    S_2:= (p - 1)^2 \binom{pn_{2p} - 1}{n_{2p} - 2} \pFq{3}{2}{1, 1, 2-n_{2p}}{2, (p - 1)n_{2p} + 2}{1 - p} \\= \sum_{k = 0}^{n_{2p} - 2} \dfrac{(n_{2p} p - 1)!}{(n_{2p} - 2 - k)! (n_{2p}(p - 1) + k + 1)!} \dfrac{(p - 1)^{k + 2}}{k+1}.
\end{multline*}
Separating out the $k = n_{2p} - 1$ in $S_1$, we find 
\begin{equation}
\begin{aligned}
    S_1 - S_2 &= \dfrac{1}{n_{2p}}\dfrac{(n_{2p} p - 1)!}{(n_{2p} - 1 - (n_{2p} - 1))! (n_{2p}(p - 1) + (n_{2p} - 1))!} (p - 1)^{n_{2p}} \\
    &+ \sum_{k = 1}^{n_{2p} - 2} \dfrac{(n_{2p} p - 1)!}{(n_{2p} - 2 - k)! (n_{2p}(p - 1) + k)!} \frac{(p - 1)^{k+ 1}}{k + 1} \left[ \dfrac{1}{n - 1 - k} - \dfrac{p - 1}{n(p - 1) + k+ 1} \right]\\
    &= \dfrac{1}{n_{2p}} (p - 1)^{n_{2p}} + p\sum_{k = 1}^{n_{2p} - 2} \dfrac{(n_{2p} p - 1)!}{(n_{2p} - 1 - k)! (n_{2p}(p - 1) + k+1)!} {(p - 1)^{k+ 1}} \\
    &= \dfrac{1}{n_{2p}} (p - 1)^{n_{2p}} + \dfrac{1}{n_{2p}}\sum_{k = 1}^{n_{2p} - 2} \binom{n_{2p} p}{n_{2p} - 1 - k} {(p - 1)^{k+ 1}} \\
    \implies S_1 - S_2 &=  \dfrac{1}{n_{2p}}\sum_{k = 1}^{n_{2p} - 1} \binom{n_{2p} p}{n_{2p} - 1 - k} {(p - 1)^{k+ 1}}.
\end{aligned}
\end{equation}
Thus, all together,
\[
    \mathcal{N}_1(0, ..., n_{2p}) = \dfrac{n_{2p}! }{12} c_p^{n_{2p}} (S_1 - S_2) = \dfrac{(n_{2p}-1)!}{12} c_p^{n_{2p}} \sum_{k = 0}^{n_{2p} - 1} \binom{n_{2p} p }{n_{2p} - 1 - k}(p -1)^{k + 1},
\]
as desired.

\section{Proof of Proposition \ref{prop:JMU-free-energy}}
\label{appendix:JMU-free-energy}

It is known \cite{MR630674} that the differential form $\omega_{n,N}$ defined in \eqref{eq:jmu-def} is closed. In this section, we compute it in terms of the free energy \eqref{eq:free-energy-def}. In other words, this calculation shows that the partition function $Z_{n,N}({\bf t})$ is an isomonodromic $\tau$-function, up to a polynomial in the variables ${\bf t}$. This, too, is a well-known result of \cite{MR1986408,MR2207650} which we reproduce here for the convenience of the reader. Consider a general polynomial potential of even degree
    \begin{equation}
        V(z;{\bf t}) := \sum_{k=1}^{2p}\frac{t_{k}}{k}z^k,
    \end{equation}
where the parameters $t_k \in \R$, and we assume that $t_{2p} >0$. To this potential we can associate a family of orthogonal polynomials as was done in Section \ref{sec:OP} and matrix-valued functions $\mb Y(z; \mb t)$, $\mb \Phi(z; \mb t)$ which satisfy Riemann-Hilbert problems \ref{rhp:y} and \ref{rhp:phi}, respectively. In this section, we will use the notation $\mb Y_{n}(z; \mb t)$, $\mb \Phi_n(z; \mb t)$ to emphasize the dependence on $n$.

\begin{lemma}\label{Lemma-D1}
    For any $k=1,...,2p$, there exists a $2\times 2$ matrix-valued polynomial $\mb B_k(z; \mb t)$ such that
        \begin{equation}
            \dpd{\mb {Y}_n}{t_k}(z; \mb t) = \mb B_k(z; \mb t)\mb {Y}_n(z; \mb t) + \frac{N}{2k} z^k\mb {Y}_n(z; \mb t)\sigma_3.
        \end{equation}
\end{lemma}
\begin{proof}
    Fix $k=1,...,2p$. Since the jumps of $\mb{\Phi}_n(z; \mb t)$ do not depend on ${\bf t}$, it follows that $\mb{\Phi}_n(z; \mb t), \pd{\mb{\Phi}_n}{t_k}(z; \mb t)$ both satisfy the jump condition \eqref{eq:Phi-jump}. Thus, the function $\pd{\mb{\Phi}_n}{t_k}(z; \mb t) \mb{\Phi}_n^{-1}(z; \mb t)$ is entire, and  the asymptotic condition \eqref{eq:phi-expansion} and a standard application of Liouville's Theorem implies $\pd{\mb{\Phi}_n}{t_k}(z; \mb t) \mb{\Phi}_n^{-1}(z; \mb t)$ is a polynomial, denoted $\mb B_k(z; \mb t)$. Rearranging this identity, we have
        \begin{equation*}
            \dpd{\mb{\Phi}_n}{t_k}(z; \mb t) = \mb B_k(z; \mb t)\mb{\Phi}_n(z; \mb t).
        \end{equation*}
    Since $\mb{\Phi}_n(z; \mb t) = \mb{Y}_n(z; \mb t) \ee^{-\frac{N}{2}V(z;{\bf t})\sigma_3}$, direct differentiation yields the result of the lemma.
\end{proof}
    \begin{lemma}(see, e.g., \cite[Section 22.2.1]{MR2191786})\label{Lemma-D2}
    For any $n \in \N$, there exists a constant $c_{n,N} \in \R$ such that
    \begin{equation}
        \mb{Y}_{n+1}(z; \mb t) = \mb A_n(z; \mb t)\mb{Y}_{n}(z; \mb t),
    \end{equation}
    where 
    \begin{equation}\label{raising-operator}
        \mb A_n(z; \mb t) = 
        \begin{bmatrix}
            z + c_{n,N}(\mb t) & \frac{h_{n,N}(\mb t)}{2\pi \ii}\medskip\\
            -\frac{2\pi \ii}{h_{n,N}(\mb t)} & 0
        \end{bmatrix},
    \end{equation}
    with $h_{n,N}(\mb t)$ as in \eqref{eq:ortho}.
    \end{lemma}
\begin{proof}
    Observe that $\mb A_n(z; \mb t) := \mb{Y}_{n+1}(z; \mb t)\mb{Y}_{n}^{-1}(z; \mb t)$ is an entire function satisfying (cf. \eqref{eq:Y-infty-expansion})
    \[
        \mb A_n(z; \mb t) = \mb{Y}_{n+1}(z; \mb t)\mb{Y}_{n}^{-1}(z; \mb t) = \Oo(z) \qasq z \to \infty.
    \]
    Thus, by Liouville's Theorem, $\mb A_n(z; \mb t)$ is a linear polynomial in $z$. Explicit computation using \eqref{eq:y} yields \eqref{raising-operator}.
\end{proof}

\begin{proposition}
    $\omega_{n}$ coincides with the differential of the partition function $Z_{n,N}({\bf t})$:
        \begin{equation}
            \omega_n = - \sum_{k=0}^{n-1}{\bf d}\log h_{k, N} = - {\bf d} \log Z_{n,N}(\bf t).
            \label{eq:omega-Z}
        \end{equation}
\end{proposition}
\begin{proof}
    We reproduce the proof of \cite{MR1986408}, whose basic strategy is to compare $\omega_{n+1}$ and $\omega_n$. Note that the second equality in \eqref{eq:omega-Z} follows from \eqref{eq:partition-hankel-det} and the second equation in \eqref{eq:ortho}. Using Lemma \ref{Lemma-D2}, we have that
        \begin{align*}
        \mb{Y}_{n+1}^{-1} \dod{\mb{Y}_{n+1}}{z} &= \mb{Y}_{n}^{-1} \mb A_n^{-1} \left(\dod{\mb A_{n}}{z}\mb{Y}_{n} + \mb A_{n}\dod{\mb{Y}_{n}}{z}\right)\\
        &=\mb{Y}_{n}^{-1}\mb A_n^{-1}\dod{\mb A_{n}}{z} \mb{Y}_{n} + \mb{Y}_{n}^{-1} \dod{\mb{Y}_{n}}{z},
        \end{align*}
        and so
        \begin{equation}
        \begin{aligned}
            \omega_{n+1, N} - \omega_{n, N} &= \frac{N}{2}\res_{z = \infty} \Tr \left( \mb{Y}_{n+1}^{-1}\dod{\mb{Y}_{n+1}}{z} {\bf d} V(z; \mb t) \sigma_3\right) - \frac{N}{2}\res_{\varkappa = \infty} \Tr \left( \mb{Y}_n^{-1}\dod{\mb{Y}_n}{z} {\bf d} V(z; \mb t) \sigma_3\right)\\
            &=\frac{N}{2}\res_{z = \infty} \Tr\left(\mb{Y}_{n}^{-1} \mb A_n^{-1} \dod{\mb A_{n}}{z}\mb {Y}_{n}{\bf d} V(z; \mb t) \sigma_3\right)\\
            &= \frac{N}{2}\res_{z = \infty}\Tr\left(\mb A_n^{-1}\dod{\mb A_{n}}{z}\mb{Y}_{n}{\bf d} V(z;\mb t) \sigma_3\mb{Y}_{n}^{-1}\right)
            \end{aligned}
            \label{eq:omega-diff}
        \end{equation}
    Multiplying the result of Lemma \ref{Lemma-D1} on the right by $\mb{Y}_{n}^{-1} \dd t_k$, and summing on $k$, we have that
        \begin{multline*}
            \left({\bf d}\mb {Y}_{n}\right) \mb{Y}_{n}^{-1} = \left(\sum_k \mb B_k \dd t_k\right) + \frac{N}{2}\mb{Y}_{n}{\bf d} V(z;\mb t)\sigma_3\mb{Y}_{n}^{-1}\\
            \Longrightarrow \frac{N}{2}\mb{Y}_{n}{\bf d} V(z; \mb t)\sigma_3\mb{Y}_{n}^{-1} = \left({\bf d}\mb{Y}_{n}\right) \mb{Y}_{n}^{-1} - \sum_k \mb B_k \dd t_k. \qquad\qquad (\star)
        \end{multline*}
        Using identity $(\star)$, we can continue the computation in \eqref{eq:omega-diff}:
        \begin{align*}
            \omega_{n+1} - \omega_{n} &= \res_{z = \infty}\Tr\left(\mb A_n^{-1} \dod{\mb A_n}{z}\left({\bf d}\mb{Y}_{n}\right) \mb{Y}_{n}^{-1}\right) - \res_{z = \infty}\Tr\left(\mb A_n^{-1}\dod{\mb A_n}{z}\sum_k \mb B_k \dd t_k\right)\\
            &= \res_{z = \infty}\Tr\left(\mb A_n^{-1} \dod{\mb A_n}{z}\left({\bf d}\mb {Y}_{n}\right) \mb {Y}_{n}^{-1}\right),
        \end{align*}
        since the last term is a polynomial and is thus residueless. Using
            \begin{equation*}
                \mb A_n^{-1} \dod{\mb A_n}{z} = \frac{2\pi \ii}{h_{n, N}}
                \begin{bmatrix}
                    0 & 0\\
                    1 & 0
                \end{bmatrix},
            \end{equation*}
        and (see, e.g., \cite[Eq. (3.21)]{MR1711036}) 
            \begin{equation*}
                \left({\bf d}\mb{Y}_{n}\right) \mb{Y}_{n}^{-1} = 
                \begin{bmatrix}
                    * & -\frac{1}{2\pi \ii}{\bf d}h_{n, N}\\
                    * & *
                \end{bmatrix}\dfrac{1}{z} + \Oo(z^{-2}),
            \end{equation*}
        where $*$ denotes terms irrelevant to the calculation, we obtain that
            \begin{align*}
                \omega_{n+1} - \omega_{n} = -\frac{1}{h_n} {\bf d}h_{n,N} = -{\bf d} \log h_{n,N}.
            \end{align*}
        Summing the above formula, we find 
            \begin{align*}
                \omega_n = \omega_0 - \sum_{k=0}^{n-1}{\bf d}\log h_{k, N}.
            \end{align*}
        Since
            \begin{equation*}
                {\bf Y}_0(z; \mb t) = \begin{bmatrix}
                    1 & \frac{1}{2\pi \ii}\int_{\R} \frac{\ee^{-NV(x;{\bf t})}}{x-z} \dd x\medskip \\
                    0 & 1
                \end{bmatrix},
            \end{equation*}
        one can directly calculate from the definition that $\omega_0 \equiv 0$. This concludes the proof.
\end{proof}

\section{Proof of \texorpdfstring{Proposition \ref{prop:R-expansion-explicit}}{Proposition computing R}}
\label{sec:prop-R-pf}
An expression for $\mb R_1(z; t)$ appeared in the appendix of \cite{MR1953782} for all $z \in \C \setminus (\partial B(2\sigma^\frac12(\mb t), \delta) \cup \partial B(- 2\sigma^\frac12(\mb t), \delta))$, but we begin by reproducing a proof of the relevant formula here for the convenience of the reader. It was shown in \cite[Section 3.4]{MR1953782} that $\mb R(z; \mb t)$ satisfies a Riemann-Hilbert problem (cf. \cite[Problem 3.2]{MR1953782}) whose jump matrix $\mb J_{\mb R}(z; \mb t)$ is meromorphic in \sloppy $B(\pm 2\sigma^\frac12(\mb t), \delta)$ for a $\delta > 0$ small enough. On $\partial B(\pm 2\sigma^\frac12(\mb t), \delta)$, the jump matrix $\mb J_{\mb R}(z; \mb t)$ admits a large $N$ asymptotic expansion
\[
    \mb J_{\mb R}(z; \mb t) \sim \I + \sum_{k = 1}^\infty \dfrac{1}{N^k} \mb J^{(\pm)}_{k}(z;\mb t), \quad z \in \partial B \left(\pm 2\sigma^\frac12(\mb t), \delta\right)
\]
where matrices $\mb J^{(\pm)}_{k}(z; \mb t)$ can be explicitly computed and are, in fact, meromorphic in $B(\pm 2\sigma^\frac12(\mb t), \delta)$. This expansion implies that matrices $\mb R_k(z; \mb t)$ appearing in \eqref{eq:R-expansion} satisfy additive Riemann-Hilbert Problems; when $k = 1$ the problem reads: 
\begin{rhp}\label{rhp:R-1}
    Seek a $2 \times 2$ matrix function $\mb R_1(z; \mb t)$ satisfying 
    \begin{itemize}
        \item[] {\bf Analyticity:} $\mb R_1(z; \mb t)$ is analytic in $\C \setminus (\partial B( 2\sigma^\frac12(\mb t), \delta) \cup \partial B(-2\sigma^\frac12(\mb t), \delta))$,
        \item[] {\bf Jump condition:} $\mb R_1(z; \mb t)$ has continuous boundary values on $B(\pm2\sigma^\frac12(\mb t), \delta)$ satisfying\footnote{Note that the subscript $\pm$ indicates boundary values taken from the left/right side of the contour, whereas the superscript $(\pm)$ indicates which of the contours the jump is on.}
        \[
            \mb R_{1, +}(z; \mb t) - \mb R_{1, -}(z; \mb t) = \mb J^{(\pm)}_{1}(z; \mb t), \quad z \in \partial B\left( \pm 2\sigma^\frac12(\mb t), \delta\right)
        \]
        \item[] {\bf Normalization:} $\mb R_{1}(z; \mb t) = \Oo(z^{-1})$ as $z \to \infty$.
    \end{itemize}
\end{rhp}
The Riemann-Hilbert Problem \ref{rhp:R-1} is solved by 
\begin{equation}
    \mb R_1(z; \mb t) = -\dfrac{1}{2\pi \ii} \int_{\partial B\left( 2\sigma^\frac12(\mb t), \delta\right)} \dfrac{\mb J^{(+)}_{1}(s; \mb t)}{s - z} \dd s  -\dfrac{1}{2\pi \ii} \int_{\partial B\left( - 2\sigma^\frac12(\mb t), \delta\right)} \dfrac{\mb J^{(-)}_{1}(s; \mb t)}{s - z} \dd s
\end{equation}
where the contour of integration is positively oriented. The latter can be computed using the residue theorem: for $z \in \C \setminus \left(B(  2\sigma^\frac12(\mb t), \delta) \cup B( - 2\sigma^\frac12(\mb t), \delta)\right)$, 
\begin{equation}
    \mb R_1(z; \mb t) = -\res_{s = 2\sigma^\frac12(\mb t)} \dfrac{\mb J^{(+)}_{1}(s; \mb t)}{s - z} - \res_{s = -2\sigma^\frac12(\mb t)} \dfrac{\mb J^{(-)}_{1}(s; \mb t)}{s - z}. 
    \label{eq:R-1-residue}
\end{equation}
To compute the right hand side of \eqref{eq:R-1-residue}, we use the explicit formulas (cf. \cite[Eqs. (3.45) and (3.51)]{MR1953782})
\begin{multline}
    \mb J^{(+)}_{1}(z; \mb t) =  \dfrac{5(z + 2\sigma^\frac12(\mb t))^\frac12}{72 (z - 2\sigma^\frac12(\mb t))^\frac12 \displaystyle\int_{ 2\sigma^\frac12(\mb t)}^z w(s; \mb t)h(s; \mb t) \dd s}  \begin{bmatrix}
        -1 & \ii \\ \ii & 1
    \end{bmatrix} \\
    +\dfrac{7(z - 2\sigma^\frac12(\mb t))^\frac12}{72 (z +2\sigma^\frac12(\mb t))^\frac12 \displaystyle \int_{2\sigma^\frac12(\mb t)}^z w(s; \mb t)h(s; \mb t) \dd s}  \begin{bmatrix}
        1 & \ii \\ \ii & -1
    \end{bmatrix}
    \label{eq:J+}
\end{multline}
\begin{multline}
    \mb J^{(-)}_{1}(z; \mb t) =  \dfrac{5(z - 2\sigma^\frac12(\mb t))^\frac12}{72 (z + 2\sigma^\frac12(\mb t))^\frac12 \displaystyle\int_{ -2\sigma^\frac12(\mb t)}^z w(s; \mb t)h(s; \mb t) \dd s}  \sigma_3\begin{bmatrix}
        -1 & \ii \\ \ii & 1
    \end{bmatrix} \sigma_3\\
    +\dfrac{7(z + 2\sigma^\frac12(\mb t))^\frac12}{72 (z -2\sigma^\frac12(\mb t))^\frac12 \displaystyle \int_{-2\sigma^\frac12(\mb t)}^z w(s; \mb t)h(s; \mb t) \dd s} \sigma_3 \begin{bmatrix}
        1 & \ii \\ \ii & -1
    \end{bmatrix} \sigma_3
    \label{eq:J-}
\end{multline}
$\mb J^{(\pm)}_{1}(z; \mb t)$ are related by the replacement $2\sigma^\frac12(\mb t) \mapsto -2\sigma^\frac12(\mb t)$ and conjugation by $\sigma_3$, and thus it suffices to compute one of the residues in \eqref{eq:R-1-residue}. For $z$ in a small enough neighborhood of $z = 2\sigma^\frac12(\mb t)$, we have 
\begin{multline}
    \int_{2\sigma^\frac12(\mb t)}^z w(s; \mb t) h(s; \mb t) \dd s \\= c_1(\mb t)\left( z - 2\sigma^\frac12(\mb t) \right)^\frac32 
    +c_2(\mb t)\left( z - 2\sigma^\frac12(\mb t) \right)^\frac52 + c_3(\mb t)\left( z - 2\sigma^\frac12(\mb t) \right)^\frac72 +\Oo\left( \left(z - 2\sigma^\frac12(\mb t) \right)^\frac92\right)
    \label{eq:w-int-exp}
\end{multline}
where 
\[
    c_1(\mb t) := \dfrac{4}{3} \sigma^\frac14(\mb t) h \left(2\sigma^\frac12(\mb t); \mb t \right) , \quad c_2(\mb t) :=  \frac25 \left(\frac{1}{4\sigma^\frac14(\mb t)} h \left(2\sigma^\frac12(\mb t); \mb t \right) + 2\sigma^\frac14(\mb t) \dpd{h}{z} \biggl|_{z = 2\sigma^\frac12(\mb t)}\right),
\]
\[
    c_3 (\mb t) := \dfrac{2}{7} \left( -\dfrac{1}{64 \sigma^{\frac34}(\mb t)} h \left(2\sigma^\frac12(\mb t); \mb t \right)+ \frac{1}{4\sigma^\frac14(\mb t)} \dpd{h}{z} \biggl|_{z = 2\sigma^\frac12(\mb t)}  +\sigma^\frac14(\mb t) \dpd[2]{h}{z} \biggl|_{z = 2\sigma^\frac12(\mb t)} \right).
\]
The following calculation can be made easier using this elementary identity:
\begin{lemma}\label{lemma:res-identity}
   Suppose that, in a neighborhood of $s = \alpha$, $f(s), g(s)$ possess Puiseux series 
   \begin{align*}
    f(s) = \sum_{j = -1}^\infty f_j (s - \alpha)^{\frac12 + j}, \qquad g(s) = \sum_{j = 1}^\infty g_j (s - \alpha)^{\frac12 + j},
   \end{align*}
   then the ratio $f(s)/g(s)$ is meromorphic at $s = \alpha$ and has a Laurent expansion of the form 
   \[
        \dfrac{f(s)}{g(s)} = -\dfrac{f_{-1}}{g_1} \dfrac{1}{(s - \alpha)^2} - \left( \dfrac{f_{-1}}{g_1}\frac{1}{(z - \alpha)^2} + \dfrac{g_1 f_0 - g_2 f_{-1}}{g_1^2}\dfrac{1}{ z- \alpha} \right) \dfrac{1}{s - \alpha} + \Oo(1)
   \]
\end{lemma}
Using Lemma \ref{lemma:res-identity}, expansion \eqref{eq:w-int-exp}, and the Taylor series of the remaining analytic factors, we find 
\begin{multline}\label{eq:jp-residue}
    -\res_{s = 2\sigma^\frac12(\mb t)} \dfrac{\mb J^{(+)}_{1}(s; \mb t)}{s - z} = \frac{5}{36}\dfrac{\sigma^\frac14(\mb t)}{c_1(\mb t)} \dfrac{1}{(z - 2\sigma^\frac12(\mb t))^2} \begin{bmatrix}
        -1 & \ii \\ \ii & 1
    \end{bmatrix}\\
    + \left( \frac{5}{72}\dfrac{\frac14 \sigma^{-\frac14}(\mb t) c_1(\mb t)- 2\sigma^\frac14(\mb t) c_2(\mb t)}{c_1^2(\mb t)} \begin{bmatrix}
        -1 & \ii \\ \ii & 1
    \end{bmatrix} + \frac{7}{144} \dfrac{\sigma^{-\frac14}(\mb t)}{c_1(\mb t)} \begin{bmatrix}
        1 & \ii \\ \ii & -1
    \end{bmatrix}\right) \dfrac{1}{z - 2\sigma^\frac12(\mb t)} 
\end{multline}
A tedious but direct calculation shows that the result now follows from the next lemma. 
\begin{lemma}
    Let $h(z; \mb t)$, $P(\sigma; \mb t)$ be as in \eqref{eq:h-formula}, \eqref{eq:P-formula}, respectively. Then, 
    \begin{align}
        h \left(2\sigma^\frac12; \mb t \right) = P_\sigma(\sigma; \mb t), \qquad 2\sigma^\frac12\dpd{h}{z}\biggl|_{z = 2\sigma^{\frac12}} = \frac43 \sigma P_{\sigma \sigma}(\sigma; \mb t),
        \label{eq:h-id}
    \end{align}
    and 
    \begin{equation}
        4\sigma \dpd[2]{h}{z}\biggl|_{z = 2\sigma^{\frac12}} = \dfrac{32}{15} \sigma^2 P_{\sigma \sigma \sigma}(\sigma; \mb t) + \dfrac{4}{3} \sigma P_{\sigma \sigma}(\sigma; \mb t)
        \label{eq:h-id-2}
    \end{equation}
    \label{lemma:h-id}
\end{lemma}
\begin{proof}[Proof of Lemma \ref{lemma:h-id}]
    Starting with the right hand side of the first identity in \eqref{eq:h-id}, evaluating \eqref{eq:h-formula} and re-indexing using $n = k + j + 1$, we have
    \begin{equation}
        h \left(2\sigma^\frac12; \mb t \right) = 1 + \sum_{k = 0}^{p - 1}  \sum_{j = 0}^{p - 1 - k}4^{k}\binom{2j}{j} t_{2(j + k + 1)} \sigma^{k + j} = 1 + \sum_{k = 0}^{p - 1} \sum_{n = k + 1}^{p} 4^k \binom{2(n - k - 1)}{n - k - 1} t_{2n} \sigma^{n - 1}.
    \end{equation}
    Interchanging the order of summation in the final expression yields 
    \begin{equation*}
        h \left(2\sigma^\frac12; \mb t \right) = 1 + \sum_{n = 1}^{p} \sum_{k = 0}^{n - 1} 4^k \binom{2(n - k - 1)}{n - k - 1} t_{2n} \sigma^{n - 1}
    \end{equation*}
    Thus, comparing the above with \eqref{eq:P-formula}, it remains to show 
    \begin{equation}
        \sum_{k = 0}^{n - 1} 4^k \binom{2(n - k - 1)}{n - k - 1} = n \binom{2n - 1}{n},
    \end{equation}
    but this follows from an argument similar to Lemma \ref{lemma:sigma-coef} and the elementary identity 
    \[
       \dfrac{1}{1 - 4x} \cdot \dfrac{1}{\sqrt{1 - 4x}} =  \frac12 \dod{}{x} \dfrac{1}{\sqrt{1 - 4x}}.
    \]
    
    The second identity in \eqref{eq:h-id} is derived similarly; starting from \eqref{eq:h-formula} we have 
    \[
        z \dpd{h}{z} = \sum_{k = 0}^{p - 1} 2k h_{2k} z^{2k} \implies 2\sigma^\frac12\dpd{h}{z}\biggl|_{z = 2\sigma^{\frac12}} = \sum_{k = 0}^{p - 1} \sum_{j = 0}^{p - k - 1} 2k \cdot4^k \binom{2j}{j} t_{2(j + k + 1)} \sigma^{k + j}.
    \]
    Introducing $n = k + j + 1$ and interchanging the sums, we have 
    \[
        2\sigma^\frac12\dpd{h}{z}\biggl|_{z = 2\sigma^{\frac12}} = \sum_{n = 1}^p \sum_{k = 0}^{n - 1} 2k \cdot 4^k \binom{2(n - k - 1)}{n - k - 1} t_{2n} \sigma^{n - 1}
    \]
    Thus, comparing with \eqref{eq:P-formula}, it remains to show that 
    \[
    \sum_{k = 0}^{n - 1} 2k \cdot 4^k \binom{2(n - k - 1)}{n - k - 1} = \frac43 n(n - 1) \binom{2n - 1}{n}
    \]
    which follows, in a similar fashion to Lemma \ref{lemma:sigma-coef}, from the elementary identity
    \[
       \dfrac{8}{(1 - 4x)^2} \cdot \dfrac{1}{\sqrt{1 - 4x}}= \dfrac{4}{3} \cdot  \dfrac12 \dod[2]{}{x} \dfrac{1}{\sqrt{1 - 4x}}. 
    \]
    Finally, to prove \eqref{eq:h-id-2}, we observe that 
    \[
        z^2 \dpd[2]{h}{z} - z \dpd{h}{z} = \sum_{k = 0}^{p - 1} 2k(2k - 2) h_{2k} z^{2k}.
    \]
    Evaluating at $z = 2\sigma^{\frac12}$, re-indexing as before, and using the second identity in \eqref{eq:h-id} we find 
    \[
        4\sigma \dpd[2]{h}{z} \biggl|_{z = 2\sigma^{\frac12}} - \dfrac{4}{3} \sigma P_{\sigma \sigma }(\sigma; \mb t) = \sum_{n = 1}^p \sum_{k = 0}^{n-1} 4^k (2k)(2k - 2) \binom{2(n - k - 1)}{n - k - 1} t_{2n} \sigma^{n-1}.
    \]
    Recalling the definition of $P(\sigma; \mb t)$, it remains to show that 
    \[
        4\sum_{k = 0}^{n-1} 4^k k(k - 1) \binom{2(n - k - 1)}{n - k - 1} =  \dfrac{32}{15} n(n - 1)(n - 2) \binom{2n - 1}{n},
    \]
    which follows from the identity 
    \[
       \dfrac{128}{(1 - 4x)^3} \cdot \dfrac{1}{\sqrt{1 - 4x}} = \dfrac{32}{15} \cdot \frac12 \dod[3]{}{x} \dfrac{1}{\sqrt{1 - 4x}}.
    \]
\end{proof}
    Now that we have $\mb R_1(z; \mb t)$ (away from $z = \pm 2 \sigma^\frac12(\mb t)$), we compute $\mb R_2(z; \mb t)$ by solving another additive Riemann-Hilbert Problem:
    \begin{rhp}\label{rhp:R-2}
    Seek a $2 \times 2$ matrix function $\mb R_2(z; \mb t)$ satisfying 
    \begin{itemize}
        \item[] {\bf Analyticity:} $\mb R_2(z; \mb t)$ is analytic in $\C \setminus (\partial B( 2\sigma^\frac12(\mb t), \delta) \cup \partial B(-2\sigma^\frac12(\mb t), \delta))$,
        \item[] {\bf Jump condition:} $\mb R_2(z; \mb t)$ has continuous boundary values on $B(\pm2\sigma^\frac12(\mb t), \delta)$ satisfying 
        \[
            \mb R_{2, +}(z; \mb t) - \mb R_{2, -}(z; \mb t) = \mb J^{(\pm)}_{2}(z; \mb t) + \mb R_{1, -}(z; \mb t) \mb J^{(\pm)}_{1}(z; \mb t) , \quad z \in \partial B\left(\pm 2\sigma^\frac12(\mb t), \delta\right)
        \]
        \item[] {\bf Normalization:} $\mb R_{2}(z; \mb t) = \Oo(z^{-1})$ as $z \to \infty$.
    \end{itemize}
\end{rhp}
This, too, can be solved explicitly: 
\begin{multline*}
    \mb R_2(z;t) = -\dfrac{1}{2\pi \ii} \int_{\partial B\left( 2\sigma^\frac12(\mb t), \delta\right)} \dfrac{\mb J^{(+)}_{2}(s; \mb t) + \mb R_{1, -}(s; \mb t) \mb J^{(+)}_{1}(s; \mb t)}{s - z} \dd s  \\
    -\dfrac{1}{2\pi \ii} \int_{\partial B\left( - 2\sigma^\frac12(\mb t), \delta\right)} \dfrac{\mb J^{(-)}_{2}(s; \mb t) + \mb R_{1, -}(s; \mb t) \mb J^{(-)}_{1}(s; \mb t)}{s - z} \dd s,
\end{multline*}
where $\mb J^{(\pm)}_{1}(z; \mb t)$ are given in \eqref{eq:J+}, \eqref{eq:J-} and, from \cite[Eqs. (3.46) and (3.52)]{MR1953782},
\begin{align}
    \mb J_2^{(+)}(z; \mb t) = \dfrac{35}{2592} \left( \displaystyle\int_{ 2\sigma^\frac12(\mb t)}^z w(s; \mb t)h(s; \mb t) \dd s \right)^{-2} \begin{bmatrix} -1 & 12\ii \\ -12\ii & -1 \end{bmatrix},\\
    \mb J_2^{(-)}(z; \mb t) = \dfrac{35}{2592} \left( \displaystyle\int_{ -2\sigma^\frac12(\mb t)}^z w(s; \mb t)h(s; \mb t) \dd s \right)^{-2} \sigma_3\begin{bmatrix} -1 & 12\ii \\ -12\ii & -1 \end{bmatrix}\sigma_3.
\end{align}
It follows from the jump condition in Riemann-Hilbert Problem \ref{rhp:R-1} that 
\[
    \mb J^{(+)}_{2}(s; \mb t) + \mb R_{1, -}(s; \mb t) \mb J^{(+)}_{1}(s; \mb t) = \mb J^{(+)}_{2}(s; \mb t) + \mb R_{1, +}(s; \mb t) \mb J^{(+)}_{1}(s; \mb t) - (\mb J^{(+)}_{1}(s; \mb t))^2,
\]
where $\mb R_{1, +}(s; \mb t)$ is given by \eqref{eq:R-1}. Once again, our task is reduced to a residue calculation: for $z \in \C \setminus \left(B(  2\sigma^\frac12(\mb t), \delta) \cup B( - 2\sigma^\frac12(\mb t), \delta)\right)$, 
\begin{multline}
    \mb R_2(z; \mb t) = -\res_{s = 2\sigma^\frac12(\mb t)} \dfrac{ \mb J^{(+)}_{2}(s; \mb t) + \mb R_{1, +}(s; \mb t) \mb J^{(+)}_{1}(s; \mb t) - (\mb J^{(+)}_{1}(s; \mb t))^2}{s - z} \\- \res_{s = -2\sigma^\frac12(\mb t)} \dfrac{ \mb J^{(-)}_{2}(s; \mb t) + \mb R_{1, +}(s; \mb t) \mb J^{(-)}_{1}(s; \mb t) - (\mb J^{(-)}_{1}(s; \mb t))^2}{s - z}. 
    \label{eq:R-2-residue}
\end{multline}
Taking advantage of the same symmetries as above, it suffices to compute the first term in \eqref{eq:R-2-residue}. The result of this calculation is \eqref{eq:R2-z-expansion}.

\bibliographystyle{abbrv}
\bibliography{biblio}
\end{document}